\documentclass{amsart}

\usepackage{amsmath,amsxtra,amssymb,amsfonts,latexsym, mathrsfs, dsfont,bm,mathtools}
\usepackage{latexsym}
\usepackage{bbm}
\usepackage[usenames]{color}
\usepackage[hyperindex]{hyperref}
\usepackage[initials]{amsrefs}
\usepackage{tikz}
\usepackage{comment}
\usepackage{color}
\usepackage{graphics,epsf}         %   OK for MAC
\usepackage{enumerate}

\newtheorem{theorem}{Theorem}[section]
\newtheorem{lemma}[theorem]{Lemma}
\newtheorem{proposition}[theorem]{Proposition}
\newtheorem{remark}{Remark}[section]
\newtheorem{example}{Example}[section]

\newtheorem{definition}{Definition}[section]
\newtheorem{corollary}[theorem]{Corollary}

\newtheorem{claim}{Claim}

\numberwithin{equation}{section}

\renewcommand{\q}{\quad}

\newcommand{\Id}{{\bf 1}}

\def\supp{{\text{\rm supp }}}

\makeatletter
\newcommand*{\rom}[1]{\expandafter\@slowromancap\romannumeral #1@}
\makeatother
\newcommand{\bsb}{\boldsymbol}

\begin{document}

\title[Sharp Estimates]{Sharp estimates for trilinear oscillatory integrals and an algorithm of two-dimensional resolution of singularities}

\author{Lechao Xiao}
\address{ Department of Mathematics, University of Pennsylvania, Philadelphia, PA 19104, USA}
\email{xle@math.upenn.edu}
%    General info
%\subjclass{Primary 54C40, 14E20; Secondary 46E25, 20C20}
\date{November 15, 2013}

%\dedicatory{This paper is dedicated to our authors.}

\begin{abstract}
We obtain sharp estimates for certain trilinear oscillatory integrals, extending Phong and Stein's seminal result to a trilinear setting. This result partially answers a question raised by Christ, Li, Tao and Thiele, concerning sharp estimates for certain multilinear oscillatory integrals. The method in this paper relies on a self-contained algorithm of resolution of singularities in $\mathbb R^2$,  which may be of independent interest. 
\end{abstract}

\maketitle

\section{Introduction}\label{intro}
\vspace{0.2in}

The purpose of this paper is to study sharp decay estimates of the following trilinear oscillatory integrals:
\begin{align}\label{lambda99}
\Lambda_S(f_1,f_2,f_3) = \iint e^{i\lambda S(x,y)} f_1(x)f_2(y)f_3(x+y)  a(x,y)\,dxdy,
\end{align}
where $a(x,y)$ is a smooth cut-off function supported in a sufficiently small neighborhood of $0$, and the phase $S(x,y)$ is a real analytic function in $\mathbb R^2$.  

\subsection{\large Background }\q

Consider the following oscillatory integral operator
\begin{align}\label{h3}
T(f)(x) = \int e^{i\lambda S(x,y)} f(y) a(x,y)\,dy,
\end{align}
where $S(x,y)$ is a smooth real-valued function in $\mathbb R^{n}\times \mathbb R^n$ and $a(x,y)$ is a smooth cut-off function. 

One of the central topics in oscillatory integrals is 
the study of the asymptotic behavior of $\|T\|_{2\to2}$ as $|\lambda|\to\infty$.
Equivalently, from a view of duality, can one find the optimal decay
$C(\lambda)$ s.t.
 \begin{align}
 |\Lambda_2(f,g)|\leq C(\lambda)\|f\|_2\|g\|_2,
 \end{align}
where
the bilinear form $\Lambda_2(f,g)$ is given by
\begin{align}\label{dual00}
\Lambda_2(f,g) =\langle T(f),g\rangle = \iint e^{i\lambda S(x,y)} f(y)\overline{g(x)} a(x,y)\,dxdy?
\end{align}

During the last decades, this problem and related topics have been extensively studied by many authors. 
Fruitful results have been obtained via rich techniques. 
We begin with a classical result from H\"ormander \cite{HOR73}, 
concerning sharp $L^2$ estimates of (\ref{h3}) when $S$ is non-degenerate: 
\begin{theorem}[\textbf{H\"ormander} \cite{HOR73}]\label{h0}
Assume $a(x,y)$ is a smooth cut-off function supported in a neighborhood of $0$ and $S(x,y)$
 is a smooth function such that 
\begin{align}\label{h1}
\left|\det  \frac{\partial^2 S}{\partial x\partial y}\right| \geq 1,  \q\mbox{for all}\q (x,y)\in \supp(a).
\end{align}
Then one has
\begin{align}\label{h2}
\|T(f)\|_2 \leq C |\lambda|^{-n/2}\|f\|_2.\end{align}
\end{theorem}

Establishing sharp estimates in a more general setting, 
in particular when $S(x,y)$ is degenerate, was significantly more difficult. 
Until the early $`$90s, by the seminal works of Phong and Stein \cites{PS94-1, PS94-2, PS97}, 
a full understanding of (\ref{h3}) was obtained when $S$ is a real analytic function of two variables.
 In their works, a systematic treatment was introduced to deal with the degenerate setting. 
The key ingredient to characterize the sharp decay rate is the geometric concept: Newton Polyhedra. 

\begin{definition}
Assume the Taylor series expansion of $S(x,y)$ is given by
\begin{align}\label{Tarlor83}
S(x,y)=\sum_{p,q\in\mathbb N}c_{p,q}x^py^q.
\end{align}
The Newton polyhedron of $S$ is defined as
$$
\mathcal N(S) ={\rm Conv}(\cup_{p,q}\{(u,v)\in\mathbb R^2: u\geq p, v\geq q \,\mbox{and}\,\,c_{p,q}\neq 0\}),
$$
where Conv$(X)$ denotes the convex hull of a set $X\subset \mathbb R^2$. 
 The Newton distance of $\mathcal N(S)$ is defined to be
$$
\delta_S =\inf \{t\in\mathbb R: (t,t)\in\mathcal N(S)\}.
$$
\end{definition}

\begin{theorem}[\textbf{Phong-Stein \cite{PS97}}]\label{psmain}
Let $S(x,y)$ be real-analytic and assume the support of $a(x,y)$ is 
contained in a sufficiently small neighborhood of $0\in\mathbb R^2$, 
then
$$
\|T(f) \|_2 \leq C |\lambda|^{-\frac{1}{2(1+\delta)}}\|f\|_2,
$$
where $\delta$ is the Newton distance of $\mathcal N(\partial_x\partial_y S)$. 
\end{theorem}
It had been shown by the classic work of Varchenko, confirming earlier hypotheses of Arnold, 
the Newton polyhedron is the notion to characterize the decay rate of scalar oscillatory integrals \cite{VAR76}.

\subsection{\large Motivation}\label{MO13} \q

In this paper, we study some trilinear analogues of the above problems. Set
\begin{align}
\Lambda_S(f_1,f_2,f_3) = \iint e^{i\lambda S(x,y)} f_1(x)f_2(y)f_3(x+y)  a(x,y)dxdy. \tag{\ref{lambda99}}
\end{align}
Likewise, we want to characterize the optimal constant $\epsilon$ s.t. the following is true
\begin{align}
|\Lambda_S(f_1,f_2,f_3)| \leq C|\lambda|^{-\epsilon}\|f_1\|_2\|f_2\|_2\|f_3\|_2,
\end{align}
for some constant $C$ independent of $f_1$, $f_2$, $f_3$ and $\lambda$.

The study of the trilinear form (\ref{lambda99}) is also motivated by the work of Christ, Li, Tao and Thiele \cite{CLTT05}, where certain multilinear oscillatory integrals were studied in a very general setting. 

To formulate the questions posed in \cite{CLTT05}, we need some preliminary notations.
Let $\bsb\pi=(\pi_1,\dots,\pi_J)$, where each $\pi_j: \mathbb R^n\to \mathbb R^{n_j}\subset \mathbb R^n$ is a surjective linear projection. Let $S:\mathbb R^n\to \mathbb R$ be a polynomial and $a(x)$ be a smooth cut-off function supported in a small neighborhood of $0\in \mathbb R^n$. For each $j$, let $f_j:\mathbb R^{n_j}\to \mathbb C$ be a measurable function. Consider the following multilinear form:
\begin{align}\label{multi}
\Lambda_{S,\bsb\pi}(f_1,f_2,\dots,f_J) = \int e^{i\lambda S(x)}a(x)\prod_{j=1}^J f_j\circ\pi_j(x)dx.
\end{align}
\textbf{Q1:} For what kind of input $(S,\bsb\pi)$, the following is true
\begin{align}\label{fund}
|\Lambda_{S,\bsb\pi}(f_1,f_2,\dots,f_J)| \leq C|\lambda|^{-\delta} \prod_{j=1}^J \|f_j\|_{p_j}
\end{align}
for some $\delta>0$, some $\bsb p= (p_1,\dots,p_J)\in [1,\infty]^J$ and all $f_j\in L^{p_j}(\mathbb R^{n_j})$? 
\\
\textbf{Q2:} If \textbf{Q1} could be answered affirmatively, what is the optimal exponent $\delta$?

Giving a complete answer to \textbf{Q1} for a most general input $(S,\bsb\pi)$ is quite challenging. Still, an affirmative answer to \textbf{Q1} was given in \cite{CLTT05} under certain dimension assumptions on $\bsb\pi$. For further progress of \textbf{Q1}, we refer the readers to \cites{CHR11-1,CHR11-2,CHS11,GR08}.

For \textbf{Q2}, some results were already known. For instance, when $J=2$, $n_1=n_2=n/2$ (assume $n$ is even) and $S$ is smooth, Theorem \ref{h0} provides a sufficient characterization when the best possible decay can be obtained. Theorem \ref{psmain} settled the case $n=J=2$ and $S$ is an arbitrary analytic function; see \cite{GR05,RY01,SEE99} for $S\in C^\infty(\mathbb R^2)$. For $n=J\geq 2$ and $S$ is a polynomial, almost sharp estimates (probably up to a power of $\log |\lambda|$) were known, by the work of Phong, Stein and Sturm \cite{PSS01}.
 
In Christ, Li, Tao and Thiele's attempt to answer \textbf{Q1}, 
an important step is a reduction to the trilinear setting.
 Thus, it is helpful to fully understand the trilinear case, 
in particular to determine the optimal exponent in (\ref{fund}) in this setting. 
This motivates us to study  
the sharp estimate of the trilinear form (\ref{lambda99}), which corresponds to the case $n=2$, $J=3$, $S$ is analytic and $\bsb\pi =\bsb\pi_0$, where 
\begin{align}\label{sim}
\bsb\pi_0(x,y)=(\pi_{01}(x,y),\pi_{02}(x,y),\pi_{03}(x,y))=(x,y,x+y).
\end{align} 
Indeed, a more general setting
\begin{align}\label{lambda}
\Lambda_{S,\bsb\pi}(f_1,f_2,f_3) = \iint e^{i\lambda S(x,y)} a(x,y)\prod_{j=1}^3f_j\circ \pi_j(x,y) dxdy,
\end{align}
can be reduced to (\ref{lambda99}) via an invertible linear transformation in $\mathbb R^2$ (see Section \ref{gg}),
where 
$\pi_j:\mathbb R^2\to \mathbb R$ are pairwise linearly independent projections for $j=1, 2, 3$.

One necessary condition for (\ref{lambda}) to possess a decay bound is that $S$ should be non-degenerate relative to $\bsb\pi$, in the sense $S(x,y)$ cannot be represented as a sum of functions of $\{\pi_j(x,y)\}$. Otherwise, if 
$$
S(x,y)=\sum_{1\leq j \leq 3}S_j\circ \pi_j(x,y),
$$
then we can incorporate each $e^{i\lambda S_j\circ \pi_j(x,y)}$ into $f_j\circ\pi_j(x,y)$ by setting 
$$
\tilde f_j\circ\pi_j(x,y)=e^{i\lambda S_j\circ \pi_j(x,y)}f_j\circ\pi_j(x,y).
$$
 Since $\|\tilde f_j\|_{p_j} =\|f_j\|_{p_j}$, one cannot expect any decay as in (\ref{fund}).% or (\ref{fund23}). 

Let $\pi_j^\perp:\mathbb R^2\to \mathbb R\subset \mathbb R^2$ be linear projections s.t. $\pi_j\circ\pi_j^\perp=0$ and $\|\pi_j^\perp\|_2=1$. Set $\bsb\pi^\perp=(\pi_1^\perp,\pi_2^\perp,\pi_3^\perp)$ and $D_{\bsb\pi^\perp} =\prod_{j=1}^3 \pi_j^{\perp}\cdot\nabla$. %, where $D_v$ is the differential operator  defined as $D_v=v\cdot \nabla$. '
Then $S$ is called simply degenerate relative to $\bsb\pi$ if $D_{\bsb\pi^\perp} S\equiv 0$ \cite{CLTT05}, 
otherwise $S$ is called simply non-degenerate relative to $\bsb\pi$. 
In addition, $S$ is simply degenerate at a point $(x_0,y_0)$ if $D_{\bsb\pi^\perp} S(x_0,y_0)= 0$. 
Simple non-degeneracy implies non-degeneracy and the converse is not true in general. 
But in our case, they are equivalent; see \textbf{Proposition 3.1} in \cite{CLTT05}.

\subsection{\large Results}\q

The following theorem, extending Theorem \ref{h0} to the trilinear setting when $n=1$, states that if $S$ is simply non-degenerate everywhere in Conv$(\supp (a))$, then one can obtain the optimal bound of (\ref{lambda}). 
\begin{theorem}\label{global70}
Assume $a(x,y)$ is a smooth cut-off function supported in a neighborhood of $0\in\mathbb R^2$ and $S(x,y)$ is smooth s.t.  
\begin{align}\label{condition1}
\left |D_{\bsb\pi^{\perp}}S(x,y)\right|\geq 1 \q \mbox{for all} \q(x,y)\in {\rm Conv}(supp(a)).
\end{align}
Then  
\begin{align}\label{tri70}
|\Lambda_{S,\bsb\pi}(f_1,f_2,f_3) |\leq C |\lambda|^{-1/6} \prod_{j=1}^3\|f_j\|_2.
\end{align}
\end{theorem}

This theorem is indeed implicitly proved by Li \cite{LI08}.  
We also extend Theorem \ref{psmain} to the trilinear form (\ref{lambda}). %Different to Phong and Stein's result, 
Different to what was expected, the characterization of the sharp exponent 
in this case is distinct from the ones in Varchenko's or Phong--Stein's results. 
Instead, it is described algebraically by the relative multiplicity of $S$.  
Write $S$ as a sum of homogeneous polynomials 
$
S(x,y) = \sum_{i} S_i(x,y), %= \sum_{i}  \sum_{p+q=i}c_{p,q}x^py^q. 
$
%where 
%$$
%S_i(x,y) =\sum_{p+q=i}c_{p,q}x^py^q. 
%$$
 where each $S_i$  is a homogeneous polynomial of degree $i$. The multiplicity of $S$, i.e. the order of the zero of the function $S$ at the origin, is 
$$
{\rm mult}(S) = \min\{i:  S_i\neq 0\}.%=  \min\{p+q:  c_{p,q}\neq 0\}. 
$$
 We also adopt the convention that ${\rm mult}(S) =-\infty$ if $S= 0$. 
 The multiplicity of $S$ relative to $\bsb\pi$ is defined as
\begin{align}
{\rm mult}_{\bsb\pi}(S) = \min\{i:D_{\bsb\pi^\perp} S_i \neq 0\} ={\rm mult}(D_{\bsb\pi^{\perp}}S) +3,
\end{align}
which is the multiplicity of the quotient of $S$ by the class of degenerate analytic functions.
Notice that if $S$ is simply degenerate, then ${\rm mult}_{\bsb\pi}(S)=-\infty$. 
One of the two main results of this paper is: 
\begin{theorem}\label{main}
Assume $S(x,y)$ is a real analytic function %which is simply non-degenerate relative to $\bsb\pi$ 
and the support of $a(x,y)$ is sufficiently small.
Then
\begin{align}\label{ex}
|\Lambda_{S,\bsb\pi} (f_1,f_2,f_3)| \leq C |\lambda|^{-\frac{1}{2{\rm mult}_{\bsb\pi}(S)}}\prod_{j=1}^3\|f_j\|_2.
\end{align}
The result (\ref{ex}) is sharp in the sense that if $a(0,0)\neq 0$, then  
\begin{align}\label{79}
|\Lambda_{S,\bsb\pi} (f_1,f_2,f_3)| \geq C' |\lambda|^{-\frac{1}{2{\rm mult}_{\bsb\pi}(S)}}\prod_{j=1}^3\|f_j\|_2,
\end{align}
as $|\lambda| \to \infty$, for some $C'>0$ and some $\{f_j\}_{1\leq j\leq 3}$.
\end{theorem}

Existence of a (non-sharp) decay rate in the bound of (\ref{ex}) was treated in \cite{CLTT05}. 
What is new and interesting here is its sharpness and its connection to the given phase. 
It is intriguing to compare the above statement with Theorem \ref{psmain}:
\begin{enumerate}
\item Where do the exponents in both Theorem come from? 
 \item 
 Why do the exponents look so different?
\item Show that both exponents are sharp. 
\end{enumerate}
We provide our answers here since they are quite accessible. 
We begin with interpreting both exponents geometrically via the Newton Polyhedron of $P$, where  
\begin{align}\label{p00}
P(x,y) =\partial_x\partial_y(\partial_x-\partial_y)S(x,y)
\end{align}
in the trilinear setting and %Notice
\begin{align}\label{b00}
P(x,y) =\partial_x\partial_yS(x,y)
\end{align}
 in the bilinear setting. Index the vertices of $\mathcal N(P)$ from left to right by $V_1=(p_1,q_1)$, $V_2=(p_2,q_2)$, $\dots$, $V_k=(p_k,q_k)$, 
the compact edges by $E_1=V_1V_2$, $E_2=V_2V_3$, $\dots$, $E_{k-1}=V_{k-1}V_{k}$, and the vertical and horizontal edges by $E_0$ and $E_k$. 
The set of faces, denoted by $\mathcal F(P)$, is the union of the set of vertices and the set of edges.  
The set of supporting lines
of $\mathcal N(P)$, denoted by $\mathcal {SL}(P)$ are the lines 
that intersect the boundary of $\mathcal N(P)$ and do not intersect any other points of $\mathcal N(P)$. 
There is a one-to-one correspondence $\mathcal M$ between $\mathcal {SL}(P)$ and $[0,\infty]$,
given by defining
$\mathcal M(L)$ to be the negative reciprocal of the slope of $L$
for each $L\in\mathcal {SL}(P)$.% 
%$m\in[0,\infty]$, there is a unique $\mathcal M^{-1}(m)\in\mathcal {SL}(P)$ with $m$ as the negative reciprocal of its slope. 
 We often use the notation $L_m\in \mathcal {SL}(P)$ to refer the supporting line $\mathcal M(L_m) =m$. 
Since each $E_j$ lies in exactly one supporting line, we also use $E_j$ to denote that line. 
Let $m_j =\mathcal M(E_j)$, then 
$$
0=m_0< m_1<m_2< \cdots < m_{k-1}<m_k=\infty.
$$
We assign two decay constants $d_L$ and $\delta_L$ to each $L\in \mathcal {SL}(P)$ as follows 
\begin{enumerate} [(1)]
\item  $d_L =\min \{d_{L,x}, d_{L,y}\}$, the minimum of the $x$-intercept and the $y$-intercept of $L$;
\item $(\delta_L,\delta_L)$ is the intersection of $L$ and the bisecting line $p=q$. 
\end{enumerate}
We associate each $L_m\in \mathcal {SL}(P)$ in the $(p,q)$-plane with a region $|y|\sim |x|^m$ in the $(x,y)$-plane,
which will be known as the `good' region defined by $L_m$; see Section \ref{res}.
The two decay constants assigned to $L\in \mathcal {SL}(P)$ are then associated to the decay rates of the trilinear and bilinear forms 
when the support of the cut-off function $a(x,y)$ is localized to the associated region of $L$.

Next we assign the corresponding decay constants to each face $F$ as follows 
\begin{align}
d_F=\sup d_L \q\textrm{and}\q \delta_F =\sup \delta_L,
\end{align}
where both supremums range over all $L\in\mathcal{SL}(P)$ containing $F$.  
%\begin{enumerate}[	(1.)]
%\item For a vertex $V$, let $d_V =\sup_L d_L$ and $\delta_V =\sup_L \delta_L$, where both supremums range over all supporting lines containing $V$;
%\item For a edge $E$, its decay constants $d_E$ and $\delta_E$ are just the same
%as those viewing $E\in\mathcal {SL}(P)$. 
%\end{enumerate}
Then
\begin{align*}
{\rm mult}(P)&=\max\{d_F: \,\, F\in\mathcal F(P) \} =\sup\{d_L: L\in\mathcal {SL}(P)\}
\end{align*}
 and
\begin{align*}
\delta_P &=\max\{\delta_F: \,\,\, F\in\mathcal F(P) \} =\sup\{\delta_L: L\in\mathcal {SL}(P)\}.
\end{align*}

For simplicity, we restrict our attention to the non-vertical edges. 
Let $E=E_j$, $m=m_j$ and $(p,q)=(p_j,q_j)$ for some $j$.
In the trilinear setting, there are subtle differences between $m\leq 1$ and $m\geq 1$. 
In the case $m\leq 1$, one has 
$ 
d_{E}=d_{E,x} =p+mq. 
$
%where $d_{E,x}$ is the $x$-intercept of $E$. 
%Here $p_j+m_{j}q_j$ is indeed the $x$-intercept of $E_j$. 
The value $|\lambda|^{-\frac{1}{2(3+d_{E})}}$ 
corresponds to the bound of $\Lambda_S$ 
when restricting $\supp (a)$ to the region associated to $E$. 
If $m\geq 1$,  then
$
d_{E}=d_{E,y} \geq p/m+q.
$
When restricting $\supp (a)$ to the regions assoicated to $E$,
 the corresponding bound of $\Lambda_S$ is $|\lambda|^{-\frac{1}{2(3+d_{E})}}$; see Proposition \ref{p1} and Proposition \ref{p3}. 
 In general, the bound of $\Lambda_S$ is $|\lambda|^{-\frac{1}{2(3+d_{F})}}$, when $\supp(a)$ is restricted to the region associated to $F$. 
%Similarly, when restricted to the `good' region defined by the vertex $V$,
% the corresponding bound of $\Lambda$ is $|\lambda|^{-\frac{1}{2(3+d_{V})}}$; see Proposition \ref{p2}.
%Finally, 
%the sharp bound $C|\lambda|^{-\frac{1}{2(3+{\rm mult}(P))}}$ 
%will be obtained via the supporting line whose slope is $-1$, i.e. in the region $|y|\sim |x|$. 
However, such difference does not arise in the bilinear setting 
and we only need one expression for $\delta_{E}$ which is $\delta_E =(p+mq)/(1+m)$. 
%When $\supp(a)$
 %is restricted to the region associated to $E$, the corresponding bound is $|\lambda|^{-\frac{1}{2(1+\delta_{E})}}$.%; See Proposition \ref{py1} and  Proposition \ref{py3}.
 %and i
 In general, the bound for the bilinear form is $|\lambda|^{-\frac{1}{2(1+\delta_{F})}}$,
  when $\supp(a)$ is restricted to the region associated to $F$. 
 %Similarly, when restricted to the `good' region defined by the vertex $V$,
 %the corresponding bound is $|\lambda|^{-\frac{1}{2(1+\delta_{V})}}$.
 %Finally, the sharp estimate is obtained via the so-called main face,
 %the edge that intersects the bisecting line or a vertex that lies on it. 

For the second question, the noticeable difference between the operators is the extra term $f_3(x+y)$ in the trilinear form.
 If $x$ and $y$ vary in intervals of length $\delta_1$ and $\delta_2$ respectively, then $(x+y)$ varies in an interval of length about $\max\{\delta_1, \delta_2\}$.
 If $f_3$ is a characteristic function supported in this interval, 
 then $\|f_3\|_2\sim \max\{\delta_1^{1/2}, \delta_2^{1/2}\}$. 
 Heuristically, this freezes the ratio $\log |\delta_1|/\log |\delta_2|$ to be 1, 
 if one wants to optimize the bound of the trilinear form. 
 %Indeed, our example proving the sharpness of Theorem \ref{main} is constructed in this favor. 
However, without the extra term $f_3(x+y)$, the ratio $\log |\delta_1|/\log |\delta_2|$ is totally free. 
The affect of such difference between the operators is realized 
by the difference between the following two Schur's type lemmas:  
\begin{lemma}\label{schur000}
Assume $a(x,y)$ is a measurable function supported in a strip of $x$-width no more than $\delta_1$ and $y$-width no more than
$\delta_2$. Assume $\|a\|_\infty \leq 1$. Then 
\begin{align}\label{schur001}
\left|\iint f_1(x)f_2(y)  a(x,y) dx dy \right|\leq C(\delta_1 \delta_2)^{1/2}\|f_1\|_2\|f_2\|_2.
\end{align}
\end{lemma}

\begin{lemma}\label{schur}
The assumptions on $a(x,y)$ are the same as in Lemma \ref{schur000}. Then
\begin{align}\label{schur1}
\left|\iint f_1(x)f_2(y) f_3(x+y) a(x,y) dx dy \right|\leq C \min\{\delta_1^{1/2}, \delta_2^{1/2}\}\|f_1\|_2\|f_2\|_2\|f_3\|_2.
\end{align}
\end{lemma}
Lemma \ref{schur000} is employed in \cite{PS97} to control the norm of the operator in Theorem \ref{psmain} 
when the phase fails to provide sufficient decay. Lemma \ref{schur} plays the same role in our proof;
see Section \ref{gg} for its proof.

In Lemma \ref{schur000}, the two parameters $\delta_1$ and $\delta_2$ are symmetric, 
 which illustrates why there is no constraint on the ratio $\log |\delta_1|/\log |\delta_2|$. 
% Geometrically, we have certain freedom to choose the slopes of the supporting lines of the Newton polyhedron.
 % A supporting lineis a line that intersects and only intersects the boundary of the Newton polyhedron. 
 Indeed, the sharpness of Theorem \ref{psmain} can be verified by setting $f_1=\Id_{[0, \delta_1]}$ and $f_2=\Id_{[0, \delta_1^m]}$
 with an appropriate choice of $\delta_1$ and with $m$ being the negative reciprocal of the slope of
 any supporting lines
 containing the so-called main face, i.e. the edge that intersects the bisecting line or the vertex that lies on it.
 
 However, such symmetry breaks down in the trilinear setting, 
 even though the parameters $x$ and $y$ appear symmetrically in $\Lambda_S$.  
One should think of (\ref{schur1}) as two different estimates: the first one with the bound $C\delta_1^{1/2}$ 
and the second one with $C\delta_2^{1/2}$.
Consequently, they will lead to two different bounds for (\ref{lambda99}) accordingly. %in terms of the 2-norms of the functions. 
The first (second) bound comes from employing the first (second) estimate of (\ref{schur1}) and Theorem \ref{local}, with decay exponents 
 represented in terms of $x$-intercepts $d_{E_j,x}$ ($y$-intercepts $d_{E_j,y}$). %in  (\ref{DEJ}).  
%The other employs the second estimate of (\ref{schur1}) and Theorem \ref{local}, with the decay exponents represented
%in terms of $y$-intercepts $d_{E_j,y}$ above. %in (\ref{DEJ}). 
This explains why we need to split the range of $m$ into two cases: $m\leq 1$ and $m\geq 1$. 
However, 
these two bounds coincide when the $x$-intercept and the $y$-intercept are equal, i.e. $m=1$.
%are equal, which happens when the two estimates of (\ref{schur1}) are the same, i.e.  $\delta_1\sim \delta_2$. 
In the picture of the Newton polyhedron, this corresponds
to the supporting line of slope $-1$, which is  
given by the equation $p+q =$mult($P$) with $|y|\sim|x|$ as the associated region. 
 This illustrates how the relative multiplicity comes into play and 
 why the sharpness of $\Lambda_S$ is obtained when $|y|\sim|x|$.

Now we come to the sharpness of the trilinear form. Write $S$ as a sum of homogeneous polynomials: 
$$
S(x,y)=\sum_{n=n_0}^\infty S_n(x,y).
$$
Without loss of generality, we assume $n_0$ is the relative multiplicity of $S$, i.e. 
\begin{align}\label{assume}
\partial_x\partial_y(\partial_x-\partial_y)S_{n_0}\neq 0;
\end{align}
see Section \ref{MO13}. 

Let $f_1= f_2$ be characteristic functions 
of the interval $I_A=[-\lambda^{-1/n_0}/A,\lambda^{-1/n_0}/A]$, and $f_3$ 
be the characteristic function of $[-2\lambda^{-1/n_0}/A,2\lambda^{-1/n_0}/A]$, 
where 
  $A>0$ is a number independent of $\lambda$, such that 
$$
|\lambda S(x,y)|\leq 2^{-100}, \q\q {\rm for \,\, all} \q x,y\in I_A. 
$$
Then
\begin{align}
\left|\iint e^{i\lambda S(x,y)} a(x,y)f_1(x)f_2(y)f_3(x+y)dxdy\right| \sim |I_A|\times |I_A| \sim \lambda^{-\frac{2}{n_0}}. 
\end{align}
Notice
\begin{align}
\|f_1\|_2\sim \|f_2\|_2\sim \|f_3\|_2 \sim \lambda^{-\frac{1}{2n_0}}. 
\end{align}
Hence, if 
\begin{align}
\left|\iint e^{i\lambda S(x,y)} a(x,y)f_1(x)f_2(y)f_3(x+y)dxdy\right| \lesssim C(\lambda) \prod_{j=1}^3\|f_j\|_2
\end{align}
then $ C(\lambda) \gtrsim \lambda^{-\frac{1}{2n_0}}=\lambda^{-\frac{1}{2{\rm mult}_{\bsb\pi_0}(S)}}$ as desired.

\subsection{\large Methods}\q

Like Phong and Stein's proof of Theorem \ref{psmain}, the proof of Theorem \ref{main} requires elaborate analysis. 
There are two main ingredients in their proof:% of Theorem \ref{psmain}: 
 \begin{enumerate}[\q\q(1)]
 \item The operator version of the van der Corput Lemma \cite{PS94-1}; see Theorem \ref{ps} and 
 \item The Weierstrass Preparation Theorem (WPT) and the Puiseux Expansion.  
 \end{enumerate}
In order to extend Phong and Stein's framework
 to the trilinear setting, we first establish the trilinear analogue of (1):  %we shall find out some trilinear substitutions for the above ingredients:  
\begin{enumerate}[\q\q (1')]
\item Theorem \ref{local}: trilinear version of Phong--Stein's operator van der Corput Lemma.
\end{enumerate}
In addition, we develop 

\q\!\!(2') a self-contained algorithm of resolution of singularities in $\mathbb R^2$,
\\ which is our second main result. We also use the WPT in the proof of the algorithm, but its use is not essential. One can use the implicit function theorem instead. 
\begin{theorem}\label{raw}
For each analytic function $P$ defined in a neighborhood of $0\in\mathbb R^2$, there is an open set $U$ containing the origin, 
such that up to a set of measure zero, one can partition $U$ into a finite collection of regions $\{V_k\}_{1\leq k\leq K}$, such that $P$ behaves almost like a monomial in each $V_k$ in the following sense.
%Then there are two positive integers $K$ and $M$ such that $U\setminus \{0\}$ can be partitioned into $K$ regions $\{V_k\}_{1\leq k\leq K}$ 
There is an integer $M\in\mathbb N$, and for each $k$ there is a diffeomorphism 
\begin{align}
\rho^{}_k: V_k&\to \rho_k^{} (V_k)
\\
(x,y)&\mapsto (x_k,y_k)
\end{align}
satisfying the following properties:
\begin{align}\label{ccc1}
P(x,y) =P_k(x_k,y_k) =x_k^{p_k}y_k^{q_k} \cdot Q_k(x_k,y_k) \q\mbox{for all}\q (x,y)\in V_k,
\end{align} 
where
\begin{enumerate}[(i)]
\item $(x_k,y_k) = \rho_k^{}(x,y)$ is given by
%$\rho_k^{-1}(x_k,y_k)$ is given by 
\begin{align}\label{ccc}
\begin{cases}
x=x_k
\\
y=\gamma_k(|x_k|^{\frac{1}{M}})+|x_k|^{\frac{M_k}{M}}y_k,
\end{cases}
\end{align}
where $M_k\in\mathbb N$ and $\gamma_k$ is a polynomial,
 unless $P(x,\gamma_k(|x|^{\frac1M}))=0$, then $\gamma_k$ is a convergent power series.
%for $1\leq k\leq K$. 
\item $P_k= P\circ \rho^{-1}_{k}$ and $(p_k,q_k)$ is a vertex of the Newton polyhedron of $P_k$;
\item The function $Q_k$ is smooth and nonvanishing near $0$ in $\rho_k(V_k)$, i.e.
$$
\lim_{(x_k,y_k)\to (0,0)}Q_k(x_k,y_k)\neq 0 \,\,\mbox{inside}\,\, \rho_k^{} (V_k);
$$
\item $\rho_k(V_k)$ (as well as $V_k$ ) is a curved triangular region, whose upper and lower boundaries 
are given by $y_k= C_k |x_k|^{m_k}$ and $y_k = C'_k |x_k|^{m'_k}$, 
 for some $0\leq m_k, m_k'\leq \infty$ with $m_kM$, $m_k'M \in \mathbb N \cup \{\infty\}$, and for some constants $C_k$, $C_k'$. 
%\item $(p_k,q_k)$ is a vertex of the Newton polyhedron of $P_k$;
\end{enumerate}
Moreover, the constants $m_k$, $m_k'$, $(p_k,q_k)$, $M_k/M$ and the function $\gamma_k$\footnote{In the case $\gamma_k$ is an infinite series, 
we can compute any partial sum of $\gamma_k$.} can be computed explicitly via the Newton polyhedra of $\{P_k\}_{1\leq k\leq K}$. 
\end{theorem}
 \begin{remark} 
See Theorem \ref{rs} in Section \ref{res} for a complete version. %Furthermore, assuming we know all the solutions of $S(x,y)=0$ of the form $y=y({x}^{1/M})$, as well as their orders, then the algorithm can be implemented by a computer. 
 \end{remark}

%In the above theorem,  almost all the important information can be computed
%\footnote{If assume all the solutions of $P(x,y)=0$ of the form $y=y(|x|^{1/M})$ and their multiplicities are known, then one can implement the algorithm in a computer.}
%in an explicit manner; see Section \ref{res}.% in terms of Newton polyhedra. 
 
The idea of employing resolution of singularities to investigate oscillatory integrals appeared in Varchenko's work \cite{VAR76}, where deep results from 
Hironaka \cite{HI64} played a crucial role. 
In \cite{PS97}, to control the lower bound of $|\partial_x\partial_yS|$, one resolves the singularities of $\partial_x\partial_yS$ by means of Puiseux series expansions of roots.
More recently, a direct algorithm of resolution of singularities in $\mathbb R^2$ was introduced by Greenblatt \cite{GR04}, where an elegant proof of Theorem \ref{psmain} was presented based on this algorithm. 

Our proof of Theorem \ref{main} and the algorithm here are both influenced by \cite{GR04}, as well as its predecessor \cite{PS97}. 
Many of the ideas inside the algorithm here are very elementary and have even been known for centuries, 
which may come back to Newton's method, known as the Newton-Puiseux algorithm for solving $P(x,y)=0$ 
by a fractional power series $y=y(x^{\frac1M})$, i.e. the Puiseux series; %via an infinite sequence of changes of variables, 
see \cite{CUT04}. 
The philosophy of the algorithm here is similar to that of the one in \cite{GR04}: 
one wants to decompose some neighborhood of a singular point of a given function 
into a finite collection of `good' regions, 
in each of which the function behaves like a monomial.  
However, the algorithm here is distinct from the one in \cite{GR04} in many aspects.
One noticeable difference is that we do not need to employ the implicit function theorem, 
while it is a key for the termination of the algorithm in \cite{GR04}.  
We refer the reader to Section \ref{model} for the main ideas of the algorithm, comparisons to the one in \cite{GR04} 
and some examples that implement the algorithm.

\subsection*{ Concluding remarks}
It is possible to generalize the simply non-degenerate case, i.e. 
Theorem \ref{global70} to higher dimensions, 
however such generalization for the degenerate case, i.e. 
 Theorem \ref{main} is extremely challenging. 
 First, the equivalence between degeneracy and simple-degeneracy breaks down completely, 
 and indeed the class of non-degenerate functions is significantly 
 larger than that of simply non-degenerate functions. 
 Even for proving the existence of decays,
 the case for higher dimensions is substantially more complicated; see \cite{CHS11}. 
  Second, progress on degenerate oscillatory integrals in higher dimensions are slow and arduous,
  and even the higher dimension analogue of Theorem \ref{psmain} is still far from being fully understood; 
 see \cite{GMT07}. % for the case when the phase is a homogenous polynomial. 
  Perhaps a more approachable problem to consider is the $n$-linear analogue of Theorem \ref{main}. 
However, starting from $n=4$, there won't be any decay if one attempts to
bound the $n$-linear form by the product of functions in $L^2(\mathbb R)$. 
Thus, one needs to consider the functions in general $L^{p_j}(\mathbb R)$ space. 
As a starting point, one should % study optimal decay estimates for
%Another interesting problem but still quite challenging
 %study sharp estimates for
  %the norm of (\ref{lambda99}) as a tri-linear functional of $L^p\times L^q\times L^r$, i.e.
   address \textbf{Q2} above for (\ref{lambda99}) for general $p$, $q$ and $r$.
 Certain symmetry, which holds in the setting of $n$ functions with $n$ variables (see \cite{PSS01}),
  breaks down completely in this trilinear setting. Consequently, 
  finding optimal decays for all triples $(p,q,r)$ is substantially more complicated than that in \cite{PSS01};  see \cite{GX15} for recent progress. 
  Finally, 
  we refer the readers to \cite{X2015} for another interesting application of the resolution algorithm, 
  where we obtain a complete characterization for the $L^p\to L^p$ mapping properties 
  for one-dimensional oscillatory integral operators.

The organization of this paper is as follows. In Section \ref{gg}, 
we reduce the trilinear form $\Lambda_{S,\bsb\pi}(f_1,f_2,f_3)$ 
to the special case $\Lambda_S(f_1,f_2,f_3)$ as in (\ref{lambda99}) 
and then prove Theorem \ref{global70} and  Theorem \ref{local}. 
The latter one is the trilinear version of Phong--Stein's van der Corput Lemma. 
In Section \ref{model}, we outline some main ideas of the algorithm. 
Model examples are also provided to illustrate how the algorithm is implemented. 
Details for the proof of the algorithm will appear in Section \ref{res}.  
The method is purely analytic. 
In Section \ref{pmain}, we apply the algorithm and Theorem \ref{local}
 to prove Theorem \ref{main}.

\subsection*{Acknowledgements} 
The author wishes to thank his advisor Xiaochun Li for suggesting this problem. 
Without his valuable and insightful suggestions and enthusiastic guidance this work can not be done. 
The author would also like to thank Philip T. Gressman, the anonymous 
referees for many constructive comments, and 
Maxim Gilula for his help in mathematical writing.

%
%
%
%
%
%
%
%
%
%
%
%
%
%
%
%
%
%
%
%
%
%
%
%
%
%
%
%
%
%
%
%
%
%
%\include{preliminary}
%
%
%
%
%
%
%
%
%
%
%
%
%
%
%
%
%
%
%
%
%
%
%
%
%
%
%
%
%
%
%
%
%

%\section{Proof of Theorem \ref{global70} and Theorem \ref{local}}\label{gg}
\section{Preliminary}
\label{gg}

In this section, we employ the basic strategy from \cite{LI08} to prove Theorem \ref{global70} and its local analogue, Theorem \ref{local}.  
%following technical theorem which is needed in the proof of Theorem \ref{main}. 
At the end, we will prove Lemma \ref{schur}.

\begin{theorem}\label{local}
Assume  $a(x,y)$ is a smooth function supported in a strip of $x$-width no more than $\delta_1$ and $y$-width no more than
$\delta_2$, satisfying the following derivative conditions
\begin{align}\label{ps1}
|\partial_y a(x,y)| \lesssim \delta_2^{-1} \q\q\mbox{and} \q\q|\partial^2_y a(x,y)| \lesssim \delta_2^{-2}. 
\end{align}
Let $\mu>0$ and $S(x,y)$ be a smooth function s.t. for all $(x,y)\in {\rm Conv}(\supp(a)):$
\begin{align}\label{S1}
|D_{\bsb\pi_0^\perp}S(x,y)| \gtrsim \mu
\q\q\mbox{and}\q\q
 |\partial_{y}^{\beta}D_{\bsb\pi_0^\perp}S(x,y)| \lesssim \frac{\mu}{\delta_2^{\beta}} \q\mbox{for}\q \beta=1,2
\end{align}
then for $\Lambda_S$ defined as in (\ref{lambda99}), one has
$$
|\Lambda_S (f_1,f_2,f_3)| \lesssim |\lambda\mu|^{-\frac{1}{6}}\prod_{j=1}^3\|f_j\|_2.
$$ 
\end{theorem}

The above theorem can be viewed as a trilinear analogue of Phong--Stein's operator version of van der Corput Lemma \cite{PS94-1}:
\begin{theorem}\label{ps}
Assume $a(x,y)$ is a smooth function supported in a strip of $x$-width no more than $\delta_1$ and $y$-width no more than
$\delta_2$, satisfying the same derivative conditions as in (\ref{ps1}). 
%\begin{align}
%|\partial_y a(x,y)| \lesssim \delta_2^{-1} \q\q\mbox{and} \q\q|\partial^2_y a(x,y)| \lesssim \delta_2^{-2}. \tag{\ref{ps1}}
%\end{align}
Suppose $\mu>0$ and $S(x,y)$ is a smooth function in $\mathbb R^2$ s.t. the following holds for all $(x,y)\in \supp(a)$:
\begin{align}\label{ps2}
&|\partial_x\partial_y S(x,y)| \gtrsim \mu\q\q\mbox{and}\q
 &|\partial_x\partial^{\beta+1}_y S(x,y)| \lesssim \frac{\mu}{\delta_2^\beta} \q\mbox{for} \q\beta=1,2
\end{align}
then, for $T$ defined as in (\ref{h3}),
$$
\|T(f)\|_2 \lesssim(\lambda \mu)^{-1/2}\|f\|_2.
$$
\end{theorem}

In both theorems above, we have adopted the notation $X\lesssim Y$ to denote $|X|\leq C Y$ where $C$ can depend on $a$ and $S$, but is independent of $\delta_1$, $\delta_2$, $\mu$ and $\lambda$. 
It's also worth pointing out that theorem \ref{ps} is not exactly the same as the one employed by Phong-Stein in \cite{PS94-1}, 
we have adopted a slightly more general version from Greenblatt \cite{GR04}. 
For the proof of Theorem \ref{ps}, we also refer the readers to \cite{GR04}.

Now we turn to the technical details in proving Theorem \ref{global70} and Theorem \ref{local}. 
First of all,  (\ref{lambda}) can be reduced to (\ref{lambda99}). 
Set 
\begin{align}\label{suppp}
\|\Lambda_{S,\bsb\pi}\| = \sup \{|\Lambda_{S,\bsb\pi}(f_1,f_2,f_3)|, \|f_j\|_2\leq 1\q\mbox{for} \q j=1,2,3\}
\end{align}
and $\|\Lambda_S\|$ is defined similarly. 
We may assume $\pi_1(x,y)=x$, $\pi_2(x,y)=y$ and $\pi_3(x,y)=Ax+By$ where $A\neq 0$ and $B\neq 0$. Change variables $u=Ax$ and $v=By$, then
\begin{align*}
\Lambda_{S,\bsb\pi}(f_1,f_2,f_3) 
&= \iint e^{i\lambda S(x,y)}f_1(x)f_2(y)f_3(Ax+By)dxdy
\\
&=\frac{1}{AB}\iint e^{i\lambda S(u/A,v/B)}f_1(u/A)f_2(v/B)f_3(u+v)dudv
\\
&=\frac{1}{AB}\iint e^{i\lambda S_{A,B}(u,v)}f_{1,A}(u)f_{2,B}(v)f_3(u+v)dudv
\end{align*}
where  $S_{A,B}(u,v)= S(u/A,v/B)$, $f_{1,A}(u)=f_1(u/A)$ and $f_{2,B}(v)=f_2(v/B)$. Notice that 
\begin{align*}
D_{\bsb\pi_0^\perp}S_{A,B}(u,v) = \frac{1}{AB}(({\partial_u}/{A}-{\partial_v}/{B})\partial_u\partial_v S)(u/A,v/B)
\end{align*}
Thus $D_{\bsb\pi_0^\perp}S_{A,B} =C D_{\bsb\pi^\perp} S$ for an appropriate constant $C$. In addition $\|f_{1,A}\|_2 =\sqrt A\|f_{1}\|_2$ and $\|f_{2,B}\|_2 =\sqrt B\|f_2\|_2 $. Finally, notice that the convexity is preserved under linear mappings. 
Therefore, for an appropriate constant $C$, one has
$$
\|\Lambda_{S,\bsb\pi}\| = C\|\Lambda_{S_{A,B}}\|.
$$

Now we turn to the proofs of Theorem \ref{global70} and Theorem \ref{local} and only need to consider $\Lambda_S$. 
For simplicity, we assume $\|f_1\|_2 =\|f_2\|_2=\|f_3\|_2 =1$. Applying change of variables $(u,v)=( x+y,y)$ and duality, one has
\begin{align}
\|\Lambda_S(f_1,f_2,f_3)\| \leq \|B(f_1,f_2)\|_2\|f_3\|_2 = \|B(f_1,f_2)\|_2,
\end{align}
where 
\begin{align}
B(f_1,f_2)(u) = \int e^{i\lambda S(u-v,v)} f_1(u-v) f_2(v) a(u-v, v) dv.
\end{align}
Employing the method of $TT^*$, one obtains
\begin{align}
\|B(f_1,f_2)\|_2^2 = \iiint e^{i(\lambda S(u-v_1,v_1)-\lambda S(u-v_2,v_2))}  f_1(u-v_1) \bar f_1(u-v_2) f_2(v_1) \bar f_2(v_2)
\\
a(u-v_1, v_1) a(u-v_2, v_2) dv_1dv_2du.
\end{align}
Change variables: $x=u-v_1$, $y=v_1$ and $\tau = v_2-v_1$ and set
\begin{align}
&S_\tau(x,y)  = S(x,y) -S(x-\tau,y+\tau)
\\
&F_\tau(x) = f_1(x)\bar f_1(x-\tau)
\\
&G_\tau(y) =f_2(y)\bar  f_2(y+\tau)
\\
&a_\tau(x,y) = a(x,y)a(x-\tau,y+\tau)\,.  \label{support}
\end{align}
This yields
\begin{align}\label{bound}
\|B(f_1,f_2)\|_2 ^2 = \iiint e^{i\lambda S_\tau(x,y)} F_\tau(x) G_\tau(y) a_\tau(x,y) dxdy d\tau\,.
\end{align}
The proofs of Theorem \ref{global70} and Theorem \ref{local} slightly diverge now and are presented in two separated subsections.

\subsection{Proof of Theorem \ref{global70}} Split $\| B(f_1,f_2)\|_2^2$ into $B_1+B_2$ according to the value of $|\tau|$ below
\\
$ \bullet $ Case 1. $ |\tau| \leq |\lambda|^{-1/3}$, 
\\
$\bullet $ Case 2. $ |\tau| \geq |\lambda|^{-1/3}$.
 \\
 In Case 1, we simply move the absolute value into the integrals, which yields 
\begin{align}
B_1 \leq \int_{|\tau|\leq |\lambda|^{-1/3} }\|F_\tau\|_1 \|G_\tau\|_1 d\tau \leq |\lambda|^{-1/3} \|f_1\|_2^2 \|f_2\|_2^2 =|\lambda|^{-1/3}.
\end{align}
In Case 2, in order to employ Theorem \ref{h0} to the inner double-integral, we assume for a moment in the support of $a_\tau$, the following holds for some positive constant $C$:
\begin{align}\label{to}
|\partial_x\partial_yS_\tau(x,y)| \geq C | \tau|.
\end{align}
By Theorem \ref{h0} and the Cauchy-Schwarz inequality, $B_2 $ is dominated by 
\begin{align}
&\int_{|\tau|\geq |\lambda|^{-1/3}} C |\lambda \tau|^{-1/2} \|F_\tau\|_2\|G_\tau\|_2 d\tau 
\\
\leq C&|\lambda|^{-1/3} \int \|F_\tau\|_2 \|G_\tau\|_2 d\tau 
\\
\leq C&|\lambda|^{-1/3} \bigg(\int \|F_\tau\|_2^2d \tau \cdot \int \|G_\tau\|_2^2 d\tau\bigg)^{1/2}  
\\
=C& |\lambda|^{-1/3} \|f\|_2^2 \|g\|_2^2 \label{end0} \leq C|\lambda|^{-1/3}  
\end{align}
Thus 
$$
\|B(f,g)\|_2^2 = B_1+B_2 \leq C|\lambda|^{-1/3}.
$$
It remains to verify (\ref{to}) on the support of (\ref{support}). Set
$$
F(t) = S_{xy} (x-t,y+t),
$$
then 
\begin{align}
|F'(t)| = |(\partial_x-\partial_y)\partial_x\partial_y S(x-t,y+t)|. \label{cc}
\end{align}
By the mean value theorem, there is a $t_0$ between $0$ and $\tau$, s.t.
\begin{align}\label{mean}
|\partial_x\partial_yS_\tau(x,y)|  = |F(0)-F(\tau)| = |\int_0^\tau F'(t)dt| = |\tau| |F'(t_0)|.
\end{align}
Notice that $(x,y)\in \supp (a)$ and $(x-\tau,y+\tau) \in \supp(a)$, then by convexity $(x-t_0,y+t_0) \in {\rm Conv}(\supp (a))$. 
Therefore,  (\ref{condition1}), (\ref{mean}) and (\ref{cc}) yield (\ref{to}).

\subsection{Proof of Theorem \ref{local}}Similarly, we split $\| B(f_1,f_2)\|_2^2$ into $B_1+B_2$ according to the value of $|\tau|$ as below
\\
$ \bullet $ Case 1. $ |\tau| \leq |\lambda\mu|^{-1/3}$, 
\\
$\bullet $ Case 2. $ |\tau| \geq |\lambda\mu|^{-1/3}$.
 \\
In Case 1, we simply move the absolute value into the integrals and thus 
\begin{align}
B_1 \leq \int_{|\tau|\leq |\lambda\mu|^{-1/3} }\|F_\tau\|_1 \|G_\tau\|_1 d\tau \leq |\lambda\mu|^{-1/3} \|f_1\|_2^2 \|f_2\|_2^2 =|\lambda\mu|^{-1/3}.
\end{align}
In Case 2, assume at a moment that (\ref{ps1}) is true for $a_\tau$ and (\ref{ps2}) are true for $S_\tau$ with $\mu$ replaced by $|\lambda \mu|$. Then Theorem \ref{ps} implies
\begin{align}
B_2 \leq C \int_{|\tau|\geq |\lambda \mu|^{-1/3}} |\lambda\mu\tau|^{-1/2} \|F_\tau\|_2\|G_\tau\|_2 d\tau  \leq C|\lambda\mu|^{-1/3}.
\end{align}
It remains to verify the conditions mentioned above. Indeed (\ref{ps1}) follows by $a_\tau(x,y) = a(x,y)a(x-\tau,y+\tau)$. $S_\tau$ satisfies the first part of (\ref{ps2}) with $\mu$ replaced by $|\lambda \mu|$ due to (\ref{S1}), (\ref{cc}), (\ref{mean}) and the convexity assumption in theorem \ref{local}. If we set 
$$
F_1(t) =\partial_x\partial_y^2S(x-t,y+t)
$$
and 
$$
F_2(t) = \partial_x\partial_y^3 S(x-t,y+t)
$$
then the second part of (\ref{ps2}) (with $\mu$ replaced by $|\lambda\mu|$) follows from (\ref{S1}), (\ref{cc}), (\ref{mean}) (with $F$ replaced by $F_1$ and $F_2$) and convexity. 

\subsection{Proof of Lemma \ref{schur}}

\begin{proof}
By the Cauchy-Schwarz Inequality, H\"older's inequality the Fubini Theorem, one has
\begin{align*}
&\left|\iint f_1(x)f_2(y)f_3(x+y)a(x,y)dxdy\right|
\\
=&\left|\int f_1(x) \left( \int f_2(y)f_3(x+y)a(x,y)dy\right)dx\right|
\\
\leq &  \bigg| \int \left( \int f_2(y)f_3(x+y)a(x,y)dy\right)^2dx\bigg|^{1/2} \cdot\|f_1\|_2 
\\
\leq & \left| \int  \left( \int |a(x,y)|^2dy\right)\left(\int \left| f_2(y)f_3(x+y)\right|^2dy\right)dx\right|^{1/2}\cdot\|f_1\|_2 
\\
\leq & C\delta_2^{1/2} \|f\|_1\|f_2\|_2\|f_3\|_2.
\end{align*}
The other bound can be obtained similarly. 
\end{proof}

 The estimates in both Lemma \ref{schur} and Lemma \ref{schur000} 
 are sharp, which can be seen by
taking $f_1 =\Id_{[0,\delta_1]}$, $f_2 =\Id_{[0,\delta_2]}$ and $f_3 =\Id_{[0,\delta_1+\delta_2]}$.

%----------------------------------------------------------------------------------------------------------------------------------------------------------
%
%
%
%
%
%
%
%
%
%
%
%
%
%
%
%
%
%
%
%\include{ideas}
%
%
%
%
%
%
%
%
%
%
%
%
%
%
%
%
%
%
%
%
%
%
%----------------------------------------------------------------------------------------------------------------------------------------------------------

\section{ideas and examples of the algorithm}
\label{model}
To employ Theorem \ref{local} to attack Theorem \ref{main}, 
one needs to decompose $\supp (a)$ into regions such that $P(x,y)$ is well-behaved, 
where $P=\partial_x\partial_y(\partial_x-\partial_y)S$.
 Ideally, one hopes $P(x,y)$ to behave like a monomial with a negligible error term. 
 The algorithm is driven by this idea. In each stage of iteration, 
 `good' regions (with the desired property) are obtained via vertices and edges of $\mathcal N(P)$ 
 when $P(x,y)$ is `nonvanishing', and `bad' regions are obtained when $P(x,y)$ `vanishes' on these edges. 
In each such `good' region, $P(x,y)$ behaves like a monomial and no further treatment is required. 
Each of those `bad' regions is then carried to the next stage of iteration. 
A branch of iterations is created for each `bad' region and one hopes
all the `bad' regions will eventually go away as the iterations go deeper.  
 To see how it works, drop all zero coefficients and write $P(x,y) $'s Taylor expansion as 
 \begin{align*}
 P(x,y) =\sum\limits_{p,q\in\mathbb N}c_{p,q}x^py^q 
 \end{align*}
The Newton diagram is the boundary of $\mathcal N(P)$, which consists of two non-compact edges, 
a finite collection of compact edges $\mathcal E(P)$ (may be empty) and a finite collection of vertices $\mathcal V(P)$. 
%The vertices and the edges are called the faces of the Newton polyhedron. 
%We use $\mathcal F(P)$ to denote all the faces, including non-compact faces. 
The Euler formula gives 
$\#\mathcal  V(P)-\#\mathcal E(P) = 1.$
For each face $F\in \mathcal F(P)$, define 
\begin{align}\label{edgeP}
P_F(x,y) = \sum\limits_{(p,q)\in F} c_{p,q} x^p y^q.
\end{align}
Choose one $L_m\in\mathcal {SL}(P)$ and consider the region $|y|\sim |x|^m$ associated to it. 
Notice that $L_m$ goes through at least one $V\in\mathcal V(P)$, say $V=(p_v,q_v)$.
%there exists at least one vertex $(p_v,q_v)=V\in\mathcal V(P) $, such that there exists a line containing $V$ with slope $-\frac 1 m$ not intersecting the interior of $\mathcal N(P)$. We %call such line a supporting line of $\mathcal N(P)$. 
Let $E_l$ and $E_r$ be the edges left to and right to $V$ with $\mathcal M(E_l)=m_l$ and $\mathcal M(E_r)=m_r$, respectively. 
Then $0\leq m_l\leq m \leq  m_r\leq \infty$. Consider the following three possible cases: 
\\
\textbf{Case (1).} $m_l<m<m_r$, 
\textbf{Case (2).} $m=m_l$ and 
\textbf{Case (3).} $m=m_r$, 
\\
which correspond to: 
\\
(1) the vertex $V$ `dominates'  $P(x,y)$,
 (2) the edge $E_l$ `dominates' $P(x,y)$ and (3) the edge $E_r$ `dominates' $P(x,y)$  respectively. \\
In \textbf{Case (1)}, $p_v+m{q_v} < p+mq$ for any other $(p,q)$ with $c_{p,q}\neq 0$. % with $(p,q)\neq (p_v,q_v)$. 
Then in the region $|y|\sim |x|^m$, 
$$
P(x,y)=P_V(x,y)+O(|x|^{p_v+mq_v+\nu}),\q{\rm for \,\,some} \,\, \nu>0.
$$ 
Given $|x|$ sufficiently small, $P_V(x,y)$ is the dominant term in $P(x,y)$, i.e.  
\begin{align}
P(x,y)\sim P_V(x,y)=c_{p_v,q_v}x^{p_v}y^{q_v}.
\end{align}
\begin{center}
\begin{tikzpicture}[scale=0.8]
\draw [<->,thick] (0,9) node (yaxis) [above] {$y$}
        |- (9,0) node (xaxis) [right] {$x$};
  \draw[help lines] (0,0) grid (8,8);
           %  \draw (0,0) --(9,9);
     \fill[gray!30] (2,8)--(2,4)--(3,2)--(5,1)--(8,1)--(8,8);
       \draw[thick] (5,0)--node[below left] {$l$}(0,5) ;
       \fill (5,1)   circle (2pt) (3,2) circle (2pt) (2,4)  circle (2pt);
        \draw [thick]	(9,1)--
        			(5,1) node[below left] {$V_3$}--
			(3,2)node[below left] {$V_2$}--
			(2,4)node[below left] {$V_1$}--
			(2,9);
	%\draw  (0,8)--(4,0);
	\node at (5,5) {Newton Polyhedron};
	\node at (2.7,4){(2,4)};
	\node at (3.5,2.3) {(3,2)};
	\node at (5,1.5) {(5,1)};
	\node at (11,6) {$P(x,y) = x^5y-x^3y^2+x^2y^4$};
	\node at (10.2,5) {$P_{V_2}(x,y) = -x^3y^2$};
	\node at (11,4) {$|x|^{2}\lesssim |y| \lesssim |x|^{1/2}$};
	\node at (7,9.5) { \emph{Case (1): The vertex} $V_2$\emph{ is dominant, where} $1/2< m < 2$};
	\node at (4.5,-1) {Figure 1.};
 % \draw (0cm,2mm) -- (30:2cm);
\end{tikzpicture}
\end{center}

\begin{center}
\begin{tikzpicture}[scale=0.8]
\draw [<->,thick] (0,9) node (yaxis) [above] {$y$}
        |- (9,0) node (xaxis) [right] {$x$};
  \draw[help lines] (0,0) grid (8,8);
           %  \draw (0,0) --(9,9);
     \fill[gray!30] (2,8)--(2,4)--(3,2)--(5,1)--(8,1)--(8,8);
     %  \draw[thick] (5,0)--node[below left] {$l$}(0,5) ;
       \fill (5,1)   circle (2pt) (3,2) circle (2pt) (2,4)  circle (2pt);
        \draw [thick]	(9,1)--
        			(5,1) node[below left] {$V_3$}--
			(3,2)node[below left] {$V_2$}--
			(2,4)node[below left] {$V_1$}--
			(2,9);
	%\draw  (0,8)--(4,0);
	\node at (5,5) {Newton Polyhedron};
	\node at (2.7,4){(2,4)};
	\node at (3.5,2.3) {(3,2)};
	\node at (5,1.5) {(5,1)};
	\node at (11,6) {$P(x,y) = x^5y-x^3y^2+x^2y^4$};
	\node at (11,5) {$P_{V_1V_2}(x,y) = -x^3y^2+x^2y^4$};
	\node at (12,4) {$|y|\sim |x|^{1/2}$};
	\draw[thick] (0,8)--(4,0);
	\node at (6,9.5) { \emph {Case(2): The edge} $V_1V_2$ \emph{ is dominant, where} $ m =1/ 2$};
 % \draw (0cm,2mm) -- (30:2cm);
\end{tikzpicture}
\end{center}

 \textbf{Case (2)}  and \textbf{Case (3)} are exactly the same, so we only discuss  
 \textbf{Case (2)} $m=m_l$ here. In addition, we focus on the right half plane $x>0$ and assume $y=r x^m$, where $r\in\mathbb R$ is a parameter. 
 Notice that $p_v+mq_v =p+mq$ for all $(p,q)\in E_l$ and $p_v+mq_v <p+mq$ for all $(p,q) \notin E_l$. Thus 
 \begin{align*}\begin{cases}
 P_{E_l}(x,rx^m)= P_{E_l}(1,r) x^{p_v+mq_v}
 \\
 P(x,y)=P_{E_l}(1,r) x^{p_v+mq_v}+O(x^{p_v+mq_v+\nu}),\q{\rm for \,\,some} \,\, \nu>0.
\end{cases} \end{align*}
As long as $r$ is away from the non-zero real roots of $P_{E_l}(1,r)$, the origin and the infinity,  $P_{E_l}(x,y)$
 is the dominant term of $P(x,y)$. We refer this as the case that the edge $E_l$ `dominates' $P(x,y)$. 
 
 What remains are the `bad' regions: when $r$ is in some neighborhood of the non-zero roots of $P_{E_l}(1,r)$,
which demands most of the work. The traditional way of further analyzing these `bad' regions would be to do a coordinate change. 
Choose a non-zero root $r_0$ of $P_{E_l}(1,r)$, set $(x_1, y_1)= (x, y -r_0 x^m)$ and consider the function $P(x,y)$ in the new coordinates, namely
 $P_1(x_1, y_1) = P(x_1, y_1+r_0 x_1^m)$. We can then apply the above arguments to the function $P_1(x_1, y_1)$, which will result in a further partition 
 of a `bad' region into `good' regions where vertices or edges dominate, and `bad' regions. Iterating such procedures would end up with an infinite collection of 
 `good' regions. As a by-product, one can find solutions for $P(x,\gamma(x))=0$, where each $\gamma (x)$ is a Puiseux series.  
 However, for analytic purpose it would be 
 of significant advantage and even crucial for the algorithm to converge. By saying an algorithm converges, 
 we always mean that it terminates after finitely many steps. Indeed, if the algorithm fails to converge,  
 \begin{enumerate}
 \item 
 we may fail to find a neighborhood for the desired partition. 
 Notice that in each stage of iteration, in order for the functions $P(x,y), P_1(x_1,y_1), \dots$ to behave like monomials, 
 we need the support of $|x|$, $|x_1|, \dots$  to be sufficiently small. 
 In the case the algorithm fails to converge, 
 it is unclear whether there is a uniform upper bound for $|x|=|x_1|=\dots$, that works for all stage of iterations. 
 Without such an upper bound, %the algorithm cannot be applied to in Theorem \ref{main} and Theorem \ref{psmain}, 
 one may not partition the support of the cut-off function $a(x,y)$ into `good' regions.  %in Theorem \ref{main} and Theorem \ref{psmain}. 
 However, this issue does not arise if the algorithm converges. 
% We do not need to worry about this problem, if the algorithm iterates only finitely many steps. 
 Another possible way to fix this issue is to investigate on how the upper bound of $|x|$ relies on the function $P(x,y)$. 
 This direction seems quite challenging but is of great interest, for it
 is related to the stability of oscillatory integrals/operators. 

\item
Even assuming we can find a neighborhood to do the partition, 
it is significantly simpler to work on a finite collection of `good' regions rather than an infinite collection. 
For the former, we essentially work in single `good' region,
since the desired estimates are allowed to rely on the cardinality of `good' regions, unless 
one seeks for a uniform bound for a class of phases.    
For the latter, 
one has to keep track of all constants appearing in all stages of iterations, and hope to be able to absolutely sum all the resulting estimates together.
Heuristically, a convergent algorithm allows us to divide a problem concerning 
an arbitrary real analytic function into finitely many subproblems, each of which is concerning a `monomial'. 
% and of same nature, and essentially one just needs to handle one subproblem.
 \end{enumerate}
Unfortunately, only performing the change of coordinates $(x_1, y_1)= (x, y -r_0 x^m)$ is not sufficient
 for the algorithm to converge. For example, if we apply the above algorithm to   
 the function 
$$
P(x,y) = \left(y-\Big(\sum_{j=0}^{\infty} r_jx^{j+1}\Big)\right)^n,
$$
we will end up with repeating the following change of variables  
$$
y_{j+1}= y_j -r_jx,\,\,\, {\rm for}\q j\geq 0,
$$
and the iterations will never stop. 
 
 How can one modify the above ideas to ensure the convergence of the algorithm?
 Greenblatt \cite{GR04} had some very nice observations in achieving this goal.
The key is to complexify the change of coordinates in each stage of iterations by means of the implicit function theorem (IFT). 
Roughly speaking, each such change of coordinates helps to decrease a certain `order' of $P(x,y)$ by at least 1,

To see how this works, choose a root  
$r_0$ of $P_{E_l}(1,r)$ and assume its order is $s_0$. Applying the IFT to $\partial_y^{s_0-1} P(x,y)=0$ yields that there is a unique 
$h(x)$ such that 
\begin{enumerate}
\item $\partial_y^{s_0-1} P(x,h(x))=0$ 
\item $h(x)$ is a real analytic function of $x^{\frac 1 {M}}$ 
whose leading term is given by $r_0x^{m}$, where $M$ is some positive integer.  
 \end{enumerate}
 Instead of doing coordinate change of the form $y_1 =y -r_0x^m$, one must do coordinate change of the form $y_1 = y - h(x)$. 
 Consequently,  each root $r_1$, an analogue of $r_0$ from the next stage of iteration, has order $s_1$ at most $s_0-1$. 
Then the algorithm must converge. 
 As a trade-off, the function $h(x)$ is given implicitly as a Puiseux series, which is 
 not easy to compute in general.  
 For instance, consider the function 
$$ 
P(x,y)=(y-x)^n+x^{n}y^{2n}.
$$
One can see that $\mathcal N(P)$ has only one compact edge $E$ and $P_E(1,r)$ has only one root 
$r_0$ of order $s_0 =n$. To apply the algorithm in \cite{GR04}, one first uses the IFT to solve $\partial_y^{n-1}P(x,y) = 0$, i.e. $y-x+cx^{n}y^{n+1}=0$ where $c=\frac{(2n)!}{n!(n+1)!}$.
The solution is then given by $y=h(x)= x+O(x^{1+\mu})$ for some $\mu>0$. Then change the coordinate $y=y_1+h(x_1)$ and plug it into $P(x,y)$. In general, to do iterations in later stages, one needs to compute the Puiseux expansion of $h(x)$ up to a certain number of terms. One then needs to run the Newton-Puiseux algorithm to do such computation.

Can one retain the simplicity and explicitness (avoid using the IFT) of the
change of variables/coordinates, and also ensure the convergence of the algorithm? 
 The answer is affirmative. To do so, we perform the change of variables in each stage of iterations as follows
 \begin{align}\label{cv}
(x,y)=(x_1,r_0x_1^{m}+y_1x_1^{m}).
\end{align} 
The $x_1^{m}$ term in front of $y_1$ plays the role in rescaling `curved' regions back into `uncurved' regions, 
allowing us to iterate in the same kind of regions.  
 %of the form: $U=(0,\epsilon)\times (-\epsilon, \epsilon)$.
Now we assume the algorithm does not stop, resulting in an infinite chain
\begin{align}\label{chain70}
[U,P]=[U_0,P_0]\to[U_1,P_1]\to[U_2,P_2]\to\dots\to [U_n,P_n]\to\cdots.
\end{align}
Each $U_n =(0,\epsilon)\times (-\epsilon, \epsilon)$ is the domain for $P_n(x_n,y_n)$, 
obtained inductively via change of variables of the form (\ref{cv}), for instance $P_1(x_1,y_1) = P(x_1, r_0x_1^{m}+x_1^{m}y_1)$. 
We then search for certain patterns inside $\mathcal N(P_n)$ as $n\to\infty$. 
Amazingly, the number of compact edges of $\mathcal N(P_n)$ converges to 1.
 More precisely, there exists an $n_0\in\mathbb N$ such that for $n\geq n_0$, 
 $\mathcal N(P_n)$ has only one compact edge $E_n$ and the restriction of $P_n$ to it is $c_n(y_n-r_nx_n^{m_n})^{s_{n_0}}$. 
 With this observation at hand, we can now turn off the iterations. 
  At the $n_0$-th stage of iteration, 
  instead of repeating the change of variables of the form (\ref{cv}), we do all of them together at a time
 $$
y_{n_0} =  y_{n_0+1}x_{n_0}^{m_{n_0}}+ \sum\limits_{k=n_0}^{\infty} r_kx_{n_0}^{m_{n_0}+m_{n_0+1}+\cdots+m_k},
 $$
then the algorithm stops immediately; see Lemma \ref{infty}. 

The benefits of the explicitness and simplicity of the change of variables 
can help keep very good track of variables in different stages. 
For any given $n$, one can write down the relation between $(x_n,y_n)$ and $(x,y)$ explicitly; see (\ref{change08}).  
One is also able to compute the Newton Polyhedron of $P_n$, 
assuming we know how to find roots of one-variable polynomials.
It is also quite convenient to estimate 
%upper bounds of the $p$-coordinate 
 %and $q$-coordinate of the far left vertex of $\mathcal N(P_n)$ inductively.  
 %Such upper bounds turn out to be useful to control the 
 lower bounds of $|P(x,y)|$ in `good' regions 
 from later stages of iterations inductively via these Newton Polyhedra; see (\ref{ind2}) and (\ref{key1}).

 \subsection{Two model examples}
 %In what follows, we provide two examples to illustrate the algorithm. 
 The function in the first example is quite simple, we aim to implement the algorithm in details. 
 In particular, we provide details to compute the size of the original domain, i.e. $\epsilon=\epsilon_0$ below. 
 However, we would recommend the readers to focus on the change of variables 
 rather than the values of $\epsilon_k$ and $\rho_k$ below. 
% The complexity of the second example 
 The second example is slightly more complicated,  with primary focus on 
 the convergence of the algorithm. 
 In these two examples, we only handle the right half plane $x>0$, 
for the other half can be handled similarly. 
% Let $\{\epsilon_k\}_{k\geq 0}$ be a strictly increasing sequence of small positive numbers. 
% Each $\epsilon_k$ which depends on $\epsilon_{k+1}$, will be the `size' of the domain in the $k$-th stage of iteration.  
\begin{example}\label{ex001}
 The first example is 
 $$
 P(x,y) =(y-x)(y-x-x^2) = (y^2-2xy+x^2) -x^2y +x^3,
 $$
 and we shall decompose the domain $U_0=U=(0,\epsilon)\times (-\epsilon, \epsilon)$ into a finite collection of `good' regions, where $\epsilon=\epsilon_0$ is sufficiently small that 
 will be determined at the end of the algorithm. The pair $[U_0,P_0]= [U,P]$ is the input and the algorithm runs as follows. 
 % At least in the first time of reading, 
%we recommend only to focus on the change of variables and consider all  
%the $\epsilon_k$'s and $\rho_k$'s as sufficiently small numbers.  
  \begin{enumerate}
\item [(0)]
Notice $\mathcal N(P)$ has two vertices $V_1=(2,0)$, $V_2=(0,2)$ and one compact edge $E=V_1V_2$, whose slope is $-1$. 
Let $\rho_0>0$ be sufficiently small, but much 
greater than $\epsilon_0$, say $\rho_0 >2^{10}\epsilon_0$. 
The polynomial $P_E(1,y)$ has only one root  $r_0=1$.  %whose order is $s_0=2$.  %Assume also $\rho\geq \epsilon_0$ and $0<x<\epsilon_1$. 
Then 
in the `good' region 
$$
U_{0,g,V_1} = \{(x,y)\in U: (1+\rho_0)x\leq y< \epsilon_0\}
$$
one has
$$
P(x,y)\sim_{\rho_0} y^2 =P_{V_1}(x,y)
$$ 
and in the `good' region 
$$
U_{0,g,V_2} = \{(x,y)\in U: -\epsilon_0<y\leq  (1-\rho_0)x\},
$$
one has
$$
P(x,y)\sim_{\rho_0} x^2 =P_{V_2}(x,y). 
$$ 
The remaining `bad' region is $ y=(1+y_1)x$, where $|y_1|<\rho_0$. 
%The $0$-th stage of iteration ends here. 
%To pass to the $1$-th stage of iteration, 
Set $x_1=x$, $y= x_1+x_1y_1$ and 
$$
P_1(x_1,y_1) = P(x_1, x_1(1+y_1)) = x_1^2 y_1^2 -x_1^3y_1. 
$$
Choose $\epsilon_1$ sufficiently small but much bigger than $\rho_0$, say $\epsilon_1>2^{10}\rho_0$,
and set $U_1=(0, \epsilon_1)\times (-\epsilon_1,\epsilon_1)$. 
\item 
Now we are in the $1$-st stage of iteration, whose input is $[U_1, P_1]$.  
Notice $\mathcal N(P_1)$ has two vertices $V_1'=(2,2)$, $V_2'=(3,1)$ and only one compact edge $E'=V_1'V_2'$. Let $\rho_1>2^{10} \epsilon_1$. 
 Similarly, in the region 
$x_1<\epsilon_1$, $(1+\rho_1) x_1\leq  y_1 <\epsilon_1$, 
$$
P_1(x_1,y_1) \sim_{\rho_1} y_1^2 x_1^2 =P_{1,V_1'}(x_1,y_1)
$$
and in the region $x_1<\epsilon_1$, $-\epsilon_1<y_1\leq  (1-\epsilon_1) x_1$ 
$$
P_1(x_1,y_1) \sim_{\rho_1} -y_1 x_1^3= P_{1,V_2'}(x_1,y_1).
$$
The `bad' region is $y_1 =(1+y_2)x_2$ with $|y_2|<\rho_1$. 
Change variables $x_2=x_1$, $y_1 =(1+y_2)x_2$ and set  
$$
P_2(x_2,y_2)  = P_1(x_2, (1+y_2)x_2) = x_2^3(1+y_2)y_2. %\sim x_2^3y_2,
$$
Choose $\epsilon_2>2^{10}\rho_1$ and set $U_2=(0, \epsilon_2)\times (-\epsilon_2,\epsilon_2)$.
\item
In the $2$-nd stage, the input is $[U_2,P_2]$. 
Now $\mathcal N(P_2)$ has only one vertex $V''= (3,1)$ and $P_2(x_2,y_2)\sim  x_2^3y_2$ given $\epsilon_2$ sufficiently small. 
The algorithm stops here. 
\end{enumerate}
One can now determine the value of $\epsilon_0$, whose choice
relies on
how small one wants the error term to be:
%For any pre given $N>0$, one can always find a $\epsilon(N)>0$ such that
by choosing $\epsilon_0 >0$ sufficiently small, we can let the ratio between the ``dominant term'' and the ``error term'' be sufficiently large. 
%(i.e. $x_2^3y_2^2$) relative to the determinant term ($x_2^3y_2$)
%It is also clear here how to choose the value of $\epsilon_2$, which depends how we want the error term 
In general, we need to compute $\{\epsilon_k\}$ and $\{\rho_k\}$ backward. In this case, one can safely take 
$2^{-10}>\epsilon_2>2^{10}\rho_1>2^{20}\epsilon_1>2^{30}\rho_0>2^{40}\epsilon_0$ and in particular any value of $\epsilon_0$ less than $2^{-50}$.

The corresponding partition of  
$U=(0,\epsilon)\times (-\epsilon, \epsilon)$ will be % into five `good' regions, in each of which $P(x,y)$ behaves like a monomial  
\begin{align*}
\begin{cases}
U_{0,g,V_1} = \{(x,y)\in U: y\geq (1+\rho_0 )x\}, \q &where \q P(x,y) \sim y^2 ,
\\
U_{0,g,V_2} = \{(x,y)\in U: y\leq (1-\rho_0 )x\}, \q &where\q P(x,y) \sim x^2 ,
\\
U_{1,g, V_1'} = \{(x,y)\in U: y_1\geq (1+\rho_1 )x_1\}, \q &where \q P(x,y) \sim x_1^2y_1^2,
\\
U_{1,g, V_2'} =\{(x,y)\in U: y_1\leq (1-\rho_1 )x_1\}, \q &where\q P(x,y) \sim -x_1^3y_1,
\\
U_{2,g} = \{(x,y)\in U: -\rho_1< y_2<\rho_1\}, \q &where\q P(x,y) \sim x_2^3y_2 ,
\end{cases}
\end{align*}
where 
\begin{align*}
\begin{cases}
x=x_1=x_2
\\
y=x+xy_1= x+x^2 +x^2 y_2.
\end{cases}
\end{align*}

To help the readers better understand how the two algorithms work differently, 
we also implement the one in \cite{GR04} for this example.

First, notice $s_0=2$, we then need to 
solve $\partial_y P(x,y) =0$. This gives $y= x+ x^2/2$.  Set $y= x+x^2/ 2+y_1$, then 
$$
\tilde P_1(x,y_1) = P(x, x+x^2 /2+y_1) = (y_1-x^2/ 2)(y_1+x^2/ 2).
$$
The only compact edge of $\mathcal N(\tilde P_1)$ has two roots, $r_1=1/2$ and $r_1' =-1/2$. 
Set $y_1 = x^2 /2 +y_2$ and assume $|y_2|<x^2/2$, then   
$$
\tilde P_2(x,y_2) =\tilde P_1(x,x^2/2+y_2) = y_2 (y_2+x^2)\sim x^2y_2.
$$
Similarly set $y_1 = -x^2 /2 +y_2'$ and assume $|y_2'|<x^2 /2$, one has 
$$
\tilde P_2'(x,y_2') =\tilde P_1(x,-x^2/2+y_2') = y_2' (y_2'-x^2)\sim - x^2y_2'.  
$$
The algorithm stops here. 
\end{example}

\begin{example}
 Since the major focus is the termination of the algorithm, we shall be brief and not worry about technical details.  
Consider
$$
P(x,y)=xy(y^2-x)(y-x^2)(y-\sum_{n=1}^\infty x^n)^2. 
$$
In the $0$-th stage of iteration, the `bad' regions are the ones near the curves $y^2-x=0$, $y-x^2=0$ and $y-x=0$. 
Thus, away from the neighborhoods of the following curves:
$y=-\sqrt x, y=\sqrt x, y=x^2$ and $y=x$, the function $P(x,y)$ behaves like a monomial. 
Each of the `bad' regions near $y=-\sqrt x$,  $y=\sqrt x$  and $y=x^2$ will go away via one step of change of variables 
-- all the resulting new Newton polyhedra have only one vertex; see Example \ref{EX08} in Section \ref{res}.

To address the `bad' region near $y=x$,
we do change of variables $x=x_1$ and $y=x_1+x_1y_1$, which yields 
$$
P_1(x_1,y_1) = P(x_1,x_1+x_1y_1) = Q_1(x_1,y_1) x_1^6\left(y_1-\sum_{n=1}^\infty x_1^n\right)^2
$$
where 
$$
Q_1(x_1,y_1) =  (1+y_1)(1+y_1-x_1)(-1+x_1(1+y_1)^2).
$$
One can ignore the function $Q_1(x_1,y_1)$ since it is non-vanishing near the origin. 
Doing change of variables of the form $y_{k-1}=x_k+x_ky_k$ does not help to resolve the singularities, 
and the Newton polyhedron of $P_{k-1}$ only shifts to the right by 2 units. But we want the Newton polyhedron to end up
 with only one vertex, which corresponds to the case where a single vertex is dominant. 
This indeed can be done as follows. Set 
\begin{align*}
\begin{cases}
x_1 =x_2
\\
y_1 =x_2y_2 + \sum_{n=1}^\infty x_2^n
\end{cases}
\end{align*}
then 
$$
P_2(x_2,y_2) =P_1(x_2,  x_2y_2 + \sum_{n=1}^\infty x_2^n) =Q_2(x_2,y_2) x_2^8y_2^2.
$$
The function $Q_2(x_2,y_2)$ is non-vanishing
and $P_2(x_2,y_2) $ behaves like the monomial $ x_2^8y_2^2$.
 The algorithm stops here.  
\end{example}

%--------------------------------------------------------------------------------------------------------------------------------------------------------------------------------
%
%
%
%
%
%
%
%
%
%
%
%
%
%
%
%
%
%
%
%
%
%
%
%\include{algorithm2}
%
%
%
%
%
%
%
%
%
%
%
%
%
%
%
%
%
%
%
%
%
%
%
%
%
%
%
%
%
%
%--------------------------------------------------------------------------------------------------------------------------------------------------------------------------------

\section{An algorithm for resolution of Singularities in $\mathbb R^2$}
\label{res}

Let $U$ be a sufficiently small neighborhood of $0$ in $\mathbb R^2$. For simplicity, we restrict our discussion to the right half-plane $x>0$, since the left half plane can be reduced to this case through the change of variables $(x,y)\to (- x, y)$ and the $y$-axis can be ignored first. 
For the rest of this section, we take $U=\{(x,y): 0<x<\epsilon, -\epsilon <y<\epsilon\}$ and choose $\epsilon $ to be sufficiently small. 
The exact value of $\epsilon$ depends on later stages of iterations and 
will be clear at the end of the algorithm. 
An inductive resolution procedure will be performed on the pair $[U,P]$, where $P$ is any analytic function defined in an open neighborhood containing $U$. 
Let $M\in\mathbb N$ be a large integer.  
\begin{definition}\label{dd1}
A region $W$ is standard if 
 $$
 W=\{(X,Y): 0<X<\epsilon, \, -\epsilon<Y<\epsilon\}
 $$
 for some $\epsilon>0$, which is referred to as the size of $W$. 
 A standard pair $[W,Q]$ is a standard region $W$ 
 together with a power series $Q(X,Y)$ of $(X^\frac{1}{M},Y)$, 
 which converges absolutely in some neighborhood of $W$. 
%Given a coordinate $(X,Y)$ and a real analytic function $Q(X,Y)$ of $(X^\frac{1}{M},Y)$, if $W=\{(X,Y): 0<X<\epsilon, \, -\epsilon<Y<\epsilon\}$, where $\epsilon>0$ is sufficiently small, then we call $W$ a standard region and $[W,Q]$ is a standard pair under the coordinate $(X,Y)$. In addition, we denote ${\rm diam}(W) = \epsilon$, the `diameter' of $W$. 
\end{definition}
The iterations will always be performed on standard pairs, with the regions and functions of $(x^{1/M},y)$ varying. 
%Moreover, $\epsilon>0$ denotes a sufficiently small number whose value may be varied in different stages of iterations; see Example \ref{ex001}. 
 \subsection{\large Part I: Single step of Partition}    \q 
Set $[U_0,P_0]= [U,P]$, $\epsilon_0 =\epsilon$ and $(x_0,y_0)=(x,y) $ to indicate the procedure is in the $0$-th stage. 
Consider 
\begin{align*}
P_E(x,y) = \sum\limits_{(p,q)\in E} c_{p,q} x^p y^q,\q\mbox{for}\q E\in \mathcal E(P).
\end{align*}

Let $V_{E,l}=(p_{E,l},q_{E,l})$ and $V_{E,r}=(p_{E,r},q_{E,r}) \in \mathcal V(P)$ be the left and right vertices of $E$. Set $m_E= \frac{p_{E,l}-p_{E,r}}{q_{E,r}-q_{E,l}}$, then the slope of $E$ is $-1/m_{E}$. The constant $m_E$ is the most important constant associated to each edge $E$. 
Set $e_E = p_{E,l}+m_Eq_{E,l}$, then $e_E = p'+m_Eq'$ for all $(p',q') \in E$ and there is a $\nu>0$ such that $p''+m_Eq''\geq e_E+\nu$ for all $(p'',q'')\notin E$. 
On the curve $y=rx^{m_E}$, where $r\in \mathbb R\setminus \{0\}$,
$$
P_E(x,y) = x^{e_E}\sum\limits_{(p,q)\in E}c_{p,q}r^q=:x^{e_E}P_E(r)
$$
and 
$$
P(x,y)-P_E(x,y) = O(x^{e_E+\nu})
$$
has a higher degree. Thus, given $|x|$ sufficiently small, $P_E(x,y)$ dominates $P(x,y)$, unless 
$$
P_E(r)=\sum\limits_{(p,q)\in E}c_{p,q}r^q\to 0.
$$ 
Assume $\{r_{E,j}\}_{1\leq j\leq J_E}$, labeled in an increasing order, % and $\{r_{E,j}\}_{-J_E'\leq j\leq -1}$ are
is the set of non-zero roots of $P_E(r)=0$ of orders $\{s_{E,j}\}_{1\leq j\leq J_E}$, respectively.  % and $\{s_{E,j}\}_{-J_E'\leq j\leq -1}$ respectively. 
Then
\begin{align}
J_E\leq \sum\limits_{1\leq j\leq J_E} s_j \leq q_{E,l}-q_{E,r},
\end{align}
since $P_E(r)  = r^{q_{E,r}}\sum\limits_{(p,q)\in E}c_{p,q}r^{q-q_{E,r}}$. For simplicity, we say $r_{E,j}$ is a root of $E$ instead of saying $r_{E,j}$ is a root of $P_E(r)$. 
Assign two constants $c_E$ and $C_E$ to each edge $E$. They are independent from $\epsilon_0$ and satisfy
\begin{align}\label{chosen35}
0<c_{E} <2^{-10} |r_{E,j}| <2^{10}|r_{E,j}|<C_E, \q\mbox{for all}\q 1\leq j\leq J_E. 
\end{align}
The numbers $2^{-10}$ and $2^{10}$ above (same as below)
are just a random choice for small and large numbers.  
Let $I_{j}^{\rho_0}(E) = (r_{E,j}-\rho_0, r_{E,j}+\rho_0)$, where $\rho_0 >0$ is much greater than $\epsilon_0$, 
but much smaller than any $c_E$ (and hence any $r_{E,j}$), and   
\begin{align}\label{rho00}
0<\rho_0 < 2^{-10} \cdot \min_{E} \{c_{E}\}
\end{align}
is to be determined at the end of iterations. 

Set 
\begin{align}\label{domain}
\begin{cases}
&I(E)=[c_E,C_E]\cup [-C_E,-c_E],
\\
\\
&I_b(E) = \cup_{1\leq j \leq J_E} I^{\rho_0}_{j}(E),
\\
\\
&I_g(E) = I(E)\setminus I_{b}(E).
\end{cases}
\end{align}
Here $I_b(E)$ are the neighborhood of the roots $\{r_{E,j}\}$ and $I_g(E)$ are the points away from the non-zero roots, 0 and $\infty$. Then
\begin{align}\label{lb}
|P_E(r)| \gtrsim_{\rho_0} 1 \q\q \mbox{for}\q r\in I_g(E).
\end{align}
Thus $P_E(x,y) $ dominates $P(x,y)$ if $y=rx^{m_E}$ and $r\in I_g(E)$. Let %For $1 \leq j\leq J_E-1$, set
\begin{align}
U_{0,g}(E) = \{(x,y)\in U_0: y= rx^{m_E},\, r\in I_g(E)\}
\end{align}
be the `good' regions generated by the edge $E$, which is a disjoint union of $(J_E+2)$ `good' regions: $U_{0,g}(E,j)$.  
Each `good' region $U_{0,g}(E,j)$ is a curved triangular region of the form 
\begin{align}\label{goodregion}
U_{0,g}(E,j) =\{(x,y)\in U_0: b_j x^{m_E}\leq y \leq B_j x^{m_E}\},
\end{align}
where $[b_j,B_j]:=I_g(E,j)$ is just a connected component of $I_g(E)$ and $(J_E+2)$ comes from the number of connected
components of $I_g(E)$. 
\\

In the above notation, the subindex $0$ in $U_{0,g}(E,j)$ indicates the algorithm is in the $0$-th stage, $g$ indicates the region is `good'. The `bad' regions are defined as:
\begin{align}
\label{badregion}
U_{0,b}(E,j) &= \{(x,y) \in U_0:  y = r x^{m_E},\, r\in I^\epsilon_{j}(E)\}
\\
&=\{(x,y) \in U_0: (r_j-\rho_0) x^{m_E} <y < (r_j+\rho_0) x^{m_E}\}, 
\end{align}
for $1\leq j\leq J_E$. 
If $\{r_{E,j}\}$ is empty, then there is no `bad' region generated by this edge and the only two `good' regions are
\begin{align}
U_{0,g}(E) = \{(x,y)\in U_0: c_E x^{m_E}< |y| < C_E x^{m_E}\}.
\end{align}

The following lemma states that $P$ behaves almost like a monomial in each `good' region $U_{0,g}(E,j)$.
\begin{lemma}\label{edge}
Let $N>0$ and $L>0$ be given. For any choices of $c_E$'s, $C_E$'s and $\rho_0$ above, 
one can choose $\epsilon_0$ sufficiently small, such that for all $E\in \mathcal N(P)$ and $(x,y) \in U^{}_{0,g}(E,j)$, one has
\begin{align}
|x^{p_{E,l}}y^{q_{E,l}}| \sim_{\rho_0} |P_E(x,y)| \geq 2^N |P(x,y) - P_E(x,y)| \label{zz}.
\end{align}
Here $(p_{E,l},q_{E,l})$ is the left vertex of the edge $E$. In addition, 
\begin{align}
|\partial_x^{\alpha}\partial_y^\beta P(x,y)| < C\min\{ 1, \,|x^{p_{E,l}-\alpha}y^{q_{E,l}-\beta}| \}\label{yy}
\end{align}
for $0\leq \alpha, \beta\leq L$. 
\end{lemma}
\begin{proof}
In the region $y=r x^{m_E}$ where $r\in I(E)$, 
$$
|P(x,y) - P_E(x,y)| <C x^{e_E+\nu},
$$
 where $\nu$ is a positive fraction (can be computed but not necessary).  
By (\ref{lb}), one has $|P_E(r)| \geq C$ for $r\in I_g(E)$, where $C =C(c_E,C_E,\rho_0, P)>0$.
Thus if $\epsilon_0 =\epsilon_0(\rho_0,\nu,C)$ is sufficiently small, then for all $(x,y)\in U_{0,g}(E,j)$ we have
$$
|P_E(x,y)| \sim_{\rho_0} |x^{p_{E,l}}y^{q_{E,l}}| \sim x^{e_{E}}> 2^N\cdot O(x^{e_E+\nu})>2^N|P(x,y)-P_E(x,y)|,
$$
%\sim x^{e_{E}+\nu} $$ 
which proves (\ref{zz}).

Now we turn to (\ref{yy}). The bound $|\partial_x^{\alpha}\partial_y^\beta P(x,y)|\lesssim 1$ is trivial. 
In the region $y=r x^{m_E}$ where $r\in I(E)$, for $0\leq \alpha,\beta\leq L$ and every $(p'',q'')\in E$, one has 
$$
|x^{p_{E,l}-\alpha}y^{q_{E,l}-\beta}| \sim |x|^{p''-\alpha}|y|^{q''-\beta}\sim |x|^{e_E-\alpha-m_E\beta},
$$
even for $p''-\alpha<0$ or/and $q''-\beta<0$. Notice $|y|\sim |x|^{m_E}$, then 
$$
|\partial_x^\alpha\partial_y^\beta (P(x,y)-P_E(x,y))|\lesssim |x|^{e_E+\nu-\alpha-m_E\beta}.
$$
Thus given $|x|$ sufficiently small, one has 
$$
|\partial_{x}^{\alpha}\partial_y^{\beta}P(x,y)| \lesssim |x^{p_{E,l}-\alpha}y^{q_{E,l}-\beta}|.
$$
This completes the the proof of (\ref{yy}).
\end{proof}
The above lemma handled the case when an edge $E$ is dominant. 
Another easy case is when a vertex $V=(p_v,q_v)$ plays the dominant role. 
In this case, let $E_l$ and $E_r$ be the edges left and right to $V$, with slopes $-1/m_{E_l}$ and $-1/m_{E_r}$  respectively. 
Then $0\leq m_{E_l }< m_{E_r}\leq \infty$. 
%Here $m_{E_l}=0$ indicates that $E_l$ is the vertical non-compact edge and $E_r$ is the horizontal non-compact edge if $m_{E_r}=\infty$. 
Consider the following region
\begin{align}\label{corner00}
U_{0,g}(V) =\{(x,y)\in U_0: C_{E_r}x^{m_{E_r}}< |y|<  c_{E_l} x^{m_{E_l}}\},
\end{align}
where $C_{E_r}$ and $c_{E_l}$ are the constants chosen in (\ref{chosen35}). 
One can always choose $\epsilon_0$ small enough 
such that the two curves $y= C_{E_r}x^{m_{E_r}}$ and $y=c_{E,l}x^{m_{E_l}}$ do not meet inside $U_0$. 
In the case $m_{E_r}=\infty$, i.e. $V$ is the rightmost vertex, 
 set 
\begin{align}\label{corner001}
U_{0,g}(V) =\{(x,y)\in U_0:  |y|<  c_{E_l} x^{m_{E_l}}\},
\end{align}
which also contains the portion of the $x$-axis inside $U_0$.
Similarly, when $m_{E_l}=0$, i.e.  $V$ is the leftmost vertex,
we replace $c_{E_l} x^{m_{E_l}}$ by $\epsilon_0$ in (\ref{corner00}).

The following lemma is the vertex analogue of Lemma \ref{edge}, whose proof is similar to that of Lemma \ref{edge}. 
\begin{lemma}\label{vertex}
Given preselected numbers $N$ and $L$, one can choose $\epsilon_0$ sufficiently small 
$($depending on $N$, $L$, $C_{E}$'s, $c_{E}$'s and $P$$)$,
 such that for $(x,y) \in U_{0,g}(V)$ one has
\begin{align}
|x^{p_{v}}y^{q_{v}}| \sim  |P_V(x,y)| \geq 2^N |P(x,y) - P_V(x,y)| \label{zzz}
\end{align}
and 
\begin{align}
|\partial_x^{\alpha}\partial_y^\beta P(x,y)| < C \min\{1, |x^{p_{v}-\alpha}y^{q_{v}-\beta}| \}\label{yyy}
\end{align}
for $0\leq \alpha, \beta\leq L$. 
\end{lemma}
Now let 
\begin{align}
\mathcal G_{0}(P_0)=\mathcal G_{0}(P) =\{U_{0,g}(V):V\in \mathcal V(P)\} \cup \{U_{0,g}(E,j): E\in\mathcal E(P)\,\, \mbox{and all}\,\,j \}, 
\end{align}
which are the collection of `good' regions in the 0-th stage.
 In addition, we say $U_{0,g}\in \mathcal G_{0}(P_0)$ is defined by $(E,m_E)$ if $U_{0,g}=  U_{0,g}(E,j)$ for some $E$ and $j$, where $-1/m_E$ is the slope of $E$, or defined by an edge $E$ for short. Similarly, $U_{0,g}\in \mathcal G_{0}(P_0)$ is defined by $(V,m_l,m_r)$ represents $U_{0,g}=  U_{0,g}(V)$ and $-1/m_l$, $-1/m_r$ are the slopes of the edges left and right to $V$, or defined by a vertex $V$ for short. 

Now we focus on the `bad' regions 
\begin{align}\label{bad7}
U_{0,b}(E,j) = \{(x,y) \in U_0: (r_j-\rho_0) x^{m_E} <y < (r_j+\rho_0) x^{m_E}\}.
\end{align}
Set
\begin{align}
\mathcal B_0 (P_0)=\mathcal B_0 (P)= \{U_{0,b}(E,j): E\in \mathcal E(P) \q\mbox{and} \q 1\leq j\leq J_E\}
\end{align}
which is the collection of `bad' regions in the $0$-stage. If $U_{0,b}\in \mathcal B_0(P_0)$ has the form of (\ref{bad7}), 
we say $U_{0,b}$ is defined by $(E, y=r_jx^{m_E})$ or defined by $y=r_jx^{m_E}$ for short. 
The following graph demonstrates a partition of $U$ (in the first quadrant) into `good' and `bad' regions, 
according to the analytic function $P(x,y)=xy(y^2-x)(y-x^2)(y-\sum_{n=1}^\infty x^n)^2$.
Notice that the function $R(x,y)=xy(y^2-x)(y-x^2)(y-x)^2$ also has the same partition.
Here, 
we choose $P(x,y)$ for the purpose of illustrating the convergence of the algorithm later. 
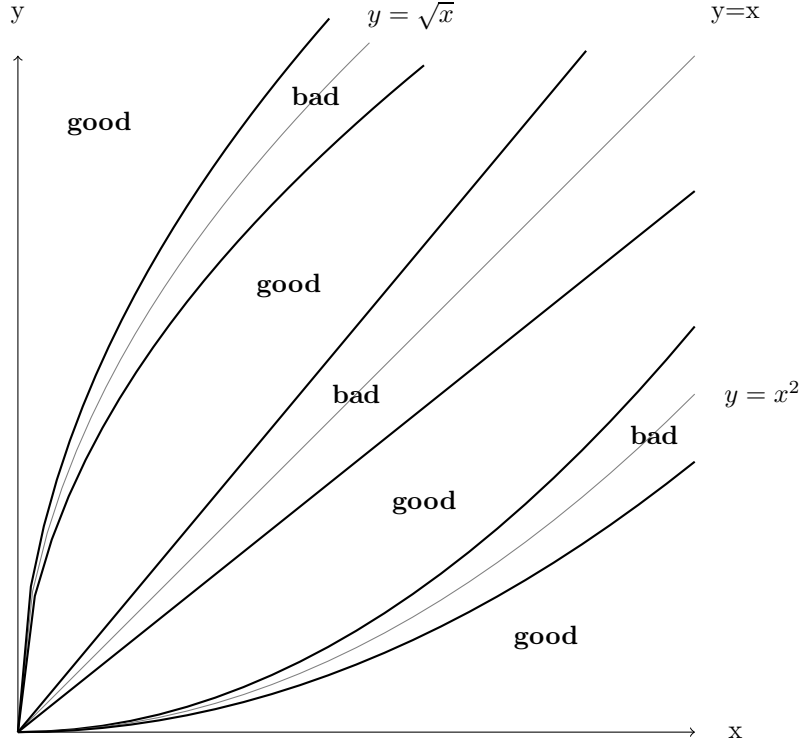
\newline
\newline

\begin{tikzpicture}[xscale=18,yscale=18]
\draw [<->] (0,0.5) -- (0,0) -- (0.5,0);
\draw[  thick, domain=0:0.5] plot (\x, {1.2*\x*\x});
\draw[ gray, domain=0:0.5] plot (\x, {\x*\x});
\draw[ thick, domain=0:0.5] plot (\x, {0.8*\x*\x});
\draw[ gray, domain=0:0.5] plot (\x, {\x});
\draw[ thick, domain=0:0.5] plot (\x, {(0.8)*\x});
\draw[ gray, domain=0:0.26] plot (\x, {sqrt(\x)});
\draw[ thick, domain=0: 0.23] plot (\x, {1.1*sqrt(\x)});
\draw[ thick, domain=0: 0.3] plot (\x, {0.9*sqrt(\x)});
\draw[  thick,domain=0:0.42] plot (\x, {(1.2)*\x});
\node at (0.53,0) {x};
\node at (0,0.53) {y};
\node at (0.53,0.53) {y=x};
\node at (0.55,0.25) {$y=x^2$};
\node at (0.29,0.53) {$y=\sqrt {x}$};
\node at (0.25,0.25) {\textbf{bad}};
\node at (0.47, 0.22) {\textbf{bad}};
\node at (0.22, 0.47) {\textbf{bad}};
\node at (0.39, 0.07)  {\textbf{good}};
\node at (0.3, 0.17)  {\textbf{good}};
\node at (0.2, 0.33)  {\textbf{good}};
\node at (0.06, 0.45)  {\textbf{good}};

\node at (0.27,-0.05) {`good' regions and `bad' regions of $P(x,y)=xy(y^2-x)(y-x^2)(y-\sum_{n=1}^\infty x^n)^2$ };
\node at (0.27,-0.08) {in the first quadrant. };
\end{tikzpicture}
We summarize the above discussion as follows:
\begin{proposition}[\textbf{A Single step of Partition}]\label{single}\q\\
Let $U$ be a standard region and $P$ be a real analytic function. If the size of $U$ is sufficiently small, then $U$ can be partitioned into two families of curved triangular regions: $\mathcal G_0(P)$ and $\mathcal B_0(P)$. For each $U_{0,g}\in \mathcal G_0(P)$, $U_{0,g}$ is defined by (\ref{goodregion}) or (\ref{corner00}). The behaviors of $P$ in $U_{0,g}$ are characterized by Lemma \ref{edge} or Lemma \ref{vertex}. Each $U_{0,b}\in \mathcal B_0(P)$ is defined by (\ref{bad7}). Finally, the cardinalities of $\mathcal G_0(P)$ and $\mathcal B_0(P)$ are finite, depending on $P$.    
\end{proposition}

\subsection{\large The resolution algorithm Part II: Iterations}\q\q  
\\

 % $\mathcal B_n$ is defined similarly in the $n$-stage. 
The next step is to iterate Proposition \ref{single}. One main problem is that the region $U_{0,b}\in \mathcal B_0(P)$ is not standard. 
Nevertheless, via an appropriate change of variables, we can always turn
a $U_{0,b}$ into a subset of a standard region $U_1$. 
%Here $P_0=P$ and $P_1$ is an analytic function of $(x^{1/M},y)$. 
Proposition \ref{single} can then be applied to $[U_1,P_1]$, where $P_1$ is obtained from $P_0=P$ via this change of variables. 
Notice that the arguments in the previous subsection work equally well for analytic functions of $(x^{1/M},y)$. 
The flowchart below illustrates the main ideas of how the algorithm runs. 
The letter `g' represents a `good' region, 
while `b' represents a bad region. 
Each time, we choose one `bad' region and `rescale' it (via change of variable) into a subset of a standard region.

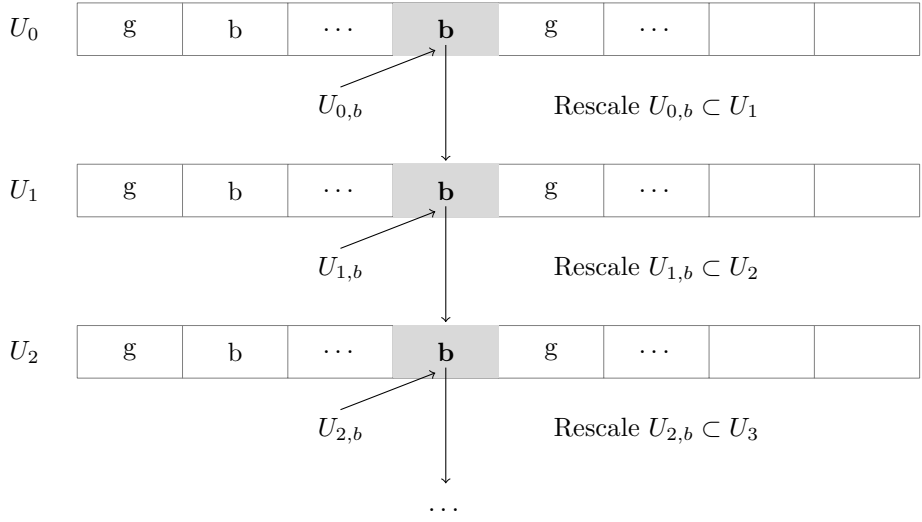
\vspace{0.5in}
\begin{tikzpicture}[xscale=1.4,yscale=0.7]
\draw [help lines] (0,0) grid (8,1);
\node at (0.5,0.5) {g};
\node at (1.5,0.5) {b};
\node at (2.5,0.5) {\dots};
%\node at (3.5,0.5) {b};
\node at (4.5,0.5) {g};
\node at (5.5,0.5) {\dots};
\fill[gray!30 ] (3,0)--(4,0)--(4,1)--(3,1);
\node at (3.5,0.5) {\textbf{b}};
%\draw [<->] (5,0) -- (0,0) -- (0,5);
%\draw (4,0) -- (0,4);
%\draw[dashed,ultra thick]
   % (1.5,3.5) to [out=-80,in=135] (2.5,1.5);
%\draw [dashed,ultra thick]
 % (2.5,1.5) to [out=-45,in=160] (4.2,0.5);
\draw [->] (3.5,0.2) --(3.5,-2);
\node at (-0.5,0.5) {$U_0$};
\node at (2.5,-1) {$U_{0,b}$};
\draw [->] (2.5,-0.6) --(3.4,0.1);
\node at (5.5,-1) {Rescale $U_{0,b} \subset U_1$};
\end{tikzpicture}

\begin{tikzpicture}[xscale=1.4,yscale=0.7]
\draw [help lines] (0,0) grid (8,1);
\node at (0.5,0.5) {g};
\node at (1.5,0.5) {b};
\node at (2.5,0.5) {\dots};
%\node at (3.5,0.5) {b};
\node at (4.5,0.5) {g};
\node at (5.5,0.5) {\dots};
\fill[gray!30 ] (3,0)--(4,0)--(4,1)--(3,1);
\node at (3.5,0.5) {\textbf{b}};
%\draw [<->] (5,0) -- (0,0) -- (0,5);
%\draw (4,0) -- (0,4);
%\draw[dashed,ultra thick]
   % (1.5,3.5) to [out=-80,in=135] (2.5,1.5);
%\draw [dashed,ultra thick]
 % (2.5,1.5) to [out=-45,in=160] (4.2,0.5);
\draw [->] (3.5,0.2) --(3.5,-2);
\node at (-0.5,0.5) {$U_1$};
\node at (2.5,-1) {$U_{1,b}$};
\draw [->] (2.5,-0.6) --(3.4,0.1);
\node at (5.5,-1) {Rescale $U_{1,b} \subset U_2$};
\end{tikzpicture}

\begin{tikzpicture}[xscale=1.4,yscale=0.7]
\draw [help lines] (0,0) grid (8,1);
\node at (0.5,0.5) {g};
\node at (1.5,0.5) {b};
\node at (2.5,0.5) {\dots};
%\node at (3.5,0.5) {b};
\node at (4.5,0.5) {g};
\node at (5.5,0.5) {\dots};
\fill[gray!30 ] (3,0)--(4,0)--(4,1)--(3,1);
\node at (3.5,0.5) {\textbf{b}};
%\draw [<->] (5,0) -- (0,0) -- (0,5);
%\draw (4,0) -- (0,4);
%\draw[dashed,ultra thick]
   % (1.5,3.5) to [out=-80,in=135] (2.5,1.5);
%\draw [dashed,ultra thick]
 % (2.5,1.5) to [out=-45,in=160] (4.2,0.5);
\draw [->] (3.5,0.2) --(3.5,-2);
\node at (-0.5,0.5) {$U_2$};
\node at (2.5,-1) {$U_{2,b}$};
\draw [->] (2.5,-0.6) --(3.4,0.1);
\node at (5.5,-1) {Rescale $U_{2,b} \subset U_3$};
\node at (3.5,-2.5) {\dots};
\end{tikzpicture}

\vspace{0.5in}
Before diving into the details, we introduce the following notations which can help keep track on certain invariants as the iterations go deeper. 
\begin{definition}\label{def}
Let $(p_l,q_l)$ and $(p_r,q_r)$ be the leftmost and rightmost vertices of $\mathcal N(P)$, the Heights of $\mathcal N(P)$ or $P$ are defined as
\begin{align*}
&{\rm  Hght}(\mathcal N(P))={\rm Hght}(P) = q_l-q_r,
\\
&{\rm Hght}^*(\mathcal N(P))={\rm Hght}^*(P) = q_l.
\end{align*}
For an edge $E\in \mathcal E(P)$, let $(p_{E,l},q_{E,l})$ and $(p_{E,r},q_{E,r})$ be its left and right vertices. Then the height of this edge is defined to be
$$
{\rm Hght}(E) = q_{E,l}-q_{E,r}.
$$
If $\{r_{E,j}\}_{1\leq j\leq J_E}$ is the set of non-zero roots of $P_E(r)$ of orders $\{s_{E,j}\}_{1\leq j\leq J_E}$, then we define the order of $E$ as 
\begin{align}
{\rm Ord}(E) = \sum\limits_{j=1}^{J_E} s_{E,j}
\end{align}
and the order of $P$ to be 
\begin{align}
{\rm Ord}(P) = \sum\limits_{E\in \mathcal E(P)} {\rm Ord}(E)=  \sum\limits_{E\in \mathcal E(P)}\sum\limits_{j=1}^{J_E} s_{E,j}.
\end{align}
Finally, we say $r$ is a root of $P(x,y)$ or $\mathcal N(P)$ if $r=r_{E,j}$ for some $E\in \mathcal E(P)$ and some $1\leq j\leq J_E$.
\end{definition}

Under the above notation, 
\begin{align}
{\rm Ord}(E)\leq {\rm Hght}(E)=q_{E,l}-q_{E,r}  
\end{align}
and
\begin{align}
{\rm Ord}(P) \leq {\rm Hght}(P)=q_l-q_r\leq q_l= {\rm Hght}^*(P).
\end{align}
Notice that  $$\#\mathcal B_0(P) = \sum_{E\in \mathcal E(P)} J_E \leq \sum_{E\in \mathcal E(P)} {\rm Ord}(E) ={\rm  Ord}(P),$$
which gives 
\begin{lemma}
$$
\# \mathcal B_0(P) \leq {\rm Ord}(P)  \leq   {\rm Hght}^*(P).
$$
\end{lemma}

 Choose a $U_{0,b}\in \mathcal B_0(P_0)$ and assume it is defined by $y=r_0x^{m_0}$. The next step is to utilize change of variables to turn $[U_{0,b},P]$ into a standard pair. % s.t. previous analysis can be performed on.
  Adopt the previous notations $[U_0,P_0]=[U,P]$ and $(x_0,y_0)=(x,y)$ and choose $x$ to be the principal 
 variable  
which will be unchanged during iterations, i.e. $x=x_n$ for all $n\in\mathbb N$. 
Change variables 
\[ \left\{ \begin{array}{ll}
         x_0=x_1 \\
        y_0=(r_0+y_1)x_1^{m_0}
        \end{array} \right. \] 
then the region $U_{0,b}$ under the new coordinates/variables is  
\[ \left\{ \begin{array}{ll}
\,\,\,\,0< x_1<\epsilon_0 \\
-\rho_0<  y_1<\rho_0.
%\end{align*}
\end{array} \right. \] 
Since $\epsilon_0\neq \rho_0$, this region is not standard, but 
 can be embedded into a larger standard region $U_1$. 
Let $\epsilon_1$ be sufficiently small but greater than $\rho_0$. Set

\[ \left\{ \begin{array}{ll}
         P_1(x_1,y_1) = P(x_1, (r_0+y_1)x_1^{m_0}) \\
         \\
        U_1= \{(x_1,y_1):\, 0<x_1<\epsilon_1, -\epsilon_1 <y_1 <\epsilon_1\}.
\end{array} \right. \] 
Then $[U_1, P_1]$ is a standard pair in the coordinates $(x_1,y_1)$.
 By applying Proposition \ref{single} to $[U_1,P_1]$, a finite collection $\mathcal G_1(P_1)$ of `good' regions $U_{1,g}$'s and a finite collection $\mathcal B_1(P_1)$ of `bad' regions $U_{1,b}$'s are obtained. In a `good' region $U_{1,g}$, the function $P_1(x_1,y_1)$ behaves like a monomial of $(x_1,y_1)$ and no further treatment is required. 
 For the `bad' regions,
choose a $U_{1,b}\in \mathcal B_1(P_1)$ and assume $U_{1,b}$ is defined by $y_1=r_1x_1^{m_1}$, i.e. 
$$
U_{1,b} = \{(x_1,y_1)\in U_1: (r_1-\rho_1)x_1^{m_1} < y_1 <(r_1+\rho_1)x_1^{m_1}\}.
$$
where $\rho_1$ is an analogue of $\rho_0$ in the $0$-th stage of iteration.  
Like what has been done, perform the following change of variables
\[ \left\{ \begin{array}{ll}
         x_1=x_2 \\
        y_1=(r_1+y_2)x_2^{m_1},
        \end{array} \right. \] 
choose $\epsilon_2$ sufficiently small but greater than $\rho_1$ and set 
%Then a new standard pair $[U_2,P_2]$ is obtained: %\footnote{Again, we choose one quadrant and still use $U_2$ to represent that quadrant.}$ is obtained as follows:
\[ \left\{ \begin{array}{ll}
P_{2}(x_2,y_2)=P_1(x_2,(r_1+y_2)x_2^{m_1})
\\
\\
U_{2}=\{(x_2,y_2):\, 0<x_2<\epsilon_2, -\epsilon_2<y_1<\epsilon_2\}
\end{array} \right. \] 
Same procedures are repeated on the standard pair $[U_2,P_2]$ and so on. 
Consequently, 
a collection of standard pairs $\{[U_n,P_n]\}$ is obtained from these iterations.
Notice that this collection forms a tree structure with $[U_0,P_0]=[U,P]$ being the top, 
 with the parameter $n$ representing the level of iterations.
Sometimes we need more information than merely the level of iterations. In those cases, 
we use the notation $[U_n,P_n]=  [U_{n,\boldsymbol{\alpha}},P_{n,\boldsymbol{\alpha}}]$, where 
% represents $[U_n,P_n]$ is obtained from the $n$-th stage of iteration (or we say a $n$-th generation of $[U_0,P_0]$). Notice that for a $n$, there can be many $[U_n,P_n]$ and the structure of $\{[U_n,P_n]\}$ is a tree (non-linear). If we want to specify the `identity' of $[U_n,P_n]$, set $[U_n,P_n]=  [U_{n,\boldsymbol{\alpha}},P_{n,\boldsymbol{\alpha}}]$ where the subindex $\bsb{\alpha}$ represents the `path' from $[U_0,P_0]$ to $[U_{n,\boldsymbol{\alpha}},P_{n,\boldsymbol{\alpha}}]$. 
%The subindex $\bsb\alpha$ can also be viewed as the code that compresses the genealogy information which is needed to obtain $[U_{n,\bsb{\alpha}}, P_{n,\bsb{\alpha}}]$ from $[U_0,P_0]$; or conversely, $P_0(x_0,y_0)$ can be `decoded' from $P_{n,\bsb{\alpha}}(x_n,y_n)$ by $\bsb{\alpha}$. More precisely:
%\\

($\star$) $\bsb\alpha$ contains the information of the change of variables, i.e. for $0\leq k\leq n-1$ the following is known if $\bsb\alpha$ is given:
\begin{align}\label{cc01}
\begin{cases}
 x_k=x_{k+1} \\ 
 y_k= (r_{k}+y_{k+1})x_{k+1}^{m_{k}}.
 \end{cases}
\end{align}
We also use $U_{n,b,\bsb\alpha}$ and $U_{n,g,\bsb\alpha}$ to represent any `bad' region and any `good' region in $U_{n,\bsb\alpha}$. Since $U_{n,\bsb\alpha}$ may have more than one such regions, we list them by $U_{n,b,\bsb\alpha,j}$ and $U_{n,g,\bsb\alpha,j}$ when necessary. The cardinality in $j$ is uniformly bounded (see Lemma \ref{count}) and thus there is no need to specify its range. 
Notice that $U_{n,g,\bsb\alpha,j}$ is a leaf of the tree, i.e. it has no child and no further analysis is needed. 
%
%
%
% New subindex is introduced here~~~~~~~~~~~~~~~~~~~~~~~~~~~~~~~~~~~~~
%
%
%

Both notations are being used here: with and without the subindex $\bsb\alpha$. To avoid confusion, we follow the rules below:
\\
(1) The one without such subindex is our primary choice. We often use $[U_n,P_n]$ to represent an arbitrary pair from the $n$-th stage of iteration. 
\\
(2) The other is used occasionally. It is employed typically when at least two different pairs from the same stage of iterations are mentioned simultaneously.

These conventions also apply to $U_{n,g}$'s, $U_{n,b}$'s, $U_{n,g,\bsb\alpha}$ and $U_{n,b,\bsb\alpha}$. 
\begin{example}\label{EX08}
The following graphs demonstrate the first step of the algorithm, for the given analytic function $P(x,y)=xy(y^2-x)(y-x^2)(y-\sum_{n=1}^\infty x^n)^2$. The Newton polyhedron $\mathcal N(P)$ has 3 compact edges: $E_1$, $E_2$ and $E_3$, see Figure A below.
\vspace{0.3in}

\begin{center}
\begin{tikzpicture}[scale=0.78]
\draw [<->,thick] (0,11) node (yaxis) [above] {$y$}
        |- (11,0) node (xaxis) [right] {$x$};
  \draw[help lines] (0,0) grid (10,10);
         
           \fill[gray!30] 
	 			  (1,10)--
			        (1,6)--
			        	(2,4)--
				%(3,3)--
				(4,2)--
				(6,1)--
				(10,1)--(10,10);
         
                \fill
        (6,1)   circle (2pt) 
       
	(4,2)  circle (2pt)
	(2,4) circle (2pt)
	(1,6) circle (2pt)
       ;

        \draw [thick]	(1,10)--
			        (1,6)--
			        	(2,4)--
				%(3,3)--
				(4,2)--
				(6,1)--
				(10,1);
	%\draw[thick]   (4,2)--(7,0)--(0,14/3);

	\node at (5,5) {$\mathcal N(P)$};
	\node at (1.8,6){\tiny  $V_1$(1,6)};
	
		\node at (2.8,4) {\tiny$V_2$(2,4)};
	%\node at (3.5,3) {\tiny(3,3)};
	\node at (4.6,2.1) {\tiny$V_3$(4,2)};
	\node at (6.6,1.2) {\tiny$V_4$(6,1)};
	%\node at (2,2.5) {$m=\frac{3}{2}$};
	%\node at (5,11) {\small $P(x,y)=xy(y^2-x)(y-x^2)(y-\sum_{n=1}^\infty x^n)^2$};
%	\node at (5,1.5) {(5,2)};
	\node at (13.1,8)  {${\rm Hght}^*(P)=6$};
	\node at (13,7)  {${\rm Hght}(P)=5$};
	\node at (13,6)  {${\rm Ord}(E_1)=1$};
	\node at (13,5)  {${\rm Ord}(E_2)=2$};
	\node at (13,4)  {${\rm Ord}(E_3)=1$};
	\node at (12.9,3)  {${\rm Ord}(P)=5$};
\node at (2,5)  {\tiny $E_1$};
\node at (3.2,3.2)  {\tiny $E_2$};
\node at (5.5,1.6)  {\tiny $E_3$};

 \node at (5,-1) {Figure A};

 % \draw (0cm,2mm) -- (30:2cm);
\end{tikzpicture}
\end{center}
In the edge $E_1$, $P_{E_1} = x(y^2-x)y^4$
has only two non-zero roots and the corresponding `bad' regions are defined by $y=x^{\frac{1}{2}}$ and $y=-x^{\frac{1}{2}}$. 
To handle the former, 
change variables: $(x,y)=(x_1, x_1^{\frac{1}{2}}(1+y_1))$ yields 
$P_1(x_1,y_1)=P(x_1,x^{\frac{1}{2}}(1+y_1)) = x_1^{\frac{7}{2}}y_1\cdot O(1)$. Here $O(1)$ represents a function which
is non-vanishing at the origin. In Figure B, $\mathcal N(P_1)$ 
has only one vertex and the algorithm stops. The latter is similar. 
\begin{center}

\begin{tikzpicture}[scale=0.6]
\draw [<->,thick] (0,11) node (yaxis) [above] {$y_1$}
        |- (11,0) node (xaxis) [right] {$x_1$};
  \draw[help lines] (0,0) grid (10,10);
         
           \fill[gray!20] 
	 			  (4,10)--
			        (4,1)--
			        	(10,1)--
				%(3,3)--
				(10,10);%--
				%(6,1)--
				%(10,1)--(10,10);
         
                \fill
        (4,1)   circle (2pt) 
       
	%(4,2)  circle (2pt)
	%(2,4) circle (2pt)
	%(1,6) circle (2pt)
       ;

        \draw [thick]	(10,1)--
			        (4,1)--
			        	(4,10);%--
				%(3,3)--
%				(4,2)--
%				(6,1)--
%				(10,1);
	%\draw[thick]   (4,2)--(7,0)--(0,14/3);

	\node at (7,5.5) {$\mathcal N(P_1)$};
%	\node at (1.8,6){\tiny  $V_1$(1,6)};
	
%		\node at (2.8,4) {\tiny$V_2$(2,4)};
	\node at (3,1.5) { $V'(4,1)$};
%	\node at (4.6,2.1) {\tiny$V_3$(4,2)};
%	\node at (6.6,1.2) {\tiny$V_4$(6,1)};
%	\node at (2,2.5) {$m=\frac{3}{2}$};
	%\node at (4.5,11){$(x,y)=(x_1,x_1^{1/2}+y_1x_1^{1/2})$};
	%\node at (4.5,12) {\small After change of variables };
	%\node at (14, 8)  {\small $P_1(x_1,y_1)=x_1^{7/2}y_1\cdot O(1)$};
	\node at (13,7) {Hght$^*(P_1)=1$};
	\node at (12.88,6) {Hght$(P_1)=0$};
	\node at (12.78,5) {Ord$(P_1)=0$};
	%\node at (5,-1) {The $\mathcal N (P_1)$ has only vertex, the algorithm ends. };
%	\node at (5,1.5) {(5,2)};
	\node at (5,-1) {Figure B};
%\node at (2,5)  {\tiny $E_1$};
%\node at (3.2,3.2)  {\tiny $E_2$};
%\node at (5.5,1.6)  {\tiny $E_3$};
%\node at 

 % \draw (0cm,2mm) -- (30:2cm);
\end{tikzpicture}
\end{center}
There is only one root for $P_{E_2} = -x^2y^2(y-x)^2$
and the corresponding `bad' region is defined by $y=x$. 
Change variables: $(x,y)=(x_1, x_1(1+y_1))$ yields 
$P_1(x_1,y_1)=x_1^6(y_1-\sum_{n=1}^\infty x^n)^2\cdot O(1)$. 
 But $\mathcal N(P_1)$ 
still has one edge which also has one non-zero root; See Figure C. The algorithm does not stop. 
Doing the change variables $(x_{k-1},y_{k-1})=(x_k,x_{k}(1+y_k))$ only gives 
$\mathcal N(P_k)= \mathcal N(P_{k-1})+(2,0)$ for $k\geq 2$.

\begin{center}
\begin{tikzpicture}[scale=0.6]
\draw [<->,thick] (0,11) node (yaxis) [above] {$y_1$}
        |- (11,0) node (xaxis) [right] {$x_1$};
  \draw[ help lines] (0,0) grid (10,10);
         
           \fill[gray!20] 
	 			  (6,10)--
			        (6,2)--
			        	(8,0)--
				%(3,3)--
				(10,0)--%--
				(10,10);
				%(6,1)--
				%(10,1)--(10,10);
         
                \fill
        (6,2)   circle (2pt) 
       
	%(4,2)  circle (2pt)
	%(2,4) circle (2pt)
	%(1,6) circle (2pt)
       ;

        \draw [thick]	(6,10)--
			        (6,2)--
			        	(8,0)--
				(10,0);%--
				%(3,3)--
%				(4,2)--
%				(6,1)--
%				(10,1);
	%\draw[thick]   (4,2)--(7,0)--(0,14/3);

	\node at (8,5) {$\mathcal N(P_1)$};
%	\node at (1.8,6){\tiny  $V_1$(1,6)};
	
%		\node at (2.8,4) {\tiny$V_2$(2,4)};
	\node at (8.9,0.3) {\tiny $V'_2(8,0)$};
	\node at (6.8,2.3) {\tiny $V'_1(6,2)$};
%	\node at (4.6,2.1) {\tiny$V_3$(4,2)};
%	\node at (6.6,1.2) {\tiny$V_4$(6,1)};
%	\node at (2,2.5) {$m=\frac{3}{2}$};
	%\node at (4.5,11) {$(x,y)=(x_1,x_1+y_1x_1)$};
	%\node at (4.5,12) {\small After change of variables };
	%\node at (15, 8) {\small $P_1(x_1,y_1)=x_1^6(y_1-\sum_{n=1}^\infty x^n)^2\cdot O(1)$};
	\node at (13,7) {Hght$^*(P_1)=2$};
	
	\node at (12.88,6) {Hght$(P_1)=2$};
	
	\node at (12.78,5) {Ord$(P_1)=2$};
	\node at (12.78,4) {Ord$(E')=2$};

	\node at (7.3,1.3) {\tiny$ E'$};
	\node at (5,-1) {Figure C};
%	\node at (5,1.5) {(5,2)};

%\node at (2,5)  {\tiny $E_1$};
%\node at (3.2,3.2)  {\tiny $E_2$};
%\node at (5.5,1.6)  {\tiny $E_3$};
%\node at 

 % \draw (0cm,2mm) -- (30:2cm);
\end{tikzpicture}

\end{center}

Finally, $P_{E_3} = -x^4y(y-x^2)$. 
The only non-zero root is $y=x^2$. 
Change of variables: $(x,y)=(x_1, x_1^2(1+y_1))$ yields 
$P_1(x_1,y_1)=x_1^8y_1\cdot O(1)$. 
The algorithm stops; see Figure D.
%See Figure D, $\mathcal N(P_1)$ 
%has only one vertex and the algorithm stops. 

\begin{center}

\begin{tikzpicture}[scale=0.6]
\draw [<->,thick] (0,11) node (yaxis) [above] {$y_1$}
        |- (11,0) node (xaxis) [right] {$x_1$};
  \draw[help lines] (0,0) grid (10,10);
         
           \fill[gray!20] 
	 			  (8,10)--
			        (8,1)--
			        	(10,1)--
				%(3,3)--
				(10,10);%--
				%(6,1)--
				%(10,1)--(10,10);
         
                \fill
        (8,1)   circle (2pt) 
       
	%(4,2)  circle (2pt)
	%(2,4) circle (2pt)
	%(1,6) circle (2pt)
       ;

        \draw [thick]	(8,10)--
			        (8,1)--
			        	(10,1);%--
				%(3,3)--
%				(4,2)--
%				(6,1)--
%				(10,1);
	%\draw[thick]   (4,2)--(7,0)--(0,14/3);

	\node at (9,6) {$\mathcal N(P_1)$};
%	\node at (1.8,6){\tiny  $V_1$(1,6)};
	
%		\node at (2.8,4) {\tiny$V_2$(2,4)};
	\node at (9.1,1.5) { V'(8,1)};
%	\node at (4.6,2.1) {\tiny$V_3$(4,2)};
%	\node at (6.6,1.2) {\tiny$V_4$(6,1)};
%	\node at (2,2.5) {$m=\frac{3}{2}$};
	%\node at (4.5,11){$(x,y)=(x_1,x_1^2+y_1x_1^2)$};
	%\node at (4.5,12) {\small After change of variables };
%	\node at (14, 8)  {\small $P_1(x_1,y_1)=x_1^8y_1\cdot O(1)$};
	\node at (13,7) {Hght$^*(P_1)=1$};
	\node at (12.88,6) {Hght$(P_1)=0$};
	\node at (12.78,5) {Ord$(P_1)=0$};
%	\node at (5,1.5) {(5,2)};
\node at (5,-1) {Figure D};%{The $\mathcal N (P_1)$ has only vertex, the algorithm ends. };

%\node at (2,5)  {\tiny $E_1$};
%\node at (3.2,3.2)  {\tiny $E_2$};
%\node at (5.5,1.6)  {\tiny $E_3$};
%\node at 

 % \draw (0cm,2mm) -- (30:2cm);
\end{tikzpicture}

\end{center}

\end{example}
\vspace{0.5in}

It is worth mentioning that, the change of variables: $(x_n,y_n) \to (x_{n+1},y_{n+1})$ maps one-to-one from $U_{n,b}$ to a rectangular subset of $U_{n+1}$. 
Thus one can also embed such subset back into $U_{n}$ and etc.. 
%$$\dots\to U_{n+2}\to U_{n+1}\to U_n\to \dots\to U_1\to U_0 .$$
In addition, if the change of variables: $(x,y) \to (x_n,y_n)$ is specified, then we can use both coordinates $(x,y)$ and $(x_n,y_n)$ to identify
points in $ U_n$ (or $U_{n,b}$, $U_{n,g}$). The relation between the two coordinates $(x,y)=\rho_n^{-1}(x_n,y_n)$ is given by
%identify $(x,y) \in U_n$ (or $U_{n,b}$, $U_{n,g}$) with $(x_n,y_n)\in U_n$ (or $U_{n,b}$, $U_{n,g}$). 
% To be precise, there is a deffeomorphism $\rho_n^{-1}$ : 
%\begin{align}
%\rho_n^{-1}: U_n&\longmapsto \rho_n^{-1}(U_n)\subset U_0
%\\
%(x_n,y_n) &\longmapsto(x,y)\label{cc03}
%\end{align}
%where (\ref{cc03}) is defined by the composition of change of variables (\ref{cc01}). 
%More precisely, $(x,y)=\rho_n^{-1}(x_n,y_n)$ means
\begin{align}\label{cc09}
\begin{cases}
 x=x_{n} \\ 
 y= r_0x^{m_0}+r_1x^{m_0+m_1}+\dots+r_{n-1}x^{m_0+\dots+m_{n-1}}+y_nx^{m_0+\dots+m_{n-1}}.
 \end{cases}
\end{align}
In the above notation, the $(k+1)$-th stage of iteration is generated by the `bad' region defined by $y_k =r_kx^{m_k}$ for $0\leq k\leq n-1$. 
Under this notation, %$P_n$ is defined by:
$
P_n=P\circ \rho_n^{-1}.
$
We should also specify $\rho_n^{-1}$ 
to $\rho^{-1}_{n,\bsb\alpha}$, if $U_{n}$ is specified to $U_{n,\bsb\alpha}$.
%Notice that for all $j$, $U_{n,g,\bsb\alpha,j}$'s and $U_{n,b,\bsb\alpha,j}$'s 
%are sharing the same $\rho^{-1}_{n,\bsb\alpha}$ with $U_{n,\bsb\alpha}$. In particular, 
%$\{\rho_{n,\bsb\alpha}^{-1}(U_{n,g,\bsb\alpha,j})\}_{n,\bsb\alpha,j}$
% are disjoint curved triangular regions in $U_0$. 
%It will become clear later that $\{\rho_{n,\bsb\alpha}^{-1}(U_{n,g,\bsb\alpha,j})\}_{n,\bsb\alpha,j}$ 
%will form a finite disjoint (dense) partition of $U_0$.% when excluding a finite collection of curves. %of the form $\{y_n=0\}$. 

To complete the algorithm,  we need to settle the following two questions:
\\
 \q\q\q(i)  \,\,In each stage of iteration, is the cardinality of $\{U_{n,b,\bsb\alpha,j}\}_{\bsb\alpha,j}$ bounded above uniformly?
\\
(ii) Does this procedure terminate after a finite steps?
\\
The answer to the second question is more delicate and extra preparation is needed, but the first can be answered by:

\begin{lemma}\label{count}
For each $n\geq 0$, the cardinality of $\{U_{n,b,\bsb\alpha,j}\}_{\bsb\alpha,j}$ 
is bounded above by ${\rm Ord}(P)$.
%$\leq Ord(P) \leq Hght(P)$.
\end{lemma}

\begin{proof}
Indeed, there is a bijection between $\{U_{0,b}\}$ and 
the non-zero roots of $P_E(r)$ $E\in \mathcal E(P)$: 
each $U_{0,b}=U_{0,b}(E,j)$ is defined by $(E, y=r_{E,j}x^{m_E})$.
 Assume the order of $r_{E,j}$ is $s_{E,j}$. 
Then $P_1(x_1,y_1)=P_{1,E,j}(x_1,y_1)$ is obtained by setting
 $P_1(x_1,y_1)=P_0(x_1,(y_1+r_{E,j})x_1^{m_E})$. 
Here, $P_{1,E,j}$ is used to specify that $P_1$ is defined by the root $r_{E,j}$. 

The following observation serves as an bridge between $\mathcal N(P_0)$ and $\mathcal N(P_1)$.
%, and eventually 
%between $\mathcal N(P_0)$ and $\mathcal N(P_n)$. 
Let $(p_{E,l},q_{E,l})$ be the left vertex of $E$ and $(p_{1,l},q_{1,l})$ 
be the leftmost vertex of $\mathcal N(P_1)$ then 
\begin{align}\label{track}
	\left\{ \begin{array}{rcl}
      p_{1,l} = p_{E,l}+m_E\cdot q_{E,l}
	\\
	\\  q_{1,l} = s_{E,j}.
	                \end{array}\right.
\end{align}
This implies the order $s_{E,j}$ (in the $0$-stage) is equal to the Hght$^*(P_{1,E,j})$.

To prove (\ref{track}), notice  
$$
P_E(x,y) = P_E(x_1,(y_1+r_{E,j})x_1^{m_E}) = x_1^{p_{E,l}+m_E\cdot q_{E,l}}y_1^{s_{E,j}}\cdot O(1).
$$
Indeed, the fact that the degree of $y_1$ is $s_{E,j}$ follows from the fact that $r_{E,j}$ 
is a root of $P_E(r)$ of order $s_{E,j}$. Moreover, every term in 
$$
P_1(x_1,y_1)-P_E(x_1,(y_1+r_{E,j})x_1^{m_E})
$$
has a $x_1$-degree strictly greater than $(p_{E,l}+m_E\cdot q_{E,l})$. Consequently, $(p_{E,l}+m_E\cdot q_{E,l},s_{E,j})$ is the leftmost vertex of $\mathcal N(P_1)$.

Thus %Ord$(P_{1,E,j})=$ 
${\rm Ord}(P_1) \leq {\rm Hght}^*(P_1)= s_{E,j}$ and the cardinality of `bad' regions $\{U_{1,b}\}$ from a single $P_1$ is at most Ord$(P_1)\leq s_{E,j}$.
 Counting all possible $P_1$ (coming from different roots of different edges), the number of all possible $U_{1,b}$ is bounded by 
\begin{align}\label{it00}
\sum\limits_{E\in \mathcal E(P)}\sum\limits_{1\leq j\leq J_E}{\rm Ord}(P_{1,E,j}) \leq \sum\limits_{E\in \mathcal E(P)}\sum\limits_{1\leq j\leq J_E} s_{E,j} = {\rm Ord}(P). %\leq {\rm Hght}^*(P).
\end{align}
The cases when $n\geq 2$ also follow from iterating (\ref{it00}).  
\end{proof}

We now turn to the second question, which is the most crucial part of the algorithm. 
Assume the procedure does not stop. We then obtain an infinite chain of pairs:
\begin{align}\label{chain}
[U_0,P_0]\to [U_1,P_1]\to[U_2,P_2]\to\dots\to[U_n,P_n]\to[U_{n+1},P_{n+1}]\to\dots
\end{align}
We shall find certain patterns inside this chain. % (\ref{chain}).
 Specify the change of variables from $[U_n,P_n]\to[U_{n+1},P_{n+1}]$ as 

\[ \left\{ \begin{array}{ll}
         x_n=x_{n+1} \\
        y_n=(r_n+y_{n+1})x_{n+1}^{m_n} \end{array} 
        \right. \]     
          
Then $r_n$ is a root of an edge in $\mathcal N(P_n)$. 
Let $s_n$ be the order of $r_n$ and
 $(p_{n,l},q_{n,l})$ be the leftmost vertex of $\mathcal N(P_n)$ 
and $(p_n,q_n)$ be the left vertex of the edge in $\mathcal N(P_n)$ 
that defines $[U_{n+1},P_{n+1}]$. By (\ref{track}) one has 
\begin{align}\label{induction}
	\left\{ \begin{array}{lcl}
      p_{n+1}\geq p_{n+1,l} = p_{n}+m_n\cdot q_{n}
	\\
	\\  q_{n+1}\leq q_{n+1,l} = s_n \leq q_n
	                \end{array}\right.
\end{align}
and thus 
\begin{align}\label{seq}
{\rm Hght}^*(P_0)\geq {\rm Hght}(P_0) \geq s_0= {\rm Hght}^*(P_1)\geq {\rm Hght}(P_1) \geq s_1 = {\rm Hght}^*(P_2)\geq% s_2\dots 
\\
\dots\geq s_{n-1}={\rm Hght}^*(P_n)\geq {\rm Hght}(P_n) \geq s_n= {\rm Hght}^*(P_{n+1})\cdots
\end{align}
Notice that for all $n$, ${\rm Hght}(P_n)$ and $s_n$ must be positive integers. 
Otherwise, if ${\rm Hght}(P_n)=0$ then $\mathcal N(P_n)$ has no compact edge and thus no root; 
if $s_n=0$, then $\mathcal N(P_n)$ has no root. In both situations, 
the chain ends at the $n$-stage, which contradicts to our assumption. 

Since (\ref{seq}) is an infinite sequence and ${\rm Hght}^*(P_0)$ is a finite positive number, 
there is a least integer $n_0\in \mathbb N$ such that for all $n\geq n_0$ one has
\begin{align}\label{*}
{\rm Hght}^*(P_n)={\rm Hght}(P_n)=s_n= {\rm Hght}^*(P_{n_0})={\rm Hght}(P_{n_0}) = s_{n_0} >0.
\end{align} 
This implies for \textbf{every} $n\geq n_0$:\\
(i) $\mathcal N(P_n)$ has only one compact edge $E_n$,\\
(ii) in this edge $E_n$, $P_n(x_n,y_n)$ has only one root $r_n$ of order $s_n=s_{n_0}$, \\
(iii) when $P_n$ is restricted to $E_n$, $P_{n,E_n}(x_n,y_n) =c_n(y_n-r_nx_n^{m_n})^{s_{n_0}}$, for some non-zero constant $c_n$. 

This is the exact pattern we are looking for, which suggests $P_{n_0}(x_{n_0},y_{n_0})$ has a factor 
$$
y_{n_0} - \sum\limits_{n=n_0}^{\infty} r_nx_{n_0}^{m_{n_0}+m_{n_0+1}+\cdots+m_n},
$$
whose order is equal to $s_{n_0}$. 
In Example \ref{EX08}, after the change of variable $y=x_1+x_1y_1$, the function $P_1(x_1,y_1)$ has only one compact edge, 
which is factored as $(y_1-x_1)^2$. From $n_0=1$, if we keep doing change of variables $y_{n-1} =x_{n} +x_ny_{n}$, we will have 
$$
P_{n,E_n}(x_n,y_n) =c_n(y_n-x_n)^{2}. 
$$
Indeed, in this example $P_1(x_1,y_1)$ has a factor $(y_1-\sum_{n=1}^\infty x_1^n)$ of order exactly 2.

 The following lemma shows that the chain (\ref{chain}) essentially ends at the $(n_0+1)$-stage. 
\begin{lemma}\label{infty}
Assume we have an infinite chain (\ref{chain}) and $n_0$ is the constant defined in (\ref{*}). Then there is a positive integer $M$ such that 
\begin{align}\label{**}
P_{n_0}(x_{n_0},y_{n_0}) =x_{n_0}^{p_{n_0}} (y_{n_0}-f(x_{n_0}))^{s_{n_0}}Q_{n_0}(x_{n_0},y_{n_0})
\end{align}
where
\begin{align}
f(x_{n_0})= \sum\limits_{n=n_0}^{\infty} r_nx_{n_0}^{m_{n_0}+m_{n_0+1}+\cdots+m_n}
\end{align}
is an analytic function of $x_{n_0}^{1/M}$ and 
$
Q_{n_0}(x_{n_0},y_{n_0})
$
is an analytic function of $(x_{n_0}^{1/M},y_{n_0})$ with 
$
Q_{n_0}(0,0) \neq 0. 
$
\end{lemma}

\begin{proof}
To obtain $P_{n_0}(x_{n_0},y_{n_0})$ from $P_0(x,y)$, we have iterated only finitely many steps. Thus $P_{n_0}(x_{n_0},y_{n_0})$ is a real analytic function of $(x_{n_0}^{1/M},y_{n_0})$, for some $M\in\mathbb N$. 
For $n\geq n_0$, the change of variables from $[U_{n},P_{n}]$ to $[U_{n+1},P_{n+1}]$ is $x_n=x_{n+1}$ and $y_n=(y_{n+1}+r_n)x^{m_n}$. The only compact edge $E_n$ of $P_n$ is of the form 
\begin{align}\label{touse}
P_{n,E_n}(x_n,y_n)=c_nx_n^{p_{n}}(y_n-r_nx^{m_n})^{s_n},
\end{align} 
where $c_n$ is a nonzero constant and $s_n=s_{n_0}$. Using induction, it is not difficult to prove that $m_nM$ is an integer for all $n\geq n_0$. Thus $P_n(x_n,y_n)$ is a real analytic function of $(x_n^{1/M},y_n)$ for all $n\in\mathbb N$.
The Weierstrass Preparation Theorem implies 
\begin{align}\label{WPT}
P_{n_0}(x_{n_0},y_{n_0}) = x_{n_0}^{p_{n_0}}Q_{n_0}(x_{n_0},y_{n_0}) R(x_{n_0},y_{n_0}),
\end{align}
 where $Q_{n_0}(x_{n_0},y_{n_0}) $ is a non-vanishing real analytic function of $(x_{n_0}^{1/M},y_{n_0})$ and 
 \begin{align*}
  R(x_{n_0},y_{n_0})=y^{s_0}_{n_0}+a_{s_0-1}(x_{n_0})y^{s_0-1}+a_{s_0-2}(x_{n_0})y^{s_0-2}+\dots+a_{0}(x_{n_0})
 \end{align*}
  is a Weierstrass polynomial, i.e. a polynomial in $y$ with analytic coefficients in $x_{n_0}^{\frac 1 M}$. 
  Our goal is to verify the following factorization
  \begin{align}\label{WP89}
   R(x_{n_0},y_{n_0}) = (y_{n_0} - f(x_{n_0}))^{s_{n_0}}
  \end{align} 
  in the sense of formal power series.  
  Notice that this also implies that $ f(x_{n_0})$ is a real analytic function of $x_{n_0}^{\frac 1 M}$ for it equals $ -\frac {a_{s_0-1}(x_{n_0})}{s_{n_0}}$.   Therefore, (\ref{WP89}) is true pointwise in some small neighborhood of the origin.

 First, notice $(p_n+s_nm_n,0)$ is the rightmost vertex of $\mathcal N(P_n)$. By setting $y_n=0$, (\ref{touse})  yields
\begin{align}\label{--}
P_n(x_n,0) = C_n x_n^{p_{n}+s_nm_n}+O(x_n^{p_{n}+s_nm_n+\nu}), 
\end{align}
where $\nu>0$. % is positive. % has $x_n$ degree strictly larger than $(p_{n}+s_nm_n)$. 
%\begin{align}\label{gd}
%P_n(x_n,y_n) \sim x_n^{p_{n}+s_nm_n} \q\mbox{for} \q (x_n,y_n)\in U_{n,g}
%\end{align}
Consider the partial sum of $f(x_{n_0})$, 
\begin{align}\label{-}
f_k(x_{n_0})= \sum\limits_{n=n_0}^{k} r_nx_{n_0}^{m_{n_0}+m_{n_0+1}+\cdots+m_n}, \q k\geq n_0.
\end{align}
Then $y_{n_0}=y_nx_{n_0}^{m_{n_0}+m_{n_0+1}+\cdots+m_{n-1}}+f_{n-1}(x_{n_0})$, $n\geq n_0+1$ 
and
\begin{align}
P_n(x_n,y_n)=P_{n_0}(x_n,y_nx_{n_0}^{m_{n_0}+m_{n_0+1}+\cdots+m_{n-1}}+f_{n-1}(x_{n_0})).
\end{align}
By (\ref{--}), we have 
%\textbf{Check below $n$ or $n-1$ !!!!!!!!!!!}
\begin{align}\label{speed}
P_{n_0}(x_{n_0},f_{n-1}(x_{n_0}))=P_n(x_n,0) =C_n x_n^{p_{n}+s_nm_n}+O(x_n^{p_{n}+s_nm_n+\nu}).
%\to 0\,\, as \,\, n\to\infty,
\end{align}
Notice $m_n\geq \frac{1}{M}$ and $s_n=s_{n_0}$ is a positive integer, thus 
$$p_{n}+s_nm_n\to \infty\q\q \mbox{ as}\q\q  n\to \infty,$$ 
which implies (in the sense of formal power series) 
\begin{align}
P_{n_0}(x_{n_0},f(x_{n_0}))=0 =R(x_{n_0},f_{n_0}(x_{n_0})),
\end{align}
since $Q_{n_0}(x_{n_0},y_{n_0}) $ is non-vanishing. 
This yields $R_{n_0}(x_{n_0},y_{n_0})$ has a factor $(y_{n_0}-f(x_{n_0}))$. To show its order is exactly $s_{n_0}$, assume 
\begin{align}\label{sp07}
R(x_{n_0},y_{n_0}) = (y_{n_0}-f(x_{n_0}))^{s}\bar R(x_{n_0},y_{n_0}),
\end{align}
where $\bar R_{n_0}(x_{n_0},f(x_{n_0}))\neq 0$. Write
$$
\bar R(x_{n_0},y_{n_0}) =\bar R(x_{n_0},f(x_{n_0}))+\bigg(\bar R(x_{n_0},y_{n_0})-\bar R(x_{n_0},f(x_{n_0}))\bigg)
$$
%and $\big(\bar R(x_{n_0},y_{n_0})-\bar R(x_{n_0},f(x_{n_0}))\big)$
and the latter term is divisible by $(y_{n_0}-f(x_{n_0}))$. 
Assume the leading term of $\bar R(x_{n_0},f(x_{n_0}))$ is $C x_{n_0}^{A}$ (the nonzero term with lowest degree). 
Then 
$$
\bar R(x_{n_0},f_{n-1}(x_{n_0})) =Cx_{n_0}^A+O(x_{n_0}^{A+\nu}) \q\q\mbox{as}\q\q n\to\infty
$$
for some $\nu >0$. Combining (\ref{sp07}) and (\ref{WPT}), one has
\begin{align}\label{bb8}
P_{n_0}(x_{n_0},f_{n-1}(x_{n_0})) =Cx_{n_0}^{p_{n_0}+s(m_{n_0}+\dots+m_n)+A}+
O(x_{n_0}^{p_{n_0}+s(m_{n_0}+\dots+m_n)+A+\nu}),%\q\mbox{as}\q n\to\infty.
\end{align}
%Here we have invoked the fact $p_{k+1}=p_k+sm_k$ for $k\geq n_0$. 
as $n\to \infty$. Notice that, for all $n> n_0$, 
\begin{align}
p_n=p_{n-1}+s_{n-1}m_{n-1}=p_{n-1}+s_{n_0}m_{n-1}.
\end{align}
Iterating this identity yields  
\begin{align}
p_n+s_nm_n=p_{n_0}+s_{n_0}(m_{n_0}+\dots+m_n).
\end{align}
Comparing (\ref{speed}) and (\ref{bb8}) yields 
\begin{align*}
\begin{cases}
A=0
\\
s=s_0
\end{cases}
\end{align*}
as desired. 

\end{proof}

\begin{remark}
The purpose of employing the Weierstrass Preparation Theorem (WPT) here is to verify the analyticity of 
$f(x_{n_0})$, but its use is not essential here. There should be other methods. 
For example, the referees kindly point out that one can employ the IFT instead.
%Basically, one takes the $(s_0-1)$ $y_0$-derivatives of $\f {P_{n_0}(x_{n_0},y_{n_0}) }{x_{n_0}^{p_0}}$ 
%
I find the use of the WPT here is quite interesting, 
for $(y_{n_0}-f(x_{n_0}))^{s_0}$ is indeed the Weierstrass polynomial of $P_{n_0}(x_{n_0},y_{n_0})$.  
One may be able to use the algorithm in this paper to compute the Weierstrass polynomial of the original function $P(x,y)$. 
\end{remark}

Based on Lemma \ref{infty}, in the $n_0$-stage, we refine the change of variables as follows
\begin{align}\label{change}
x_{n_0} = x_{n_0+1}\q\mbox {and} \q y_{n_0}-f(x_{n_0}) = y_{n_0+1}x_{n_0}^{m_{n_0}}.
\end{align}
This will help us to eliminate all the nonzero roots of the new function $P_{n_0+1}$,
and thus eliminate the `bad' regions. 
Indeed, by (\ref{**}) one has
\begin{align}\label{***}
P_{n_0+1}(x_{n_0+1},y_{n_0+1}):&=P_{n_0}(x_{n_0+1}, y_{n_0+1}x_{n_0+1}^{m_{n_0}}+f(x_{n_0+1})) 
\\
&=x_{n_0+1}^{p_{n_0}+s_{n_0}m_{n_0}} y_{n_0+1}^{s_{n_0}} Q_{n_0+1}(x_{n_0+1},y_{n_0+1})
\end{align}
where 
$$
Q_{n_0+1}(x_{n_0+1},y_{n_0+1})= Q_{n_0}(x_{n_0+1},y_{n_0+1}x_{n_0+1}^{m_{n_0}}+f(x_{n_0+1})),
$$
is non-vanishing near the origin.  
One can see that $\mathcal N(P_{n_0+1})$ has only one vertex and $P_{n_0+1}(x_{n_0+1},y_{n_0+1})$ behaves like a monomial.
 Set 
$$U_{n_0+1} = \{(x_{n_0+1}, y_{n_0+1}): (x_{n_0},y_{n_0})\in U_{n_0,b} \},$$ 
then $U_{n_0+1,g} =U_{n_0+1}$, $U_{n_0+1,b} =\emptyset$ and the procedure ends here. 
Thus if we take $N_P$ to be the maximum of all such $n_{0}+1$ from all branches of iterations, 
then the resolution algorithm ends at the $N_P$-stage. 
Set
%$\mathcal G_n = \cup_{P_n} \mathcal G_n(P_n)$, 
\begin{align}\label{goodn}
&\mathcal G_n =\cup_{\bsb\alpha}\cup_j \{U_{n,g,\bsb\alpha,j}\},
\\
&\mathcal V_n = \cup_{\bsb\alpha} \mathcal V(P_{n,\bsb\alpha}),
\\
&\mathcal E_n = \cup_{\bsb\alpha} \mathcal E(P_{n,\bsb\alpha})
\end{align}
which represent the `good' regions, vertices and compact edges in the $n$-th stage respectively. The followings represent all the `good' regions, vertices and compact edges in all stages:
%where the union is taken over all possible $P_n$ in the $n$-th stage 
\begin{align}\label{goodall}
&\mathcal G =\cup_{0\leq n\leq N_P}\mathcal G_{n}=\cup_{ n}\cup_{\bsb\alpha}\cup_j \{U_{n,g,\bsb\alpha,j}\},
\\
&\mathcal V = \cup_{0\leq n\leq N_P}\mathcal V_n = \cup_{n}\cup_{\bsb\alpha} \mathcal V(P_{n,\bsb\alpha}),
\\
&\mathcal E =\cup_{0\leq n\leq N_P} E_n= \cup_{ n }\cup_{\bsb\alpha} \mathcal E(P_{n,\bsb\alpha}).
\end{align}
%It is clear now the `bad' regions defined by an edge $E\in\mathcal E(P)$, are essen
For $n\geq 1$, we often refer to $U_{n,g,\bsb\alpha,j}$ as a `good' region
defined by an edge of $E\in\mathcal E(P)$ in higher stage of iteration. 
Thus each `good' region is either defined by a vertex $V\in\mathcal V(P)$ or by an edge $E\in\mathcal {E}(P)$, or 
simply say defined by a face $F\in\mathcal F(P)$.

We can now determine $\epsilon_0$, the size of the original region $U_0$. 
One needs to go backward, starting from every leaf of the tree, i.e. every `good' region to the top. 
Each such path (from a chosen leaf to the top of the tree) will give rise to an upper bound of $\epsilon_0$ and
 one can just take $\epsilon_0$ to be any number smaller than all those bounds.  
We briefly describe how to do it. 
Notice first that all the roots $r_E$'s from different edges of different stages of iterations are independent from
sizes of the regions $U_n$'s, so are all assigned constants $c_E$'s and $C_E$'s defined in (\ref{chosen35}). For convenience, 
all the $\{\epsilon_n\}$ and $\{\epsilon_n'\}$ chosen below are less than $2^{-10} c_E$ for all $E$.    
\begin{enumerate}
\item Take one unchosen leaf from the tree. We use $[U_{n_0+1}, P_{n_0+1}]$ in the above case as our example here. 
Recall that $\epsilon_{n_0+1}$ denotes the size of $U_{n_0+1}$. 
\item There is an of upper bound $\epsilon_{n_0+1}'>0$, such that for every $\epsilon_{n_0+1}<\epsilon_{n_0+1}'$,   
the function $P_{n_0+1}$ behaves like a monomial in $U_{n_0+1}$. One can adjust the value of $\epsilon_{n_0+1}'$ 
to control the relative smallness of the error term. 
Choose one such $\epsilon_{n_0+1}$ and move up to $[U_{n_0}, P_{n_0}]$. 
\item Choose $0<\rho_{n_0}<\epsilon_{n_0+1}$ 
 and set  
$$
\xi_{n_0} =\inf\{|P_E(r)|: r\in I_{g}(E)\}
$$
where the infimum runs over all compact edges of $\mathcal N(P_{n_0})$ 
and $ I_{g}(E)$ is given in (\ref{domain}) which relies on $\rho_{n_0}$. % which also depends on $\rho_{n_0}$.
Then there is a $0<\epsilon_{n_0}'<\epsilon_{n_0+1}$ depending on $\rho_{n_0}, \xi_{n_0}$ and $P_{n_0}$,
such that for every $\epsilon_{n_0}<\epsilon_{n_0}'$, 
the function $P_{n_0}$ behaves like a monomial in all the `good' regions of $U_{n_0}$. 
Choose one such $\epsilon_{n_0}$ and move up to $[U_{n_0-1}, P_{n_0-1}]$.
\item Iterate (3) until we reach the top $[U_0,P_0]$, which will gives rise to one $\epsilon_0'$. 
\item Iterate (1), (2), (3) and (4) until there is no unchosen leaf. We then obtain finitely many $\epsilon_0'$'s.
Choose one $\epsilon=\epsilon_0$ less than any of them.
\end{enumerate}

\begin{theorem}\label{rs}
For each analytic function $P$ defined in a neighborhood of $0\in\mathbb R^2$, 
there is a standard region $U$,  % be a standard region whose size $\epsilon$ is sufficiently small. 
which can be partitioned into a finite collection of `good' regions 
$
\{\rho_{n}^{-1}(U_{n,g})\cap U,
U_{n,g}\in\mathcal G\}
$
where 
$$
\rho_n^{-1}(x_n,y_n)=(x,y) = (x_0,y_0)
$$
%is defined inductively by 
%\begin{align}\label{cc31}
%\begin{cases}
% x_k=x_{k+1} \\ 
 %y_k= (r_{k}+y_{k+1})x_{k+1}^{m_{k}}
 %\end{cases}
%\end{align}
%for $k=0,\dots, n-1$ via the chain
is given by 
\begin{align}\label{change08}
\begin{cases}
 x_n=x_{0}=x \\ 
 y_n= \gamma_n(x) +y_n   x^{m_0+\dots+m_{n-1}} %(r_{k}+y_{k+1})x_{k+1}^{m_{k}}
 \end{cases}
\end{align}
where $\gamma_n(x)$ is defined as follows. If $[U_n,P_n]$ lies in an infinite chain 
\begin{align}\label{subchain02}
[U_0,P_0]\to[U_1,P_1]\to \dots \to [U_n,P_n]\to [U_{n+1},P_{n+1}]\to\dots
\end{align}
 and 
 $n-1=n_0$ for some $n_0$ defined as in (\ref{*}), then 
 \begin{align}\label{gamma1x}
 \gamma_n(x) = \sum\limits_{k=0}^{\infty} r_kx_{0}^{m_{0}+m_{1}+\cdots+m_k},
 \end{align}
which converges for $x$ with $(x,y)\in U$; otherwise 
 \begin{align}\label{gamma2x}
 \gamma_n(x) = \sum\limits_{k=0}^{n-1} r_kx_{0}^{m_{0}+m_{1}+\cdots+m_k}.
 \end{align}
%\begin{align}\label{subchain0}
%[U_0,P_0]\to[U_1,P_1]\to \dots \to [U_n,P_n].
%\end{align}
%However, if (\ref{subchain0}) is a subchain of an infinite chain:
%\begin{align}\label{subchain02}
%[U_0,P_0]\to[U_1,P_1]\to \dots \to [U_n,P_n]\to [U_{n+1},P_{n+1}]\to\dots
%\end{align}
% and 
 %$n-1=n_0$ for some $n_0$ defined as in (\ref{*}), then the last step of the change of variables is redefined by
 %\begin{align}\label{change09}
 %\begin{cases}
%x_{n_0} = x_{n_0+1}
%\\
 %y_{n_0}-f(x_{n_0}) = y_{n_0+1}x_{n_0}^{m_{n_0}}
%\end{cases}
%\end{align}
%where
%\begin{align}\label{ddd00}
%f(x_{n_0})= \sum\limits_{k=n_0}^{\infty} r_kx_{n_0}^{m_{n_0}+m_{n_0+1}+\cdots+m_k}.
%\end{align}
 %via the chain.
%${\rm Hght}^*(P_k)$
We assume also that the order of $r_k$ is $s_k$. 
 For $0\leq k \leq n$, let $(p_{k,l},q_{k,l})$ be the leftmost vertex of $\mathcal N(P_k)$. 
For $0\leq k\leq n-1$, 
  let $(p_{k},q_{k})$ be the left vertex of the edge that defines $U_{k,b}\subset U_k$,
and $(p_{n},q_{n})$ be defined as follows: 
if $U_{n,g}$ is defined by a vertex $V$, then $(p_{n},q_{n})=V$; 
else $(p_{n},q_{n})$ is the left vertex of the edge that defines $U_{n,g}$. 
One has 
\begin{align}\label{ind1}
q_n\leq q_{n,l} \leq s_{n-1}\leq q_{n-1}\leq q_{n-1,l} \dots \leq q_1\leq s_0\leq q_0.
\end{align}
For every supporting line $L_{m_n}\in \mathcal {SL}(P_n)$ through $(p_n,q_n)$, one also has 
\begin{align}\label{ind2}
p_n +m_nq_n \leq p_0+m_0q_0+m_1q_{1,l}\dots+m_{n}q_{n, l }  \leq p_0+q_0m_0+s_0(m_1+\dots +m_{n}).
\end{align}
In addition, for any given $L\in \mathbb N$, for all $0\leq \alpha, \beta\leq L$ and $(x,y)=\rho_n^{-1}(x_n,y_n)\in \rho_{n}^{-1}(U_{n,g})\cap U$ one has
\begin{align}
|P(x,y)| = |P_n(x_n,y_n)| \sim |x_n^{p_{n}}y_n^{q_{n}}| \label{key1}
\\
|\partial_{x_n}^{\alpha} \partial_{y_n}^{\beta}P_n(x_n,y_n)| \lesssim  \min\{1, | x_n^{p_{n}-\alpha}y_n^{q_{n}-\beta} |\}\label{key2}
\end{align}
and
\begin{align}\label{new}
%|\partial_{x}^{\alpha} \partial_{y}^{\beta}P(x,y)|  \lesssim \min\{1, |x^{p_{n}-\alpha - \beta(m_0+\dots +m_{n-1})}y_n^{q_{n}-\beta}|\} .
| \partial_{y}^{\beta}P(x,y)|  \lesssim \min\{1, |x^{p_{n} - \beta(m_0+\dots +m_{n-1})}y_n^{q_{n}-\beta}|\} .
\end{align}
 
\end{theorem}
\begin{remark}

\begin{enumerate}[(i)]
\item 
This theorem also works for $U= (-\epsilon,\epsilon)\times (-\epsilon,\epsilon)$ (with a smaller $\epsilon$) and thus any open subset of $U$. 
First of all, we can apply this theorem to a standard region in the left half plane, i.e. to $U=(-\epsilon, 0)\times (-\epsilon, \epsilon)$ by setting $x=-x$. 
Secondly, one can incorporate the $y$-axis into the `good' regions defined by the leftmost vertex of $\mathcal N(P)$; see
￼￼(\ref{corner001}) for the case of incorporating the $x$-axis. % into the `good' regions defined by the rightmost vertex. 
%Therefore, this theorem applies to $U= (-\epsilon,\epsilon)\times (-\epsilon,\epsilon)$, as well as any open subset of $U$. 
\item
In application, the estimate (\ref{key1}) is often used in conjunction with (\ref{ind1}) and (\ref{ind2})
 to relate the lower bound of $|P(x,y)|$ in `good' regions from higher stages of iterations to the original Newton polyhedron of $P$.
 The estimates in (\ref{ind1}) and (\ref{ind2}) contain more details than are needed for the proof of Theorem \ref{main}. 
 For example the middle term of (\ref{ind2}) is not used explicitly in this proof. 
 We include such details for the purpose of future reference, 
 since we may need strong information from higher stages of iterations. 
\end{enumerate}
\end{remark}
\begin{proof}
The partition in the theorem is a consequence of the algorithm. 
The region $\rho_{n,g}^{-1} (U_{n,g})$ does not necessarily all lie in $U$ and  
hence we only need to take the portion in $U$, i.e. $\rho_{n,g}^{-1} (U_{n,g})\cap U$.
The change of variables (\ref{change08}) is obtained by iterating 
\begin{align}\label{cc31}
\begin{cases}
 x_k=x_{k+1} \\ 
 y_k= (r_{k}+y_{k+1})x_{k+1}^{m_{k}}.
 \end{cases}
\end{align}
The estimate (\ref{ind1}) follows directly from (\ref{induction}) and (\ref{ind2}) is a consequence of iterating the following estimates  
$$
p_n+m_nq_n \leq p_{n,l}+m_nq_{n,l} =   \left (p_{n-1}+m_{n-1}q_{n-1} \right)+m_nq_{n,l},
$$
where first estimate comes from the fact that $(p_{n,l}, q_{n,l})$ lies on or above the supporting line $L_{m_n}$ 
and the identity is due to (\ref{induction}). 

In addition, (\ref{key1}) and(\ref{key2}) come from Lemma \ref{edge} and Lemma \ref{vertex}. 
Finally, (\ref{new}) is an outcome of (\ref{change08}) and the chain rule, since
$$
\partial y /\partial y_n =x^{m_0+\cdots +m_{n-1}}. 
$$
\end{proof}

\subsection{\large A smooth partition}\q\q

In the above theorem, $U$ is decomposed into disjoint `good' regions $\rho_n^{-1}(U_{n,g})$'s.
 However, when it comes to application (in analysis), an overlap version is often more suitable 
since it provides extra room to fit in a smooth partition. 
Furthermore, there is an extra benefit in our problem:
 this extra room can help us to overcome the
convexity assumption in Theorem \ref{local}. 
To be precise, the `good' regions are in general not convex, even if
they are localized to pieces with the $x$-supports being dyadic intervals.
If we take the convex hull of one such piece, as
needed in Theorem \ref{local}, the convex hull may lie in more than one `good' regions. 
If so, in such convex hull we may fail to find a uniform monomial to bound $|P(x,y)|$ from below. 
Fortunately, the algorithm here provides sufficient flexibility to adjust the boundary of `good' regions, which
allows us to cut each of them into a number of pieces so that each piece
\begin{enumerate} 
\item is small enough with the convex hull lying essentially in the same `good' region; and 
\item is large enough with the total number of pieces bounded above uniformly by a harmless constant; See Lemma \ref{c01}
and Lemma \ref{c02}. 
\end{enumerate}

For these reasons, we slightly enlarge $U_{n,g}$ to $U_{n,g}\subset U^*_{n,g}\subset U^{**}_{n,g}$ and 
$U_{n,b}$ to $U_{n,b}\subset U^*_{n,b} \subset  U^{**}_{n,b}$. 
The $U^*_{n,g}$'s are used to fit in a smooth partition and 
$ U^{**}_{n,g}$'s are used to address the convexity assumption.  
The first step is to enlarge $I(E)$, $I_g(E)$ and $I_b(E)$ as follows 
\begin{align}
I^*(E) = [\frac{1}{2}c_E, 2C_E],\q\q I^{**}(E) = [\frac{1}{4}c_E, 4C_E],
\end{align}
\begin{align}
I^*_g(E) =I^*(E) \setminus \left( \cup_{1\leq j\leq J_E}I^{\frac{\rho_0}{2}}_{j}(E)				\right),
\q
I^{**}_g(E) =I^{**}(E) \setminus \left( \cup_{1\leq j\leq J_E}I^{\frac{\rho_0}{4}}_{j}(E)				\right)
\end{align}
and
\begin{align}
I^*_b(E) = \cup_{1\leq j\leq J_E}I^{{2\rho_0}}_{j}(E)	,	
\q\q
I^{**}_b(E) = \cup_{1\leq j\leq J_E}I^{{4\rho_0}}_{j}(E).		
\end{align}
Notice that our choices of $\rho_0$ in (\ref{rho00}) ensures that for all $E$,
$\{I^{{4\rho_0}}_{j}(E)\}$ do not overlap. 
Then we can defined the enlarged `good' regions as 
\begin{align}\label{good1*}
U^*_{0,g}(E) = \{(x,y)\in U_0: y=rx^{m_E},\,\, r\in I_g^*(E) \}\, 
\end{align}
and
\begin{align}
U^{**}_{0,g}(E) = \{(x,y)\in U_0: y=rx^{m_E},\,\, r\in I_g^{**}(E) \}.
\label{good1**}
\end{align}
Both $U^*_{0,g}(E) $ and $U^{**}_{0,g}(E)$ consist of $(J_E+2)$ 
curved triangular regions $\{U^{*}_{0,g}(E,j)\}$ and $\{U^{**}_{0,g}(E,j)\}$ respectively.
 In addition, one has
\begin{align}
U^{}_{0,g}(E,j) \subset U^{*}_{0,g}(E,j)\subset U^{**}_{0,g}(E,j).
\end{align}
The `good' regions defined by a vertex are enlarged to: 
\begin{align}\label{good*}
U^*_{0,g}(V) =\{(x,y)\in U_0: \frac{C_{E_r}}{2}x^{m_{E_r}}< y<  2c_{E_l} x^{m_{E_l}}\},
\\
\label{good**}
U^{**}_{0,g}(V) =\{(x,y)\in U_0: \frac{C_{E_r}}{4}x^{m_{E_r}}< y<  4c_{E_l} x^{m_{E_l}}\},
\end{align}
and finally the enlarged `bad' regions are
\begin{align}
\label{badregion*}
U^*_{0,b}(E,j) = \{(x,y) \in U_0: (r_j-2\rho_0) x^{m_E} <y < (r_j+2\rho_0) x^{m_E}\},
\\
\label{bad**}
U^{**}_{0,b}(E,j) = \{(x,y) \in U_0: (r_j-4\rho_0) x^{m_E} <y < (r_j+4\rho_0) x^{m_E}\}.
\end{align}
For $n\geq 1$, $U^*_{n,g}$'s, $U^*_{n,b}$'s and $U^{**}_{n,g}$'s, $U^{**}_{n,b}$'s are defined similarly.
 Since $\rho_k$ can be chosen sufficiently small, the above definitions do not cause any conflict. 
We have: 
\begin{corollary}\label{rs*} 
Let $U$ and $P$ as in Theorem \ref{rs}. 
Then all the estimates in Theorem \ref{rs} hold with $U_{n,g}$ replaced by $U_{n,g}^{**}$.
In addition, $U$ is contained in the union of all $U_{n,g}^{**}$. 
\end{corollary}

For a smooth partition, we issue some technical problems first. 
Let $c$ be a positive constant far away from all the roots of $P_E(x,y)$ for all $E$, say $c>2^{10}C_E$ for $C_E$
 defined in (\ref{chosen35}) and for all $E\in \mathcal E(P)$. 
 Choose a constant $\epsilon=2^{-5}c$.
Then $y=\pm cx$ divides the plane into four regions: $R_1$, $R_2$, $R_3$ and $R_4$, 
which represents the East, North, West and South regions respectively.
 Let $\{\Psi_j\}_{1\leq j\leq 4}$ be smooth functions such that 
\begin{align}
1 =\sum_{j=1}^4\Psi_j(x,y),\q\q (x,y)\neq (0,0).
\end{align}
Here $\Psi_1(x,y)$ is supported in 
$$
R_1^{\epsilon} = \{(x,y): x>0, \,-(c+\epsilon)x<y<(c+\epsilon )x\}
$$ 
and $\Psi_1(x,y)=1$ if 
\begin{align}
-(c-\epsilon)x<y<(c-\epsilon )x.
\end{align}
In addition, $\Psi_1$ satisfies 
\begin{align}
|\partial_y^\beta \Psi_1(x,y)| \leq C_{\beta} |y|^{-\beta},\q  \mbox{for any} \q\beta\in \mathbb N.
\end{align} 
The other functions $\Psi_2$,$\Psi_3$ and $\Psi_4$ are defined similarly in the other regions. 

For a given analytic function $P(x,y)$,
let $W$ be a small neighborhood of $0$ so that one can apply Theorem \ref{rs} and Corollary \ref{rs*} to $[W, P]$. 
Let $\Phi(x,y)$ be a smooth function adapted to $W$, in the sense $\supp \Phi \subset W$ and $\Phi(x,y) =1$ if $ (2x,2y)\in W$. Then
\begin{align}
\Phi(x,y)=\Phi(x,y)\sum_{j=1}^4 \Psi_j(x,y), \q\q \mbox{for} \q (x,y)\neq (0,0). 
\end{align}
We focus on $\Phi\Psi_1$, since discussions for the others are similar. Let
$$
U =W\cap R_1^\epsilon%\cap \{(x,y): x>0, \,-(c+\epsilon)x<y<(c+\epsilon )x\},
$$
then $\Phi\Psi_1$ is supported in $U$.

We now apply the resolution algorithm % the above resolution of singularities 
to $P(x,y)$ in the region $U$, which gives rise to a collection of `bad' regions $\{U^{*}_{n,b,\bsb\alpha,j}\}_{(n,\bsb\alpha,j)}$. 
For a fixed $U^*_{n,b,\bsb\alpha,j}$,  $\rho^{-1}_{n,\bsb\alpha}(U^*_{n,b,\bsb\alpha,j})$ is equal to
%\begin{align}
 %U\cap \{(x,y): r_0x^{m_0}+\dots+r_{n-1}x^{m_0+\dots+m_{n-1}}+(r_n-2\rho_n)x^{m_0+\dots+m_{n}}<y
%\\
%<r_0x^{m_0}+\dots+r_{n-1}x^{m_0+\dots+m_{n-1}}+(r_n+2\rho_n)x^{m_0+\dots+m_{n}}\}.
%\end{align}
\begin{align}
 U\cap \{\gamma_n(x)+(r_n-2\rho_n)x^{m_0+\dots+m_{n}}<y
< \gamma_n(x)+(r_n+2\rho_n)x^{m_0+\dots+m_{n}}\},
\end{align}
where $\gamma_n(x)$ is given by (\ref{gamma1x}) or (\ref{gamma2x}). 
The assumption that $U$ is contained in the east region $ R_1^\epsilon$ ensure that all the $m_0$ is at least 1, i.e.
the power of the leading term of $\gamma_n(x)$ is at least 1. 

We can then define a smooth function $\Phi_{n,b,\bsb\alpha,j}$ supported in $\rho^{-1}_{n,\bsb\alpha}(U^*_{n,b,\bsb\alpha,j})$ 
and $\Phi_{n,b,\bsb\alpha,j}(x,y)=1$ if $(x,y)\in \rho^{-1}_{n,\bsb\alpha}(U_{n,b,\bsb\alpha,j})$. 
In addition, the following is true%derivatives condition is true
\begin{align}\label{dede}
|\partial_y^{\beta}\Phi_{n,b,\bsb\alpha,j}(x,y)|\leq C_{\beta}|x|^{-\beta (m_0+\dots+m_n)}
\q\forall \beta\in \mathbb N.
\end{align}
Then 
\begin{align}
\Phi(x,y)\Psi_1(x,y)\bigg(1-\sum_{\bsb\alpha}\sum_j \Phi_{0,b,\bsb\alpha,j}(x,y)\bigg)
\end{align}
can be written as 
\begin{align}
\sum_{\bsb\alpha}\sum_j\Phi_{0,g,\bsb\alpha,j}(x,y)
\end{align}
where each $ \Phi_{0,g,\bsb\alpha,j}(x,y)$ is supported in the `good' region $U^*_{0,g,\bsb\alpha,j}$. Similarly, 
\begin{align}
\Phi(x,y)\Psi_1(x,y)\bigg(\sum_{\bsb\alpha}\sum_j \Phi_{0,b,\bsb\alpha,j}(x,y)\bigg)\bigg(1-\sum_{\bsb\alpha}\sum_j \Phi_{1,b,\bsb\alpha,j}(x,y)\bigg)
\end{align}
can be written as 
\begin{align}
\sum_{\bsb\alpha}\sum_j\Phi_{1,g,\bsb\alpha,j}(x,y)
\end{align}
where $\Phi_{1,g,\bsb\alpha,j}$ is supported in $\rho_{1,\bsb\alpha}^{-1}(U^*_{1,g,\bsb\alpha,j})$. 
Then we iterate the above procedures until the end of the algorithm. Combining (\ref{dede}), we obtain a smooth partition version of Theorem \ref{rs}.

\begin{theorem}\label{rs_smooth}
Let $\Phi$, $\Psi_1$ and $P$ be as above. Then
\begin{align}
\Phi(x,y)\Psi_1(x,y) =\sum_n \sum_{\bsb\alpha}\sum_j \Phi_{n,g,\bsb\alpha,j}(x,y), 
\end{align}
where $ \Phi_{n,g,\bsb\alpha,j}(x,y)$ is a smooth function 
supported in $\rho_{n,\bsb\alpha}^{-1}(U^*_{n,g,\bsb\alpha,j})$ and 
where $\{U_{n,g,\bsb\alpha,j}\}$ is the collection of `good' regions as in Theorem \ref{rs}.
 The behaviors of $P(x,y)$ in `good' regions 
$\rho_{n,\bsb\alpha}^{-1}(U^{*}_{n,g,\bsb\alpha,j})\cap U$
 and $\rho_{n,\bsb\alpha}^{-1}(U^{**}_{n,g,\bsb\alpha,j})\cap U$ 
are the same as Corollary \ref{rs*}. 
Moreover, $\Phi_{n,g,\bsb\alpha,j}(x,y)$ satisfies the following derivative conditions:

$(1)$ If $U_{n,g,\bsb\alpha,j}$ is defined by an edge, 
then $\rho^{-1}_{n,\bsb\alpha}(U^*_{n,g,\bsb\alpha,j})$ is given by 
in a curved triangular region of the form 
%\begin{align}
%b_{n,g,\bsb\alpha,j} x^{m_0+\dots+m_n} \leq  y-(r_0x^{m_0}+\dots+r_{n-1}x^{m_0+\dots+m_{n-1}}) \leq B_{n,g,\bsb\alpha,j} x^{m_0+\dots+m_n}
%\end{align} 
\begin{align}
b_{n,g,\bsb\alpha,j} x^{m_0+\dots+m_n} \leq  y-\gamma_n(x) \leq B_{n,g,\bsb\alpha,j} x^{m_0+\dots+m_n}
\end{align} 
for some nonzero constants $b_{n,g,\bsb\alpha,j}$ and $B_{n,g,\bsb\alpha,j}$ with the same signs. 
Then 
\begin{align}\label{de01}
|\partial_y^{\beta}\Phi_{n,g,\bsb\alpha,j}(x,y)|\leq C_{\beta}|x|^{-\beta (m_0+\dots+m_n)},
\q\forall \beta\in \mathbb N.
\end{align} 
 $(2)$Otherwise, $U_{n,g,\bsb\alpha,j}$ is defined by a vertex.  The region $U_{n,g,\bsb\alpha,j}$ may lie only in the first quadrant, only in the fourth quadrant or in both of them. 
 In the first case,  
  $\rho^{-1}_{n,\bsb\alpha}(U^*_{n,g,\bsb\alpha,j})$ is given by
%\begin{align}
%b_{n,g,\bsb\alpha,j} x^{m_0+\dots+m_{n-1}+m_{n,r}} <  &y-(r_0x^{m_0}+\dots+r_{n-1}x^{m_0+\dots+m_{n-1}})
%\\&
%<B_{n,g,\bsb\alpha,j}  x^{m_0+\dots+m_{n-1}+m_{n,l}},
%\end{align} 
\begin{align}
b_{n,g,\bsb\alpha,j} x^{m_0+\dots+m_{n-1}+m_{n,r}} <  &y-\gamma_n(x)
<B_{n,g,\bsb\alpha,j}  x^{m_0+\dots+m_{n-1}+m_{n,l}},
\end{align} 
where $0\leq m_{n,l} < m_{n,r}< \infty$. 
In the upper portion of $ \rho^{-1}_{n,\bsb\alpha}(U^*_{n,g,\bsb\alpha,j})\backslash \rho^{-1}_{n,\bsb\alpha}(U_{n,g,\bsb\alpha,j})$, one has
\begin{align}
|\partial_y^{\beta}\Phi_{n,g,\bsb\alpha,j}(x,y)|\leq C_{\beta}|x|^{-\beta (m_0+\dots+m_{n-1}+m_{n,l})},
\q\forall \beta\in \mathbb N
\end{align}
and in the lower portion
\begin{align}
|\partial_y^{\beta}\Phi_{n,g,\bsb\alpha,j}(x,y)|\leq C_{\beta}|x|^{-\beta (m_0+\dots+m_{n-1}+m_{n,r})}
\q\forall \beta\in \mathbb N.
\end{align}
The second case is similar. In the last case, 
 $U^*_{n,g,\bsb\alpha,j}$ is defined by the rightmost vertex of $\mathcal N(P_{n,\bsb\alpha})$ and $\rho^{-1}_{n,\bsb\alpha}(U^*_{n,g,\bsb\alpha,j})$ is given by
%\begin{align}
%|y-(r_0x^{m_0}+\dots+r_{n-1}x^{m_0+\dots+m_{n-1}})|\lesssim x^{m_0+\dots+m_{n-1}+m_{n,l}},
%\end{align} 
\begin{align}
|y-\gamma_n(x)|\lesssim x^{m_0+\dots+m_{n-1}+m_{n,l}},
\end{align} 
one has
\begin{align}
|\partial_y^{\beta}\Phi_{n,g,\bsb\alpha,j}(x,y)|\leq C_{\beta}|x|^{-\beta (m_0+\dots+m_{n-1}+m_{n,l})}
\q\forall \beta\in \mathbb N.
\end{align}
\end{theorem}

%--------------------------------------------------------------------------------------------------------------------------------------------------------------------------------
%
%
%
%
%
%
%
%
%
%
%
%
%
%
%
%
%
%
%
%\include{proof} 
%
%
%
%
%
%
%
%
%
%
%
%
%
%
%
%
%
%
%
%--------------------------------------------------------------------------------------------------------------------------------------------------------------------------------

\section{Proof of Theorem \ref{main}}\label{pmain}

Set $P=\partial_x\partial_y(\partial_x-\partial_y)S$ 
and assume $a(x,y)$ is supported in $W/100$, where $W$ is a neighborhood of $0$ such that we can apply
Theorem \ref{rs} and Corollary \ref{rs*} to $[W, P]$. 
 Same as what was done in Theorem \ref{rs_smooth},
 we divide $W$ into 4 regions by lines $y=cx$ and $y=-cx$ and  
restrict our discussion in the east region
\begin{align}\label{AUK}
U=W\cap  \{(x,y): x>0, \,-(c+\epsilon)x<y<(c+\epsilon )x\},
\end{align}
since the others can be reduced to $U$ by either changing $x$ to $-x$ or permuting $x$ and $y$, or both. 
Let $\Psi_j(x,y)$ and $\Phi(x,y)$ be smooth functions as in Theorem \ref{rs_smooth}, then 
\begin{align}
a(x,y) =a(x,y)\Phi(x,y) \sum_{j=1}^4\Psi_j(x,y),\q{\rm for}\q (x,y)\neq (0,0)
\end{align} 
and 
\begin{align}
\Lambda_S(f_1,f_2,f_3) =\sum_{j=1}^{4}\Lambda_S^j(f_1,f_2,f_3)
\end{align}
where
\begin{align}
\Lambda_S^j(f_1,f_2,f_3) =\iint e^{i\lambda S(x,y)}a(x,y)\Phi(x,y)\Psi_j(x,y)f_1(x)f_2(y)f_3(x+y)dxdy.
\end{align}
We only focus on $j=1$. Theorem \ref{rs_smooth} yields: 
$$
a(x,y)\Phi(x,y)\Psi_1(x,y)=\sum\limits_{0\leq n\leq N_P} \sum\limits_{\bsb\alpha }\sum\limits_j a_{n,g,\bsb\alpha,j}(x,y)
$$
where 
\begin{align}\label{an}
a_{n,g,\bsb\alpha,j}(x,y) =  \Phi_{n,g,\bsb\alpha,j}(x,y)a(x,y).
\end{align}
Set 
\begin{align} \label{j-th}
\Lambda_{n,g,\bsb\alpha,j}(f_1,f_2,f_3) = \iint e^{i\lambda S(x,y)} f_1(x)f_2(y)f_3(x+y) a_{n,g,\bsb\alpha,j}(x,y)  dxdy,
\end{align}
then 
\begin{align}\label{sum00}
\Lambda^1_S(f_1,f_2,f_3)=\sum\limits_{0\leq n\leq N_P} \sum\limits_{\bsb\alpha }\sum\limits_j \Lambda_{n,g,\bsb\alpha,j}(f_1,f_2,f_3).
\end{align}
Since the summands in (\ref{sum00}) have only finitely many terms, it suffices to bound each of them separately.  
For simplicity, we use $\Lambda_{n,g}$ to represent $\Lambda_{n,g}=\Lambda_{n,g,\bsb\alpha,j}$ 
for some $\bsb\alpha$ and $j$.  Then Theorem \ref{main} follows from
\begin{theorem}\label{detailsthm009}
If $U_{n,g}$ is defined by $F\in\mathcal F(P)$, then 
\begin{align}
\|\Lambda_{0,g}\| \lesssim |\lambda|^{-\frac{1}{2(3+d_F)}}.
\end{align}
\end{theorem}
We split the proof into three cases (i) $n=0$ and $U_{0,g}$ 
is defined by an edge $E$, (ii) $n=0$ and $U_{0,g}$ is defined by a vertex $V$  and (iii) $n\geq 1$.
\begin{proposition}\label{p1}
If $U_{0,g}$ is defined by an edge $E$, then
\begin{align}
\|\Lambda_{0,g}\| \lesssim |\lambda|^{-\frac{1}{2(3+d_E)}}.
\end{align}
\end{proposition}
\begin{proposition}\label{p2}
If $U_{0,g}$ is defined by a vertex $V$, then 
\begin{align}
\|\Lambda_{0,g}\| \lesssim |\lambda|^{-\frac{1}{2(3+d_V)}}.
\end{align}
\end{proposition}
\begin{proposition}\label{p3}
If $n\geq 1$ and $U_{n,g}$ is defined by an edge $E$, then 
\begin{align}
\|\Lambda_{n,g}\| \lesssim |\lambda|^{-\frac{1}{2(3+d_E)}}.
\end{align}

\end{proposition}
Heuristically, 
one can consider Proposition \ref{p1} as an estimate of $\Lambda_S$ in one supporting line, i.e. in $E\in\mathcal {SL}(P)$, 
and Proposition \ref{p2} as the sum/integral of such estimates over all supporting lines through $V$. 
%The proof of Proposition \ref{p1} indeed contains almost all the essence in that of others. 
Although the function $P_n(x_n,y_n)$ becomes more singular as $n$ increases, the region $U_{n,g}$ becomes nicer, 
which allows one to exploit certain orthogonality.  
The rest of this section is devoted to the proofs of the above Prospositions. 
\subsection{Proof of Proposition  \ref{p1}}  \q
%\begin{proof} [Proof of Proposition \ref{p1}.] \q\q
\\
Let $0<\sigma<1$ be a dyadic number and $\phi_{\sigma}(x)$ 
be a smooth function supported in $\frac{1}{2}\sigma<x < 2\sigma$, such that 
\begin{align}
\sum_{0<\sigma<1}\phi_{\sigma}(x) =1\q\mbox{for} \q   0<x<{1/10}. 
\end{align} 
Set 
\begin{align} \label{jj-th}
\Lambda_{0,g,\sigma_1,\sigma_2}(f_1,f_2,f_3) = \iint e^{i\lambda S(x,y)} f_1(x)f_2(y)f_3(x+y) a_{0,g}(x,y)  \phi_{\sigma_1}(x)\phi_{\sigma_2}(y)  dxdy,
\end{align}
where $\sigma_1$ and $\sigma_2$ are positive dyadic numbers 
and $a_{0,g} =a_{0,g,\bsb\alpha,j}$ for some $(\bsb\alpha,j)$. Notice that 
 $$
 \supp (a_{0,g}) \subset U_{0,g}^*,%= \{(x,y)\in U: |y|\sim |x|^m\}%(r-2\epsilon)x^m<y<(r+2\epsilon)x^m\}
 $$
which is a `good' region defined by $E$. Notice that $m=\mathcal M(E)\geq 1$ due to (\ref{AUK}). Thus 
\begin{align}\label{com7788}
|y|\sim x^{m} \q \mbox {and} \q\sigma_2\sim\sigma_1^{m}\leq \sigma_1.
\end{align}
This yields, for fixed $\sigma_1$, there are only finitely many choices of $\sigma_2$. 
We lose no generality in assuming, for a given $\sigma_1$, that $\sigma_2$ is fixed.  
To apply Theorem \ref{local}, we need to verify its assumptions. 
Let $K\in\mathbb N$ and equally divide the interval $(\sigma_1/2,2\sigma_1)$ 
into $K$ subintervals $\{I_k\}_{1\leq k\leq K}$ and set 
$$
U_{0,g,k}^*=\{(x,y)\in U^*_{0,g}\,: \,x\in I_k  \}.
$$
\begin{lemma}\label{c01}
There is a constant $K\in\mathbb N$ which is dependent of $\sigma_1$ and $\sigma_2$, such that  
\begin{align}
{\rm Conv}(U_{0,g,k}^*) \subset U^{**}_{0,g},
\end{align}
for all $1\leq k\leq K$. 
\end{lemma}
%Lemma \ref{convex} yields $convex(U_{0,g,\bsb\alpha,k}^*) \subset U^{**}_{0,g,\bsb\alpha}$. 
The proof of this lemma is postponed to the end of this section. Now
let $(p_l,q_l)$ be the left vertex of $E$. 
Then for every $(x,y)\in {\rm Conv}(U_{0,g,k}^*)\subset U^{**}_{0,g}$,
 %Lemma \ref{c01},
  Theorem \ref{rs} and Corollary \ref{rs*} yield,
\begin{align}\label{lo1}
|P(x,y)| \gtrsim |x|^{p_l}|y|^{q_l}\sim \sigma_1^{p_l}\sigma_2^{q_l},
\end{align}
and for $\beta=0,1,2$ 
\begin{align}\label{lo2}
|\partial_y^{\beta}P(x,y)|  \lesssim |x|^{p_l}|y|^{q_l-\beta}\sim \sigma_1^{p_l}\sigma_2^{q_l-\beta}.
\end{align}
Theorem \ref{rs_smooth} together with (\ref{an}) yields
\begin{align}\label{lo3}
|\partial_y^{\beta} a_{0,g}(x,y)| \lesssim \sigma_2^{-\beta}.
\end{align}
By invoking Theorem \ref{local}, one has
\begin{align}
\|\Lambda_{0,g,\sigma_1,\sigma_2,k}\| \lesssim |\lambda\sigma_1^{p_l}\sigma_2^{q_l}|^{-1/6},%\sim |\lambda\sigma_1^{p_l+mq_l}|^{-1/6}
\end{align} 
where $\Lambda_{0,g,\sigma_1,\sigma_2,k}(f_1,f_2,f_3)$ is given by 
$$
\iint e^{i\lambda S(x,y)} f_1(x)\Id_{I_{k}}(x)f_2(y)f_3(x+y) a_{0,g}(x,y) 
\phi_{\sigma_1}(x)\phi_{\sigma_2}(y)dxdy.
$$
Summing over $1\leq k\leq K$ ($K$ is a constant) yields:
\begin{align}\label{9090}
\|\Lambda_{0,g,\sigma_1,\sigma_2}\| \lesssim |\lambda\sigma_1^{p_l}\sigma_2^{q_l}|^{-1/6}.
%\sim |\lambda\sigma_1^{p_l+mq_l}|^{-1/6}
\end{align} 
Employing Lemma \ref{schur}, (\ref{com7788}) and combining (\ref{9090}), one obtains:

\begin{align}\label{combine}
\|\Lambda_{0,g,\sigma_1,\sigma_2}\| \lesssim
	\left\{ \begin{array}{rcl}
      |\lambda \sigma_1^{p_l}\sigma_2^{q_l}|^{-1/6}\sim | \lambda\sigma_2^{q_l+{p_l}/{m}}|^{-1/6},
	\\
	\\  \min\{\sigma_1,\sigma_2\}^{1/2}  =\sigma_2^{1/2}.
	                \end{array}\right.
\end{align}
Since for fixed $\sigma_2$, $\sigma_1$ is fixed, summing over $\sigma_2$ yields
\begin{align}
\bigg\|\sum\limits_{\sigma_1,\sigma_2}\Lambda_{0,g,\sigma_1,\sigma_2}\bigg\| \lesssim |\lambda|^{-\frac{1}{2(3+q_l+p_l/m)}} = |\lambda|^{-\frac{1}{2(3+d_E)}},
\end{align}
as desired. 

%\end{proof}

\subsection{Proof of Proposition  \ref{p2}}  \q
\\
Like the proof of Proposition \ref{p1}, insert the smooth support $\phi_{\sigma_1}(x)\phi_{\sigma_2}(y)$ into $\Lambda_{0,g}$. 
Set $V=(p,q)$ and assume $-1/m_l$ and $-1/m_r$ be the slopes of the edges left and right to $V$. 
Due to the assumption on $U$ (\ref{AUK}):
we may replace $m_l$ with 1 if $m_l<1$. Thus\begin{align}
  \infty\geq m_{r}> m_{l}\geq 1.
\end{align}
Notice that 
$$
\sigma_2^{1/m_r} \gtrsim \sigma_1 \gtrsim \sigma_2^{1/m_{l}} \gtrsim\sigma_2.
$$
Consider all $(\sigma_1,\sigma_2)$ with $\sigma_2 \lesssim \lambda_2:= |\lambda|^{-\frac{1}{3+q+p/m_{l}}}=|\lambda|^{-\frac{1}{3+d_V}}$. 
By Lemma \ref{schur} and the triangle inequality, we have
\begin{align}\label{bbb}
\Bigg\| \sum_{\sigma_2\lesssim \lambda_2}\bigg(\sum_{\sigma_1}\Lambda_{0,g,\sigma_1,\sigma_2}\bigg) \Bigg\|
\lesssim
\sum_{\sigma_2\lesssim \lambda_2} |\sigma_2|^{1/2}
\lesssim |\lambda_2|^{1/2}=  |\lambda|^{-\frac{1}{2(3+d_V)}}.%\leq |\lambda|^{-\frac{1}{2(3+d)}}.
\end{align}
Now assume $\sigma_2\gtrsim \lambda_2$. 
As it was done in the proof of Proposition \ref{p1},
we invoke the same trick to meet the convexity assumption of Theorem \ref{local}: 
splitting $\Lambda_{0,g,\sigma_1,\sigma_2}$ into the sum of $\Lambda_{0,g,\sigma_1,\sigma_2,k}$, 
applying Theorem \ref{local} to each $\Lambda_{0,g,\sigma_1,\sigma_2,k}$ 
and summing them together. This gives 

\begin{align}
\|\Lambda_{0,g,\sigma_1,\sigma_2}\| \lesssim |\lambda \sigma_1^{p} \sigma_2^{q}|^{-1/6}.\label{bbound}
\end{align}
Since $\sigma_1 \gtrsim \sigma_2^{1/m_{l}}$, thus
\begin{align}
\sum_{\sigma_1}\|\Lambda_{0,g,\sigma_1,\sigma_2}\| \lesssim |\lambda \sigma_2^{p/m_l+q}|^{-1/6}.
\end{align}
Summing all $\sigma_2\gtrsim \lambda_2$ we obtain the same bound as (\ref{bbb}).

\subsection{Proof of Proposition  \ref{p3}}  \q
\\

Assume the change of variables $(x,y)=\rho_k^{-1}(x_k,y_k)$ is given as in Theorem \ref{rs}, 
as well as the parameters $r_k$, $s_k$ and $(p_k,q_k)$ for $0\leq k\leq n$. 
Theorem \ref{rs} and Corollary \ref{rs*}  yield
\begin{align}
|P(x,y)| =|P_n(x_n,y_n)|\sim  |x_n^{p_n}y_n^{q_n}|
\end{align}
for all $(x_n,y_n) \in U^{**}_{n,g}$. Since $U^{**}_{n,g}\subset U^{**}_{n}$ is a `good' region, 
we can find $m_n' $ and $m_n$ s.t.
$$
|x_n|^{m_n'}\lesssim |y_n| \lesssim |x_n|^{m_n},
$$
where $0\leq m_n\leq m_n'\leq \infty$. 
In addition, if $m_n=m_n'$, then $U^{**}_{n,g}$ is defined by an edge; otherwise by a vertex. 
Dyadically decompose $x_n\sim \sigma_1$ and let 
$\Lambda_{n,g,\sigma_1}$ denote the operator $\Lambda_{n,g}$ when $x_n$ 
is localized to this region. 

Due to the `almost orthogonality' below,
we only need to handle one single $\sigma_1$. % in the case $n\geq 1$ due to the `almost orthogonality' below.
\begin{claim}\label{ah99}
There is a constant $L$ s.t. if
\begin{align}\label{goal99}
\|\Lambda_{n,g,\sigma_1}\| \leq A
\end{align}
for some $A>0$ then 
\begin{align}\label{goal7799}
\sum_{\sigma_1}\|\Lambda_{n,g,\sigma_1}\| \leq LA
\end{align}
\end{claim}

\begin{proof} [Proof of Claim \ref{ah99}.] \q
Recall that 
\begin{align}\label{var99}
y=y(x) =\gamma_n(x)+y_nx^{m_0+\dots+m_{n-1}}
\end{align}
where
$$
 \gamma_n(x) = \sum\limits_{k=0}^{*} r_kx_{0}^{m_{0}+m_{1}+\cdots+m_k}.
$$
with $* = n-1$ or $\infty$; see (\ref{gamma1x}) and (\ref{gamma2x}). 
Set 
$$
Y(\sigma) = \{y(x): x\sim \sigma\}.
$$
For $\sigma$ small enough, one has $|y|\sim |r_0\sigma^{m_0}|$ for $y\in Y(\sigma) $. Thus, one can find $L\in\mathbb N$ such that 
\begin{align}\label{dj35}
Y(\sigma_1)\cap Y(\sigma_2) =\emptyset
\end{align}
given $\sigma_1 \geq 2^{L} \sigma_2$ and $\sigma_1$ is small.  

Consider the congruence classes modulo $L$ for all small $\sigma_1$. Let  
$$
H_\ell =\{\sigma_1=2^{-h}: h\equiv \ell \q mod\q L \}
$$
for $0\leq \ell <L$. Notice that 
$$
\Lambda_{n,g,\sigma_1}(f_1,f_2,f_3)= \Lambda_{n,g,\sigma_1}(f_1\Id_{I_{\sigma_1}},f_2\Id_{Y(\sigma_1)},f_3),
$$
where $I_{\sigma_1}$ denotes such $x$ with $x\sim \sigma_1$. 
The triangle inequality, the Cauchy-Schwartz inequality and (\ref{goal}) yield 
\begin{align*}
\Big|\sum\limits_{\sigma_1\in H_\ell}\Lambda_{n,g,\sigma_1}(f_1,f_2,f_3)\Big|
\leq 
&\sum\limits_{\sigma_1\in H_\ell}\Big|\Lambda_{n,g,\sigma_1}(f_1\Id_{I_{\sigma_1}},f_2\Id_{Y(\sigma_1)},f_3)\Big|
\\
\leq
&
A\sum\limits_{\sigma_1\in H_\ell} \|f_1\Id_{I_{\sigma_1}}\|_2 \|f_2\Id_{Y(\sigma_1)}\|_2 \|f_3\|_2
\\
\leq 
&
A\|\sum\limits_{\sigma_1\in H_\ell} f_1\Id_{I_{\sigma_1}}\|_2 \|\sum\limits_{\sigma_1\in H_\ell}f_2\Id_{Y(\sigma_1)}\|_2 \|f_3\|_2
\end{align*}
which is controlled by $A\|f_1\|_2\|f_2\|_2\|f_3\|_2$ due to (\ref{dj35}). The desired estimates then follow by the triangle inequality  
$$
\Big\|\sum\limits_{\sigma_1}\Lambda_{n,g,\sigma_1}\Big\| \leq L \sup_{0\leq \ell < L}\Big\|\sum\limits_{\sigma_1\in H_\ell}\Lambda_{n,g,\sigma_1}\Big\|\leq LA.
$$
\end{proof}

Dyadically decompose $(x_n,y_n)$ as
  \[ \left\{ \begin{array}{ll}
         |x_n| \sim \sigma_1
         \\
        |y_n| \sim \sigma_2 x_n^{m_n}\sim \sigma_2 \sigma_1^{m_n}.
        \end{array} \right. \] 
Use $U_{n,g}(\sigma_1,\sigma_2)$
and $\Lambda_{n,g,\sigma_1,\sigma_2}$ to denote 
 the corresponding localization of the `good' region $U_{n,g}$ and
the operator $\Lambda_{n,g}$ respectively. 
In what follows, we shall focus on one single $\sigma_1$ and
 prove a bound for $\|\Lambda_{n,g,\sigma_1}\|$ that is independent of $\sigma_1$.

Even in the region $U_{n,g}(\sigma_1,\sigma_2)$, there is some orthogonality that one can employ.  In fact,
 the change of variables (\ref{var99})
yields that the length of every $x$-section of $U_{n,g}(\sigma_1,\sigma_2)$ % (as well as $U^{**}_{n,g}$) 
denoted by
 $\triangle y$ is bounded by 
\begin{align}\label{dy}
\triangle y \lesssim \sigma_2\sigma_1^{m_0+\cdots+m_n}.
\end{align}
In addition, (\ref{var99}) gives   
$$
\triangle y \sim \triangle x \frac{dy}{dx}\sim \triangle x \cdot \sigma_1^{m_0-1}
$$
and thus
\begin{align}\label{dx}
\triangle x \sim \sigma_2 \sigma_1^{m_0+\cdots+m_n-(m_0-1)}.
\end{align}
Notice also $m_0\geq 1$ due to the assumption $|y|\lesssim |x|$ on $U$. 
We then divide the interval $(\sigma_1/2, 2\sigma_1)$ equally into $H$ subintervals $\{I_h\}_{1\leq h\leq H}$, where 
$$
H =K \cdot {\sigma_1}\Big(\sigma_2 \sigma_1^{m_0+\cdots+m_n-(m_0-1)}\Big)^{-1}.
$$
Here $K$ is a large constant, independent of $(\sigma_1,\sigma_2)$ and designed
 to treat the convexity assumption in Theorem \ref{local}. 
Set 
\begin{align}
U^*_{n,g,h} = \{(x,y)\in U^*_{n,g}:x \in I_h \}
\end{align}
and
$$
Y(I_h)=\{y(x):x\in I_h\,\, \mbox{and} \,\, y_n\sim \sigma_2\sigma_1^{m_n}\}
$$
where $y(x)$ is defined in (\ref{var99}). Then $\Lambda_{n,g,\sigma_1,\sigma_2}$ 
can be further decomposed into $\{\Lambda_{n,g,\sigma_1,\sigma_2, h}\}_{1\leq h\leq H}$ by restricting $x\in I_h$. By (\ref{var99}), 
$y=y(x)$ is monotone given $|x|$ sufficiently small. 
Hence, given $L=L(P,\epsilon, K)$ large enough, by (\ref{dy}) and (\ref{dx}) one has
\begin{align}\label{disjoint}
Y(I_h)\cap Y(I_{h'}) = \emptyset \q if \q |h-h'|\geq L,
\end{align}
which implies the following `almost orthogonality' principle, 
\begin{claim}\label{ah}
If there is a constant $A$ s.t. 
\begin{align}\label{goal}
\|\Lambda_{n,g,\sigma_1,\sigma_2, h}\| \leq A,\q\mbox{for all} \q 1\leq h\leq H, 
\end{align}
then  
\begin{align}\label{goal77}
\|\Lambda_{n,g,\sigma_1,\sigma_2}\| \leq LA.
\end{align}
\end{claim}
The proof of Claim \ref{ah} is almost identical to that of Claim \ref{ah99} and it is omitted. 
To prove (\ref{goal}) for 
 a desired bound $A$, 
we also need the following lemma which is similar to Lemma \ref{c01} and whose proof can be found 
at the end of this section. 
\begin{lemma}\label{c02}
There is a $K=K(P,\epsilon,n)$, such that for $1\leq h\leq H$, one has
$$
{\rm Conv} (\rho_n^{-1}(U_{n,g,h}^*))\subset \rho^{-1}_n(U_{n,g}^{**}).
$$
\end{lemma}
With all the preparation above, we now complete the proof of Proposition \ref{p3}. 
Invoking Lemma \ref{schur}, Theorem \ref{rs}, Corollary \ref{rs*} and Theorem \ref{local} yields
\begin{align}\label{aug}
\|\Lambda_{n,g,\sigma_1,\sigma_2,h}\| \lesssim
\left\{\begin{array}{rcl}
|\lambda \sigma_1^{p_n+q_nm_n} \sigma_2^{q_n}|^{-1/6}
\\
\\
(\sigma_2\sigma_1^{m_0+\cdots+m_n})^{1/2},
\end{array}\right.
\end{align}
for every $1\leq h\leq H$. By Claim \ref{ah}, summing over $h$ yields

\begin{align}\label{aug1}
\|\Lambda_{n,g,\sigma_1,\sigma_2}\| \lesssim
\left\{\begin{array}{rcl}
|\lambda \sigma_1^{p_n+q_nm_n} \sigma_2^{q_n}|^{-1/6}
\\
\\
(\sigma_2\sigma_1^{m_0+\cdots+m_n})^{1/2}.
\end{array}\right.
\end{align}
Balancing these two estimates via $\sigma_1$ yields 
\begin{align}\label{sum}
\|\Lambda_{n,g,\sigma_1,\sigma_2}\| \lesssim |\lambda|^{-\frac{1}{2(3+(p_n+q_nm_n)/(m_0+\dots+m_n))}}\cdot 
			\sigma_2^{\frac{p_n+m_nq_n-q_n(m_0+\cdots+m_n)}{3(m_0+\cdots+m_n)+p_n+m_nq_n}}.
\end{align}
By (\ref{ind1}) and (\ref{ind2}), we conclude that 
\begin{align}\label{qq78}
p_n+m_nq_n-q_n(m_0+\cdots+m_n)\geq 0.
\end{align}
In the case (\ref{qq78}) is strictly positive, 
we can sum over $\sigma_2$ in (\ref{sum}), which yields 
\begin{align}
\sum_{\sigma_2}\|\Lambda_{n,g,\sigma_1,\sigma_2,h}\| \lesssim |\lambda|^{-\frac{1}{2(3+(p_n+q_nm_n)/(m_0+\dots+m_n))}} \leq 
|\lambda|^{-\frac{1}{2(3+d_E)}}, 
\end{align}
where the latter inequality will be proved in a moment. \\

Otherwise 
\begin{align}\label{critical}
p_n+m_nq_n-q_n(m_0+\cdots+m_n)= 0,
\end{align}
which implies
\begin{align}\label{con78}
\begin{cases}
p_{0}=0
\\
q_{0}=s_0= q_{1}= q_{2}= \cdots = q_{n-1}= q_n.
\end{cases}
\end{align}
Then (\ref{aug1}) becomes:
\begin{align}\label{aug99999}
\|\Lambda_{n,g,\sigma_1,\sigma_2}\| \lesssim
\left\{\begin{array}{rcl}
|\lambda \sigma_1^{q_{0}(m_0+\cdots+m_n)} \sigma_2^{q_{0}}|^{-1/6}
\\
\\
(\sigma_2\sigma_1^{m_0+\cdots+m_n})^{1/2}.
\end{array}\right.
\end{align}
Summing over $\sigma_2$ yields
\begin{align*}
\left\|\Lambda_{n,g,\sigma_1}\right\|=\left\|\sum_{\sigma_2}\Lambda_{n,g,\sigma_1,\sigma_2}\right\| \leq \sum_{\sigma_2}\left\|\Lambda_{n,g,\sigma_1,\sigma_2}\right\| \lesssim |\lambda|^{-\frac{1}{2(3+q_{0})}}
= |\lambda|^{-\frac{1}{2(3+d_E)}}
\end{align*}
The last equality comes from $p_{0}=0$,  $m_0\geq 1$ and thus ${\rm mult}(P)=d_E=q_0$. 

It remains to verify  $(p_n+q_nm_n)/(m_0+\dots+m_n) \leq d_E$. Indeed, 
 (\ref{ind2}) yields
%notice that $V_k=(p_{k-1}+m_{k-1}s_{k-1},s_{k-1})$ as the far left vertex of the Newton Polyhedron in the $k$-stage, is above the line passing $(p_k,q_k)$ with slope $-\frac{1}{m_k}$. Thus 
%\begin{align}
%p_{k}+m_kq_{k}\leq  p_{k-1}+m_{k-1}s_{k-1}+m_ks_{k-1} \leq  p_{k-1}+m_{k-1}q_{k-1}+m_kq_{0},
 %\end{align}
%since $q_0 \geq q_{k-1}\geq s_{k-1}$ for all $ k\geq 1$. 
%Iterating the above formula yields 
%\begin{align}
%p_{n}+m_nq_{n}\leq  p_{0}+m_{0}q_{0}+m_1q_0 +\dots +m_nq_0 %\leq  p_{k-1}+m_{k-1}q_{k-1}+m_kq_{0}.
% \end{align}
%Therefore
\begin{align}
\frac{p_n+q_nm_n}{m_0+\cdots+m_n} \leq
\frac{(p_{0}+m_0{q_{0}})+s_{0}(m_1+\cdots+m_n)}{m_0+(m_1+\cdots+m_n)} 
\leq \frac{p_{0}+m_0{q_{0}}}{m_0} = d_E,
\end{align}
since 
%$m_0\geq 1$ implies
\begin{align}
 s_0 \leq  q_0  \leq \frac{p_{0}+m_0{q_{0}}}{m_0} \leq  d_E .
\end{align}

\subsection{Verification of Lemma \ref{c01} and Lemma \ref{c02}} 
\q\q

We only provide the proof of Lemma \ref{c02} for the other is similar and even simpler. 
First, notice that the upper and the lower boundaries of $U^*_{n,g,h}$ can be represented by two curves:
\begin{align*}
&\bar\gamma_1(x) =r_0x^{m_0}+r_1x^{m_0+m_1}+\cdots+r_{n-1}x^{m_0+\cdots+m_{n-1}}+r_{n,1}x^{m_0+\cdots+m_{n-1}+m_{n}}
\\
&\bar\gamma_2(x)= r_0x^{m_0}+r_1x^{m_0+m_1}+\cdots+r_{n-1}x^{m_0+\cdots+m_{n-1}}+r_{n,2}x^{m_0+\cdots+m_{n-1}+m_{n}'}
\end{align*}
and the upper and the lower boundaries of $U^{**}_{n,g,h}$ by
\begin{align*}
&\bar\gamma_1^*(x) =  r_0x^{m_0}+r_1x^{m_0+m_1}+\cdots+r_{n-1}x^{m_0+\cdots+m_{n-1}}+r^*_{n,1}x^{m_0+\cdots+m_{n-1}+m_{n}}
\\
&\bar\gamma_2^*(x)= r_0x^{m_0}+r_1x^{m_0+m_1}+\cdots+r_{n-1}x^{m_0+\cdots+m_{n-1}}+r^*_{n,2}x^{m_0+\cdots+m_{n-1}+m_{n}'}
\end{align*}
where $r_{n,1}<r^*_{n,1}$, $r_{n,2}>r^*_{n,2}$ and $0\leq m_n\leq m_n'$. If $m_n= m_n'$, we have in addition $r_{n,1}>r_{n,2}$. 
Notice that all the curves above have only 
finitely many terms of fractional powers of $x$, 
even in the case when $n=n_0-1$ for some $n_0$ defined as in (\ref{*}). 
Assume without loss of generality $r_0>0$
%$r_{n,1}, r_{n,2}, r_{n,1}^*, r_{n,2}^*>0$. 
and in particular all the curves above are increasing functions of $x$. 
The assumption $m_0\geq 1$ (see (\ref{AUK})) implies that the curves $\bar\gamma_1$, $\bar\gamma_2$, $\bar\gamma_1^*$ and $\bar\gamma_2^*$ 
are concave up. Thus we only need to take care of the upper boundary of $U^*_{n,g,h}$.
 Use $\sigma_{1,h}$ to denote the left end point of the interval $I_h$, for $1\leq h\leq H$.
 By the definition of convexity, one needs to verify that if $K$ is sufficiently large, then 
\begin{align}\label{tp}
t\bar\gamma_1( \sigma_{1,h})+(1-t)\bar\gamma_1(\sigma_{1,h+1}) <\bar\gamma^*_1(t\sigma_{1,h}+(1-t)\sigma_{1,h+1}) 
\end{align}
for all $0\leq t\leq 1$ and $\sigma_1>0$ sufficiently small. 

Since both $\bar\gamma_1$ and $\bar\gamma_1^*$ are increasing functions,
 it suffices to show 
\begin{align}\label{tp56}
\bar\gamma_1( \sigma_{1,h+1}) <\bar\gamma^*_1(\sigma_{1,h}) .
\end{align}
By the Mean Value Theorem, there is a $\sigma\in I_h$ s.t. 
\begin{align}
\bar\gamma_1( \sigma_{1,h+1})-\bar\gamma_1( \sigma_{1,h}) =\bar\gamma_1'(\sigma) (\sigma_{1,h+1}-\sigma_{1,h} )
=\frac{3}{2K} \sigma_2\sigma_1^{m_0+\cdots+m_n-(m_0-1)}\bar\gamma_1'(\sigma).
\end{align}
Since $\sigma_1/2\leq\sigma\leq 2\sigma_1$, then $|\bar\gamma'_1(\sigma)| \leq C\sigma_1^{m_0-1}$
for constant $C$. Thus
\begin{align}\label{bi1}
\bar\gamma_1( \sigma_{1,h+1})-\bar\gamma_1( \sigma_{1,h}) \leq \frac{C}{K}\sigma_2\sigma_1^{m_0+\cdots+m_n}<\frac{C}{K}\sigma_1^{m_0+\cdots+m_n},
\end{align}
since $\sigma_2<1$. On the other hand 
\begin{align}\label{bi2}
\bar\gamma_1^*(\sigma_{1,h}) -\bar\gamma_1(\sigma_{1,h}) = (r^*_{1,n}-r_{1,n})\sigma_{1,h}^{m_0+\dots+m_n}\geq (r^*_{1,n}-r_{1,n})(\sigma_1/2)^{m_0+\cdots+m_n}
\end{align}
Thus by choosing
$$
K> \frac{C}{r^*_{1,n}-r_{1,n}} \cdot  2^{m_0+\cdots+m_n},
$$
  (\ref{bi1}) and (\ref{bi2}) imply (\ref{tp56}) and thus (\ref{tp}). 
\newline

%------------------------------------------------------------------------------------------------------------------------------------------------------------------------------------------------------------------------
%
%
%
%
%  end 
%
%
%
%------------------------------------------------------------------------------------------------------------------------------------------------------------------------------------------------------------------------

\bibliographystyle{plain}

%\bibliographystyle{te}


\end{document}